\documentclass[a4paper,reqno,11pt]{amsart}
\usepackage[left=1 in, right=1 in,top=1 in, bottom=1 in]{geometry}
\usepackage{amsfonts}
\usepackage{amssymb} 
\usepackage{amsthm} 
\usepackage{amsmath} 
\usepackage{mathrsfs} 
\usepackage{color}
\usepackage{enumerate}
\usepackage{hyperref}
\usepackage[numbers,sort&compress]{natbib}
\usepackage{pdfsync} 
\usepackage{esint}
\usepackage{graphicx}
\usepackage{float}
\usepackage{epstopdf}
\usepackage{caption}
\usepackage{subfigure}
\usepackage{url}
\usepackage{dsfont}
\usepackage{enumitem}
\usepackage{bm}
\usepackage{setspace}
\usepackage{inputenc} 
\newtheorem{theorem}{Theorem}[section]
\newtheorem{definition}{Definition}[section]
\newtheorem{lemma}{Lemma}[section]
\newtheorem{remark}{Remark}[section]
\newtheorem{proposition}{Proposition}[section]

\numberwithin{equation}{section}

\renewcommand{\d}{\operatorname{d}}

\renewcommand{\u}{\mathbf{u}}
\newcommand{\w}{\mathbf{w}}

\newcommand{\ls}{\leqslant}

\newcommand{\ra}{\rightarrow}

\newcommand{\eps}{\varepsilon}
\renewcommand{\u}{{\bf u}}

\renewcommand{\v}{{\bf{v}}}
\newcommand{\z}{{\bf{z}}}

\newcommand{\tA}{\tilde{\mathcal{A}}}
\newcommand{\B}{\tilde{\mathcal{B}}}
\usepackage{stmaryrd}

\allowdisplaybreaks[4]

\begin{document}

\title[Stationary solution to SEP in Bounded Domain]{Stationary solution to Stochastically Forced Euler-Poisson Equations in Bounded Domain: \\
 Part 1. 3-D Insulating Boundary}

\author{Yachun Li$^{1}$, Ming Mei$^{2,3,4}$, Lizhen Zhang$^{5}$}
\dedicatory{ {\small\it $^1$School of Mathematical Sciences, CMA-Shanghai, MOE-LSC, SHL-MAC, \\
 Shanghai Jiao Tong University, 200240, China}\\
{\small\it $^2$School of Mathematics and Statistics, Jiangxi Normal University\\ Nanchang, 330022, China} \\
{\small\it $^3$Department of Mathematics, Champlain College Saint-Lambert\\
	     \small\it     Saint-Lambert, Quebec, J4P 3P2, Canada} \\
{\small\it $^4$Department of Mathematics and Statistics, McGill University \\
	     \small\it     Montreal, Quebec, H3A 2K6, Canada  }\\
{\small\it $^5$School of Mathematical Sciences, Shanghai Jiao Tong University\\
	     \small\it   Shanghai, 200240, China }\\
{\tt Emails: ycli@sjtu.edu.cn; ming.mei@mcgill.ca; Zhanglizhen@sjtu.edu.cn;}
}

\begin{abstract}

This paper is concerned with  $3$-D stochastic Euler-Poisson equations with insulating boundary conditions forced by the Wiener process.  We first establish the global existence and uniqueness of the solution to the system, then we prove that the solution converges to its steady-state time-asymptotically. To obtain the converging rate, we need to develop weighted energy estimates, which are not required for the deterministic counterpart of the problem. Moreover, we observe that the invariant measure is just the Dirac measure generated by the steady-state, in which the time-exponential convergence rate to the steady-state plays an essential role.


\end{abstract}

\date{\today}

\maketitle
 \setcounter{tocdepth}{2}

\section{Introduction}

Euler-Poisson equations is important in the analysis and design of semiconductor devices, offering a more precise description of physical phenomena \cite{Guo2006StabilityOS} compared to the conventional drift-diffusion model. Furthermore, in the extreme ultraviolet (EUV) lithography, stochastic effects sometimes cause unwanted defects and pattern roughness in chips \cite{deBisschop2018StochasticEI}, that may impact the performance of a chip, or cause a device to fail. Hence, there is a pressing need to investigate the dynamic model of semiconductors perturbed by stochastic forces within mathematical frameworks.
 The stochastically forced Euler-Poisson equations (SEP for short) in a bounded smooth domain $U\subset \mathds{R}^{3}$ reads as
\begin{equation}
\left\{\begin{array}{l}\label{3-D Euler Poisson}
\rho_{t}+ \nabla \cdot \left(\rho \u\right) =0, \\
 \d  \left(\rho\u\right) + \left(\nabla \cdot \left(\rho \u\otimes \u \right)+\nabla P\left(\rho\right) - \rho \nabla \Phi\right)\d t = - \frac{\rho\u}{\tau}\d t + \mathbb{F}\left(\rho,\mathbf{u}\right) \d  W, \\
 \triangle \Phi=\rho-b,
\end{array}\right.
\end{equation}
where $``\d" $ in \eqref{3-D Euler Poisson} is the differential notation with respect to time $t$, in comparison to gradient $\nabla$ and Laplacian $\triangle$ for spatial derivatives, $\rho$ is the electron density of semiconductors,
$\u$ denotes the particle velocity. $P\left(\rho\right)$ is the pressure, $\Phi$ is the electrostatic potential,
$\tau$ is the velocity relaxation time and $b(x)$ is called the doping profile, which is positive and immobile.
The above mentioned unknowns $\rho=\rho(\omega,t,x)$, $\u=\u(\omega,t,x)$, $\Phi=\Phi(\omega,t,x)$, and $P\left(\rho\right)=P\left(\rho(\omega,t,x)\right)$ are stochastic processes  as  functions with respect to $\omega$, $t$, and $x$, where $\omega$ is a sample in the complete probability space $\left(\Omega, \mathcal{F}, \mathbb{P}\right)$. For convenience, we use the simplified notions $\rho$, $\u$, $\Phi$, and $P\left(\rho\right)$ here and hereafter.
 $W$ is an $\mathcal{H}$-valued cylindrical Brownian motion defined on the filtrated probability space $\left(\Omega, \mathcal{F}, \mathbb{P}\right)$, where $\mathcal{H}$ is an auxiliary separable Hilbert space, $\mathcal{F}$ is the filtration, see the definitions of filtration and Wiener process in Appendix \ref{append}.

 Let $\{e_{k}\}_{k=1}^{+\infty}$ be an orthonormal basis in $\mathcal{H}$, then the Brownian motion $W$ can be written in the form of
$W=\sum\limits_{k=1}^{+\infty}e_{k}\beta_{k}$, where $\{\beta_{k}(t); k \in \mathds{N}, t\geqslant 0\}$ is a sequence of independent, real-valued standard Brownian motions.
Let $H$ be a Bochner space.  $\mathbb{F}\left(\rho,\mathbf{u}\right)$ is an $H$-valued operator from $ \mathcal{H}$ to $ \mathcal{H}$. Denoting the inner product in $\mathcal{H}$ as $  \langle\cdot,\cdot \rangle $,
the inner product
\begin{equation}
\langle\mathbb{F}\left(\rho,\mathbf{u}\right),e_{k}\rangle = \mathbf{F}_{k}\left(\rho,\u\right)
\end{equation}
 is an $H$-valued vector function, which shows the strength of the external stochastic forces by
 \begin{equation}
 \mathbb{F}\left(\rho,\mathbf{u}\right)\d  W=\sum\limits_{k=1}^{+\infty}\mathbf{F}_{k}\left(\rho,\u\right)\d  \beta_{k} e_{k}. 
 \end{equation}
Throughout the paper, we assume that
\begin{align}\label{condition for F for 3-d}
\mathbf{F}_{k}\left(\rho,\u\right) = a_{k}\rho \u Y\left(\rho, \u\right),
\end{align}
where ${a_{k}}$ are positive constants, $Y\left(\rho, \u\right)$ is a smooth function of $\rho$ and $\u$, and can be bounded by the homogeneous polynomials.

Subjected to the stochastic Euler-Poisson equations \eqref{3-D Euler Poisson}, the proposed boundary is the insulating boundary:
\begin{align}\label{insulated boundary condition}
\u\cdot \nu=0, \quad \nabla \Phi\cdot \nu=0,
\end{align}
where $\nu$ is the outer normal vector of $U$; and the initial data is:
\begin{align}\label{initial conditions}
(\rho,\u_0,\Phi)|_{t=0}=\left(\rho_{0}\left(\omega,x\right),~\u_{0}\left(\omega,x\right), ~\Phi_{0}\left(\omega,x\right) \right),
\end{align}
which is given in  the probability space  $\left(\Omega, \mathcal{F}, \mathbb{P}\right)$, $\rho_{0}\left(\omega,x\right)>0$. Here and hereafter, we simply denote the initial data by
 $\left(\rho_{0},\u_{0},\Phi_{0}\right)$ without confusion.

The hydrodynamic model of semiconductors was first introduced by Blotekjaer \cite{Blotekjaer}, which is the deterministically dynamical model presented by Euler-Poisson equations mathematically. For 1-D case, the initial-boundary value problems to Euler-Poisson equations with the insulating boundary and the Ohmic contact boundary were studied by Hsiao-Yang \cite{Hsiao2001AsymptoticsOI}, Li-Markowich-Mei \cite{Li-Markowich-Mei-2002}, respectively, where the solutions are showed to converge to the corresponding subsonic steady-states time-asymptotically, where the doping profile is needed to be flat: $|b'(x)|\ll 1$. Such a restriction was then released by Nishibata-Suzuki \cite{Suzuki-2007} and Guo-Strauss \cite{Guo2006StabilityOS} independently. For $N$-D case, Guo-Strauss \cite{Guo2006StabilityOS} first considered the deterministic 3-D Euler-Poisson equations in bounded domain with insulating boundary, and showed the convergence of solutions to the 3-D subsonic steady-states. Subsequently, Mei-Wu-Zhang \cite{MeiWu2021Stability} investigated the convergence to the steady-states for the $N$-D radial Euler-Poisson equations with the Ohmic contact boundary. For the whole space without boundary effects, the Cauchy problems to deterministic Euler-Poisson equations were extensively studied in \cite{DMRS,Huang1,Huang2,Huang2011,Huang2011SIAM,Kawashima1984SystemsOA}. For the case of free boundary with vacuum, we refer to \cite{Luo1,Mai,Zeng} and the references therein. For the formulation of singularities in compressible Euler-Poisson equations and the large time behavior of Euler equations with damping, one can refer to \cite{Sideris-Thomases-Wang} and \cite{Wang-Chen1998JDE}, respectively.


When the hydrodynamic model of semiconductors is counted into the stochastic affections, it then becomes the stochastic
Euler-Poisson equations with uncertain extra disturbances, see
\eqref{3-D Euler Poisson} with the Wiener process $\mathbb{F}\left(\rho,\mathbf{u}\right) \d  W$. This is a new model for semiconductor devices and never touched yet. The main issue of the paper is to investigate this 3-D SEP in bounded domain with insulating boundary, and are going to prove the convergence of solutions to the stochastic steady states.  The coefficient function of Wiener process $\mathbf{F}_{k}\left(\rho,\u\right)$, depending on the solutions $\rho$ and $\u$, is called the multiplicative noise. In most cases, the multiplicative noise magnifies the perturbation and thereby complicating the well-posedness of solutions for evolution systems. The stochastic forces are at most H\"older-$\frac{1}{2}-$continuous in time $t$, resulting in reduced regularity of velocity with respect to time. So from a mathematical standpoint, the study of the stochastic problem helps us to study how the solutions to stochastic Euler-Poisson equations behave in the absence of strong regularity in time. Further, this encourages exploring whether the desirable property remains under the influence of particular types of noise. This is the first attempt to study the asymptotic behavior of solutions to stochastic 3-D Euler-Poisson equations.

 For stochastic evolution systems, the solution is called the stationary solution provided that the increment of solutions during evolution is time-independent. Originally, the study of stationary measures dates back to the works of Hopf \cite{Hopf1932TheoryOM}, Doeblin \cite{Doeblin1938}, Doob \cite{Doob1953}, Halmos \cite{Halmos1946, Halmos1947}, Feller \cite{Feller1957}, and Harris and Robbins \cite{Harris-Robbins1953, Harris1956}, who contributed to the theory of discrete Markov processes from 1930s to 1950s. The study of invariant measure of fluid models dates back to Cruzeiro \cite{Cruzeiro1989} for stochastic incompressible Navier-Stokes equations in 1989, by Galerkin approximation with dimensions $D\geqslant 2$.
 Flandoli \cite{Flandoli1994} proved existence of an invariant measure by the ``remote start" method for 2-D incompressible Navier-Stokes equations in 1994.
One year later Flandoli-Gatarek \cite{Flandoli1-Gatarek1995} showed existence of stationary solution for 3-D incompressible Navier-Stokes equations by a different method with \cite{Cruzeiro1989}.
  In 2002, Mattingly \cite{Mattingly2002} proved the existence of exponentially attracting invariant measure with respect to initial data, for incompressible N-S equations.
Later, Goldys-Maslowski \cite{Goldys-Maslowski2005} showed that transition measures of the 2-D stochastic Navier-Stokes equations converge exponentially fast to the corresponding invariant measures in the distance of total variation. Then for 3-D case, Da Prato and Debussche \cite{DaPrato-Debussche2003} constructed a transition semigroup for 3-D stochastic Navier-Stokes equations without the uniqueness, which allows for rather irregular solutions. Flandoli-Romito \cite{Flandoli-Romito2008} used the classical Stroock-Varadhan
type argument to find the almost sure Markov selection. The above works are for the incompressible case. For stochastic compressible Navier-Stokes equations, Breit-Feireisl-Hofmanov\'a-Maslowski \cite{Breit-Feireisl-Hofmanova-Maslowski2019} proved the existence of stationary solutions. Compared to Navier--Stokes equations, the regularity effect of viscosity is lost for Euler system.
Hofmanov\'a-Zhu-Zhu \cite{Hofmanova-Zhu-Zhu2022} selected the dissipative global martingale solutions to the stochastic incompressible Euler system, and obtained the non-uniqueness of strong Markov solutions. Very recently, they \cite{Hofmanova-Zhu-Zhu2024} showed that stationary solution to the Euler equations is a vanishing viscosities limit in law of stationary analytically weak solutions to Navier-Stokes equations. In terms of the non-uniqueness studies, some scholars believe that a certain stochastic perturbation can provide
a regularizing effect of the underlying PDE dynamics. For instance, Flandoli-Luo \cite{Flandoli-Luo2021} showed that a noise of transport type prevents a vorticity blow-up in the incompressible Navier-Stokes equations. A linear multiplicative noise prevents the blow up of the velocity with high probability for the 3-D Euler system, which was shown by Glatt-Holtz-Vicol \cite{Glatt-Holtz-Vicol2014}. Gess-Souganidis \cite{Gess-Souganidis2017} investigated the large-time behavior and established the existence of an invariant measure for stochastic scalar conservation laws, demonstrating that an algebraic decay rate in time holds. In their work, they introduced a particular type of noise that provided stronger regularization properties for the problem. Then Dong-Zhang-Zhang \cite{DZZ2023} proved the existence of stationary solutions with the multiplicative noise.  For stochastic conservation laws, Da Prato-Gatarek studied the existence and uniqueness of invariant measure for stochastic Burgers equation \cite{DaPrato-Gatarek1995}. Da Prato-Zabczyk listed the basic theory of stationary solutions of general stochastic PDEs in view of invariant measure in book \cite{Da-Prato-Zabczyk2014}. Bedrossian-Liss \cite{Bedrossian-Liss2024} gave the existence of stationary measures for stochastic ordinary differential equations with a nonlinear term. To the best of our knowledge, the stationary solutions of SEP have not been explored previously. For our SEP, the electrostatic potential term $\rho \nabla \Phi\d t$ and the relaxation term $\frac{\rho\u}{\tau}\d t$ are actually damping terms providing better regularity than Euler equations. In this paper, we could show the existence and uniqueness of invariant measure in more regular space.

 It is worth noting that the stationary solution we consider is in view of invariant measure. In this paper, the concepts of stationary solution for stochastically forced system \eqref{3-D Euler Poisson} and steady state $\left(\bar{\rho}(\omega, x), \bar{\u}(\omega, x), \bar{\Phi}(\omega, x)\right)$ for the following deterministic system \eqref{steady state for nd insulating} are distinguished. Firstly, we establish the global existence and uniqueness of perturbed solutions around the steady state for the Euler-Poisson equations. Subsequently, we demonstrate the existence of stationary solutions and invariant measure based on the {\it a priori} energy estimates and weighted energy estimates.

We recall the steady state and recount the basic conclusion on the existence and uniqueness of $\left(\bar{\rho}(\omega, x), \bar{\u}(\omega, x), \bar{\Phi}(\omega, x)\right)$.
Within the probability space $\left(\Omega, \mathcal{F}, \mathbb{P}\right)$, the steady state $\left(\bar{\rho}(\omega, x), \bar{\u}(\omega, x), \bar{\Phi}(\omega, x)\right)$ are assumed to adhere to the following equations
 \begin{align}
\left\{\begin{array}{l}\label{steady state for nd insulating}
\nabla \cdot \left(\bar{\rho} \bar{\u}\right) =0, \\
\nabla \cdot \left(\bar{\rho} \bar{\u}\otimes \bar{\u}\right) + \nabla P\left(\bar{\rho}\right) - \bar{\rho} \nabla \bar{\Phi}= - \frac{\bar{\rho}\bar{\u}}{\tau}, \\
\triangle \bar{\Phi} =\bar{\rho}-b(x).
 \end{array}\right.
\end{align}
For the deterministic steady state with insulating boundary condition, Guo-Strauss \cite{Guo2006StabilityOS} gave the proof for existence and uniqueness of $\left(\bar{\rho}(x ), \bar{\u}(x ), \bar{\Phi}(x )\right)=\left(\bar{\rho}(x ), 0, \bar{\Phi}(x )\right)$.
By substituting ${\left(\ref{steady state for nd insulating}\right)}_{1}$ into ${\left(\ref{steady state for nd insulating}\right)}_{2}$, and take $\nabla \cdot $ on ${\left(\ref{steady state for nd insulating}\right)}_{2}$, we have
\begin{align}
\nabla \cdot \left(\bar{\rho}\bar{\u}\cdot\nabla \bar{\u} \right)+\triangle P\left(\bar{\rho}\right) -\nabla \cdot \left(\bar{\rho} \nabla \bar{\Phi}\right)=0.
\end{align}
If $\bar{\u}=0$, it deduces to
\begin{align} \label{equation of bar rho}
 P'\left(\bar{\rho}\right)\triangle \bar{\rho} +  P''\left(\bar{\rho}\right)\left|\nabla \bar{\rho}\right|^{2} -\nabla \bar{\rho}\nabla \bar{\Phi}-\bar{\rho}\left(\bar{\rho}-b\right)=0,
\end{align}
where $ P'\left(\bar{\rho}\right)>0= \left|\bar{\u}\right|^{2}$ so that the equation of $\bar{\rho}$ given in \eqref{equation of bar rho}, is uniformly elliptic. In this paper, we consider the subsonic case, i.e., the condition $P'\left(\rho\right)> \left|\u\right|^{2}$ holds under consideration.
For every $\omega \in \Omega$, $\left(\bar{\rho}(\omega, x), 0, \bar{\Phi}(\omega, x)\right)= \left(\bar{\rho}(x), 0, \bar{\Phi}(x)\right)$ is the unique solution of \eqref{steady state for nd insulating}, which is called steady state in this paper. We will denote $\left(\bar{\rho}(\omega, x), 0, \bar{\Phi}(\omega, x)\right)$ by $\left(\bar{\rho}, 0, \bar{\Phi}\right)$ for convenience in the following. The law of steady state is Dirac measure $\delta_{\bar{\rho}}\times \delta_{0}\times \delta_{\bar{\Phi}}$, see Appendix \ref{append}. We conclude the following lemma for steady state. Here $\bar{U}$ denotes the closed set of $U$.
\begin{proposition}\label{exist of steady state}
 Let $b(x)>0 $ in $\bar{U}$ and $P:\left(0, \infty\right)\rightarrow \left(0, \infty\right) $ be smooth with $P\left(0\right)=0$. Then there exists $\left(\bar{\rho} , \bar{\u} , \bar{\Phi} \right)$, $\forall \omega\in \Omega$, a unique smooth steady-state solution of the insulating problem with the Neumann boundary condition
\begin{align}
\frac{\partial \bar{\Phi} }{\partial \nu}|_{\partial U}\equiv 0,
\end{align}
such that there holds
\begin{align}
\bar{\rho} >\underline{\rho}>0,\quad \left|\nabla \bar{\rho} \right|>0,  \quad\bar{\Phi} >0,\quad \forall x\in \bar{U}, \quad \mathbb{P} ~{\rm a.s.},
\end{align}
where $\underline{\rho}$ is a constant, and
\begin{align}
  \int_{U} \bar{\rho} \d x=\int_{U}b(x)\d x, \quad \mathbb{P} ~{\rm a.s.}
\end{align}
\end{proposition}

Let $Q\left(\rho\right) $ be such that $\nabla Q\left(\rho\right) =\nabla \Phi $~(cf. \cite{Guo2006StabilityOS}). Then,
the steady state satisfies
\begin{align}
\nabla \bar{Q}(\bar{\rho}) =\nabla \bar{\Phi}, \quad  \triangle \bar{\Phi} =\bar{\rho} -b(x).
\end{align}

We consider the solutions $\left(\rho, \u , \Phi  \right) $  of hydrodynamic system around the steady state $\left(\bar{\rho} ,0,\bar{\Phi} \right) $ and we denote
\begin{align}
\sigma =\rho -\bar{\rho} , \quad \phi =\Phi -\bar{\Phi} .
\end{align}

Our main result is on the existence of solutions near the steady state, and asymptotic stability for insulating boundary condition.

 We denote by $\left\|\cdot\right\|$, $\left\|\cdot\right\|_{\infty}$, and $\left\|\cdot\right\|_{k}$ the $L^{2}(U)$-norm, $L^{\infty}(U)$-norm, and $H^{k}(U)$-norm, respectively. Let $\mathcal{L}\left(\cdot\right)$ be the law of random variables in $\left(\Omega, \mathcal{F}, \mathbb{P}\right)$, see the definition in Appendix \ref{append}. 
$L^{2m}\left(\Omega; C\left([0,T]; H^{k}\left(U\right)\right)\right)$ is the space in which the $2m$-th moment of $C\left([0,T]; H^{k}\left(U\right)\right)$-norm of random variables is bounded. We state our main theorems as follows.

\begin{theorem}[Global existence]
Let $U$ be a smooth bounded domain in $\mathds{R}^3$ 
and the pressure $P:(0, \infty) \rightarrow(0, \infty)$ be a smooth function, with $P(\cdot)>0$ and $P^{\prime}(\cdot)>0$. Let $ \left(\bar{\rho}, 0, \bar{\Phi}\right)$ be the smooth steady state in Proposition \ref{exist of steady state}.
and
\begin{align}\label{Initial formula of phi}
\triangle \Phi_{0}=\rho_{0}-b(x),
\end{align}
then in $\left(\Omega, \mathcal{F}, \mathbb{P}\right)$, there exists a unique global-in-time strong solution $(\rho,\u,\Phi)$ to the initial and boundary problem \eqref{3-D Euler Poisson}-\eqref{insulated boundary condition}-\eqref{initial conditions}:
\begin{align}
 \rho, ~ \u \in  L^{2m}\left(\Omega; C\left([0,T]; H^{3}\left(U\right)\right)\right), \quad \Phi\in L^{2m}\left(\Omega; C\left([0,T]; H^{5}\left(U\right)\right)\right), \forall ~T>0,
\end{align}
up to a modification, for any fixed integer $m\geqslant 2$.
\end{theorem}
Moreover, for the small perturbation problem, there hold the the existence of invariant measure and decay rate.
\begin{theorem}[Convergence to steady state]
Assume that the stochastic forces satisfies
\begin{align}\label{2 Condition for stochastic force}
 \sum a_{k}^{2}=1, \quad  \left|Y\left(\rho, \u\right)\right|\ls C \left|\rho \u\right|,\quad \left\|\nabla_{\rho,\u} Y\left(\rho, \u\right)\right\|_{L^{\infty}}\ls C, \quad \left\|\nabla_{\rho,\u}^{2}Y\left(\rho, \u\right)\right\|_{L^{\infty}}\ls C.
\end{align}
Here, $\nabla_{\rho,\u}$ denotes the differential operator with respect to $\rho$ and $\u$.
If there exists a constant $\eps>0$ such that the initial condition $\left(\rho_{0}, \u_{0}, \Phi_{0}\right)$ satisfies \eqref{Initial formula of phi} and
\begin{align}\label{initial conditions assumption}
\mathbb{E}\left[\left(\left\| \rho_{0}- \bar{\rho}\right\|_{3}^{2}+\left\| \u_{0}\right\|_{3}^{2}+ \left\|\nabla\Phi_{0}-\nabla\bar{\Phi}\right\|^{2}\right)^{m}\right] \ls \eps^{2m}, \quad \forall ~m\geqslant 2,
\end{align}
 then there hold the decay rate and the existence of invariant measure:
\begin{enumerate}
  \item there are positive constants $C$ and $\alpha$ such that the expectation
\begin{align}\label{asymptotic behavior}
& \mathbb{E}\left[\left(\sup\limits_{s\in[0,t]}\left(\left\| \rho - \bar{\rho} \right\|_{3}^{2}+\left\| \u\right\|_{3}^{2}+\left\| \nabla \Phi-\nabla \bar{\Phi} \right\|^{2} \right)\right)^{m}\right] \notag\\
 \ls &C e^{-\alpha mt }\mathbb{E}\left[\left(\left\| \rho_{0} - \bar{\rho}\right\|_{3}^{2}+\left\| \u_{0} \right\|_{3}^{2}+ \left\|\nabla\Phi_{0} -\nabla\bar{\Phi} \right\|^{2}\right)^{m}\right],
\end{align}
holds, where $C$ is independent on $t$ and $C$ is the $m$-th power of some constant;
  \item the invariant measure generated by $\frac{1}{T}\int_{0}^{T} \mathcal{L}\left(\rho\right)\times\mathcal{L}\left(\u\right)\times\mathcal{L}\left(\Phi\right) \d t$ is exactly the Dirac measure of steady state $\left(\bar{\rho}, 0, \bar{\Phi}\right)$.
\end{enumerate}
\end{theorem}

\begin{remark}
After passing to the limit $t\ra \infty$ in \eqref{asymptotic behavior}, the stationary solution coincides with the steady state $\mathbb{P}$ {\rm a.s.}, since the $m$-th moment of their difference tends to zero.
\end{remark}

\begin{remark}
If for any $\omega\in \Omega$,
\begin{align}
\left(\left\|\rho_{0} - \bar{\rho} \right\|_{3}^{2}+\left\| \u_{0} \right\|_{3}^{2}+\left\|\nabla\Phi_{0} -\nabla\bar{\Phi} \right\|^{2}\right)  \ls \eps^{2} ,
\end{align}
then there exists some constant $\tilde{C}$, such that the asymptotic stability holds $\mathbb{P}$~{\rm a.s.}:
\begin{align}
 &\sup\limits_{s\in[0,t]}\left(\left\| \rho -\bar{\rho} \right\|_{3}^{2}+\left\| \u \right\|_{3}^{2}+\left\|\nabla\Phi -\nabla\bar{\Phi} \right\|^{2}\right)\ls 2\tilde{C} e^{-\alpha t} \eps^{2}.
 \end{align}
Actually, by Chebyshev's inequality (see Appendix \ref{append}), it holds
\begin{align}
&\mathbb{P}\left[\left\{\omega\in \Omega|  \left\| \rho -\bar{\rho} \right\|_{3}^{2}+\left\| \u \right\|_{3}^{2}+\left\|\nabla\Phi -\nabla\bar{\Phi} \right\|^{2}
>  2\tilde{C} e^{-\alpha t} \eps^{2}\right\}\right]\notag\\
\ls &\frac{\mathbb{E}\left[\left|\left\| \rho -\bar{\rho} \right\|_{3}^{2}+\left\| \u \right\|_{3}^{2}+\left\|\nabla\Phi -\nabla\bar{\Phi} \right\|^{2}\right|^{m}\right]}{\left(2\tilde{C} e^{-\alpha t} \eps^{2}\right)^{m}}\\
\ls &\frac{\mathbb{E}\left[\tilde{C}e^{-\alpha t}\left(\left\| \rho_{0} - \bar{\rho} \right\|_{3}^{2}+\left\| \u_{0} \right\|_{3}^{2}+ \left\|\nabla\Phi_{0} -\nabla\bar{\Phi} \right\|^{2}\right)^{m}\right]}{\left(2\tilde{C} e^{-\alpha t} \eps^{2}\right)^{m}}=\frac{1}{2^{m}}.\notag
\end{align}
Let $m\ra \infty$, then it holds
$$
\mathbb{P}\left[\left\{\omega\in \Omega| \left\| \rho -\bar{\rho} \right\|_{3}^{2}+\left\| \u \right\|_{3}^{2}+\left\|\nabla\Phi -\nabla\bar{\Phi} \right\|^{2}>2\tilde{C} e^{-\alpha t} \eps^{2}\right\}\right]\ra 0,
$$
i.e., \\
 $\left\| \rho -\bar{\rho} \right\|_{3}^{2}+\left\| \u \right\|_{3}^{2}+\left\|\nabla\Phi -\nabla\bar{\Phi} \right\|^{2}\ls 2\tilde{C} e^{-\alpha t} \eps^{2}$ holds $\mathbb{P}$ {\rm a.s.} for every $s\in [0, t]$.
\end{remark}

\begin{remark}
The argument in this paper implies the same existence and asymptotic stability of solutions around the steady state for the 2-D system with insulating boundary conditions. Repeating the argument, by Sobolev's embedding, the existence of perturbed solutions and asymptotic stability of steady state for 1-D system with insulating boundary conditions holds: $\rho$ and $\u$ are in $L^{2m}\left(\Omega; C\left([0,T]; H^{2}\left(U\right)\right)\right)$, $\Phi\in L^{2m}\left(\Omega; C\left([0,T]; H^{4}\left(U\right)\right)\right)$ in $\left(\Omega, \mathcal{F}, \mathbb{P}\right)$.
\end{remark}

As mentioned before, the study of stochastic Euler-Poisson equations, totally from the existing studies for the deterministic case, is new and challenging. The idea of the proof is as follows. We first prove the local existence by Banach's fixed point theorem, then we establish the uniform energy estimates in time $t$ to show the global existence of $\left(\sigma , \u , \phi  \right) $. Furthermore, we prove the weighted energy estimates so that we can obtain the asymptotic stability for steady states with the insulating boundary conditions. The {\it a priori} estimates imply the tightness of approximates measures, which will converge to an invariant measure by Krylov-Bogoliubov's theorem in a complete probability space. The global existence does not require the small perturbation condition \eqref{2 Condition for stochastic force} and \eqref{initial conditions assumption}. However, the existence of invariant measure in Theorem \ref{exists of invariant meas}, requires \eqref{2 Condition for stochastic force} and \eqref{initial conditions assumption}.
From the weighted energy estimates, we then prove that the invariant measure for \eqref{3-D Euler Poisson} is exactly the law of steady state, c.f. Section \ref{section invariant measure}. This intricate relationship has not been uncovered in the asymptotic behavior analysis of stochastic Navier-Stokes equations yet \cite{Mattingly2002,Goldys-Maslowski2005}.

Here we  explain in detail the main difficulties  we face to and the strategies we are going to propose.
\begin{enumerate}
  \item \textbf{No temporal solutions due to the stochastic term.} Since Brownian motion is at most H\"oder-$\frac{1}{2}-$ continuous with respect to $t$ and it is nowhere differentiable, we do not have $\frac{\d W}{\d t}$ or $\frac{\d \left(\rho\u\right)}{\d t}$ either. No temporal derivative is involved in the norm of solutions. Thus, in deterministic cases \cite{Guo2006StabilityOS,MeiWu2021Stability}, the spatial estimates bounded by the temporal derivatives estimates like
      \begin{align}
      \qquad \|\left(\rho -\bar{\rho}\right)\|^2+\|\nabla \left(\rho -\bar{\rho}\right)\|^2 +\|\nabla \cdot \u\|^2 \leqslant C\left(\left\|u_t\right\|^2+\left\|\left(\rho -\bar{\rho}\right)_t\right\|^2+\|u\|^2+\interleave w\interleave^3\right),
      \end{align}
      do not apply to this stochastic case, where $\interleave \cdot \interleave$ means the temporal and spatial mixed derivatives. Consequently, the different energy estimates with the spatial and temporal mixed estimates are necessary in this paper. The spatial derivative estimates is based on It\^o's formula. We also symmetrize the system compatible with the insulating boundary conditions, to control the linear term and to facilitate the {\it a priori} estimates.

       It is interesting that the noise in form of \eqref{condition for F for 3-d} is in the higher order of $\u$ than the Lipschitz continuous on $\u$. This is reasonable when we consider the small perturbation around the steady state, which is different with most cases in which Lipshcitz continuous coefficients give birth to wellposedness. In this case the influence of stochastic force does not been exaggerated so much.
\item \textbf{Weighted energy estimates on account of the estimates of the stochastic integral.} Recalling the 3-D deterministic case \cite{Guo2006StabilityOS}, for instance, based on the energy estimates, one can obtain ordinary differential inequality (ODI). Then they multiply the ODI with the exponential function of $t$ directly to facilitate the stability analysis.
    However, in this paper, in order to estimates the stochastic term, we apply the Burkholder-Davis-Gundy's inequality to the stochastic integral of the Wiener process. Then the {\it a priori} estimates \eqref{energy estimate for insulating bdy con} is already in the form of time integrals rather than an ODI. Integration with respect to time twice could not imply the  asymptotic stability. Consequently, direct acquisition of asymptotic stability becomes challenging. To overcome this obstacle, we employ the weighted energy estimates. Moreover, the weight is determined by the {\it a priori} estimates which should be obtained first, cf. Section  \ref{Weighted decay estimates}.
\end{enumerate}

This paper is organized as follows. Section 2 is dedicated to establishing the global existence of solutions around the steady state. In Section  3, we investigate the asymptotic stability of semiconductor equations. Finally, in Section  4, we demonstrate the existence and property of invariant measures. Section  5 is the Appendix, in which we provide an overview of stochastic analysis theories that are employed in this study.

\smallskip
\smallskip

\section{Global existence of solutions}

In this section, we first establish the local existence of strong solutions by Banach's fixed point theorem.
Specifically, we derive the system of perturbed solutions in matrix notation by \eqref{3-D Euler Poisson}. In Step 1, We symmetrize it to simplify the energy estimates and proceed to linearize the system. In Step 2, following a standard procedure in view of Picard interation, we establish the uniform estimates onto mapping. We utilize It\^o's formula and the Burkholder-Davis-Gundy inequality to estimate the stochastic force. In Step 3, we demonstrate contraction. In Step 4, we get the {\it a priori} estimates in \S 2.2 so as to obtain the global existence of $(\rho-\bar{\rho},\u)$, or equivalently, $\rho ,\u \in L^{2m}\left(\Omega; C\left([0,T]; H^{3}\left(U\right)\right)\right)$ in \S 2.3. Step 5 is about the proof of global existence.

In form of $\left(\sigma, \u, \phi\right)=\left(\rho-\bar{\rho}, \u, \Phi-\bar{\Phi} \right)$, the hydrodynamic system deforms into
\begin{equation}
\left\{\begin{array}{l}\label{deformed sto Euler-Poisson}
\sigma_{t}+ \nabla \cdot \left(\left(\bar{\rho}+\sigma\right)\u\right) =0, \\
 \d  \left(\u\right) + \left(\left(\u \cdot\nabla \right)\u+\u +\nabla Q\left(\bar{\rho}+\sigma\right) -\nabla Q\left(\bar{\rho}\right)\right)\d t = \nabla \phi \d t + \frac{\mathbb{F}}{\bar{\rho}+\sigma}\d  W, \\
 \triangle \phi=\sigma.
\end{array}\right.
\end{equation}
Here we view $\tau$ as a constant $1$ without loss of generality for the stability analysis.
In terms of component, by Taylor's expansion, it holds
\begin{align}
 &\nabla Q\left(\bar{\rho}+\sigma\right) -\nabla Q\left(\bar{\rho}\right)=\left(Q\left(\bar{\rho}+\sigma\right)-Q\left(\bar{\rho}\right)\right)_{,i} \notag\\
=&Q'\left(\bar{\rho}+\sigma\right)\left(\bar{\rho}+\sigma\right)_{,i}-Q'\left(\bar{\rho}\right)\bar{\rho}_{,i} \notag\\
=&Q'\left(\bar{\rho}+\sigma\right)\sigma_{,i}+\left(Q'\left(\bar{\rho}+\sigma\right)-Q'\left(\bar{\rho}\right)\right)\bar{\rho}_{,i} \\
=&Q'\left(\bar{\rho}+\sigma\right)\sigma_{,i}+Q''\left(\bar{\rho}\right)\sigma \bar{\rho}_{i}+\left(Q'\left(\bar{\rho}+\sigma\right)-Q'\left(\bar{\rho}\right)-Q''\left(\bar{\rho}\right)\sigma\right)\bar{\rho}_{,i}\notag\\
:= & Q'\left(\bar{\rho}+\sigma\right)\sigma_{,i}+Q''\left(\bar{\rho}\right)\sigma \bar{\rho}_{,i}+ h_{i},\notag
\end{align}
where $\left(\cdot\right)_{, i}$ means the derivative with respect to $x_{i}$, and
\begin{align}
h_{i}=O\left(\sigma^{2}\right).
\end{align}
In term of component, there holds
\begin{align}\phi_{,i}=\triangle^{-1}\sigma_{,i},
\end{align}
  where $\triangle^{-1}$ is well-defined under the condition \eqref{insulated boundary condition}. In matrix notation, denoting $\w=\left[\begin{array}{l}\sigma \\ \u\end{array}\right]$, we write the system as
\begin{align}\label{sto Euler-Poisson system}
\d  \w+\left(\mathcal{A}^{1}\w_{,1}+\mathcal{A}^{2}\w_{,2}+\mathcal{A}^{3}\w_{,3}+\mathcal{B}\w+\mathcal{L}_{\u}\right)\d t= \mathcal{L}_{\mathcal{\phi}}\d t+ f,
\end{align}
where
\begin{align}
\mathcal{A}^{1}=\left[\begin{array}{cccc}
u^1 & \bar{\rho}+\sigma & 0 & 0 \\
Q'\left(\bar{\rho}+\sigma\right) & u^1 & 0 & 0 \\
0 & 0 & u^1 & 0 \\
0 & 0 & 0 & u^1
\end{array}\right],
\end{align}
\begin{align}
\mathcal{A}^{2}=\left[\begin{array}{cccc}
u^2  & 0 & \bar{\rho}+\sigma  & 0 \\
0 & u^2 & 0 & 0 \\
Q'\left(\bar{\rho}+\sigma\right) & 0 & u^2 & 0 \\
0 & 0 & 0 & u^2
\end{array}\right],
\end{align}
\begin{align}
\mathcal{A}^{3}=\left[\begin{array}{cccc}
u^3  & 0  & 0 & \bar{\rho}+\sigma \\
0 & u^3 & 0 & 0 \\
0 & 0 & u^3 & 0 \\
Q'\left(\bar{\rho}+\sigma\right) & 0 & 0 & u^3
\end{array}\right],
\end{align}
\begin{align}
\mathcal{B}=\left[\begin{array}{cccc}
0 & \bar{\rho}_{,1} & \bar{\rho}_{,2} & \bar{\rho}_{,3} \\
Q''\left(\bar{\rho}\right) \bar{\rho}_{,1} & 0 & 0 & 0 \\
Q''\left(\bar{\rho}\right) \bar{\rho}_{,2} & 0 & 0 & 0 \\
Q''\left(\bar{\rho}\right) \bar{\rho}_{,3} & 0 & 0 & 0
\end{array}\right],
\end{align}
\begin{align}
\mathcal{L}_{\u}=\left[\begin{array}{c}0\\ u^{1}\\ u^2\\ u^3\end{array}\right],\quad \mathcal{L}_{\mathcal{\phi}}=\left[\begin{array}{c}0 \\ \triangle^{-1}\sigma_{,1}\\ \triangle^{-1}\sigma_{,2}\\  \triangle^{-1}\sigma_{,3}\end{array}\right],
\quad f=-\left[\begin{array}{c}
0\\ h\left(\sigma\right)_{,1}-\mathbb{F}^{1}\d  W \\ h\left(\sigma\right)_{,2}-\mathbb{F}^{2}\d  W\\ h\left(\sigma\right)_{,3}-\mathbb{F}^{3}\d  W
\end{array}\right].
\end{align}
\smallskip
\begin{flushleft}
\textbf{Step 1: Symmetrizing and Linearizing.}
\end{flushleft}
We define the symmetrizer $\mathcal{D}=\operatorname{diag}\left[Q'\left(\bar{\rho}+\sigma\right),\bar{\rho}+\sigma,\bar{\rho}+\sigma,\bar{\rho}+\sigma\right] := \operatorname{diag}\left[d_{1}, d_{2}, d_{3}, d_{4} \right]$. Then the system deforms into
\begin{align}\label{symmetrized system}
\mathcal{D} \d  \w+\left(\tA^{1}\w_{,1}+\tA^{2}\w_{,2}+\tA^{3}\w_{,3}+\B\w+\tilde{\mathcal{L}}_{\u}\right)\d t= \tilde{\mathcal{L}}_{\phi}\d t+ \tilde{f},
\end{align}
where
\begin{align}
\tA^{1}=\left[\begin{array}{cccc}
u^1 Q'\left(\bar{\rho}+\sigma\right) & \left(\bar{\rho}+\sigma\right) Q'\left(\bar{\rho}+\sigma\right) & 0 & 0 \\
\left(\bar{\rho}+\sigma\right) Q'\left(\bar{\rho}+\sigma\right) & \left(\bar{\rho}+\sigma\right)u^1 & 0 & 0 \\
0 & 0 & \left(\bar{\rho}+\sigma\right)u^1 & 0 \\
0 & 0 & 0 & \left(\bar{\rho}+\sigma\right) u^1
\end{array}\right],
\end{align}
\begin{align}
\tA^{2}=\left[\begin{array}{cccc}
u^2 Q'\left(\bar{\rho}+\sigma\right)  & 0 & \left(\bar{\rho}+\sigma\right) Q'\left(\bar{\rho}+\sigma\right)  & 0 \\
0 & \left(\bar{\rho}+\sigma\right) u^2 & 0 & 0 \\
 \left(\bar{\rho}+\sigma\right) Q'\left(\bar{\rho}+\sigma\right) & 0 & \left(\bar{\rho}+\sigma\right) u^2 & 0 \\
0 & 0 & 0 & \left(\bar{\rho}+\sigma\right) u^2
\end{array}\right],
\end{align}
\begin{align}
\tA^{3}=\left[\begin{array}{cccc}
u^3 Q'\left(\bar{\rho}+\sigma\right)  & 0  & 0 &  \left(\bar{\rho}+\sigma\right) Q'\left(\bar{\rho}+\sigma\right) \\
0 &\left(\bar{\rho}+\sigma\right) u^3 & 0 & 0 \\
0 & 0 & \left(\bar{\rho}+\sigma\right)u^3 & 0 \\
 \left(\bar{\rho}+\sigma\right) Q'\left(\bar{\rho}+\sigma\right) & 0 & 0 & \left(\bar{\rho}+\sigma\right)u^3
\end{array}\right],
\end{align}
\begin{align}
\B=\left[\begin{array}{cccc}
0 & \bar{\rho}_{,1} & \bar{\rho}_{,2} & \bar{\rho}_{,3} \\
Q''\left(\bar{\rho}\right) \bar{\rho}_{,1} & 0 & 0 & 0 \\
Q''\left(\bar{\rho}\right) \bar{\rho}_{,2} & 0 & 0 & 0 \\
Q''\left(\bar{\rho}\right) \bar{\rho}_{,3} & 0 & 0 & 0
\end{array}\right],
\end{align}
\begin{align}\label{Lu and Lphi}
\tilde{\mathcal{L}}_{\u}=\left(\bar{\rho}+\sigma\right)\left[\begin{array}{c}0\\ u^{1}\\ u^2\\ u^3\end{array}\right],\quad \tilde{\mathcal{L}}_{\phi}=\left(\bar{\rho}+\sigma\right)\left[\begin{array}{c}0 \\ \triangle^{-1}\sigma_{,1}\\ \triangle^{-1}\sigma_{,2}\\  \triangle^{-1}\sigma_{,3}\end{array}\right],
\end{align}
\begin{align}
\tilde{f}=-\left(\bar{\rho}+\sigma\right)\left[\begin{array}{c}
0\\ h\left(\sigma\right)_{,1}\d t-\mathbb{F}^{1}\d  W \\ h\left(\sigma\right)_{,2}\d t-\mathbb{F}^{2}\d  W\\ h\left(\sigma\right)_{,3}\d t-\mathbb{F}^{3}\d  W
\end{array}\right].
\end{align}

\subsection{Local existence}\label{local existence for 3-D insul}
In this subsection, the main estimates for the stochastic forces are taking the assumption of \eqref{2 Condition for stochastic force} for instance.
Similar to the approach in \cite{Kawashima2003LargeTimeBO,Nishibata2009AsymptoticSO}, we first linearize the system and then we use Banach's fixed point theorem to get the local existence by the {\it a priori} energy estimates.

The linearized system is
\begin{align}
\label{linearized system}
&\mathcal{D}(\hat{\sigma}) \d  \w+\left(\tA^{1}\left(\hat{\w}\right)\w_{,1}+\tA^{2}(\hat{\w})\w_{,2}+\tA^{3}(\hat{\w})\w_{,3}+\B\w\right)\d t \\
 = &-\tilde{\mathcal{L}}_{\hat{\u}}\left(\hat{\sigma}, \u\right)\d t+\tilde{\mathcal{L}}_{\hat{\phi}}\left(\hat{\sigma}, \hat{\phi}\right)\d t+ \tilde{f}(\hat{\w}),\notag
\end{align}
where $\hat{\w}=\left[\begin{array}{l}\hat{\sigma} \\ \hat{\u}\end{array}\right]$ is given, $\hat{\sigma} \in C\left([0,T]; H^{3}\left(U\right)\right)$, $\hat{\u} \in C\left([0,T]; H^{3}\left(U\right)\right)$. We denote $M=\sup\limits_{t\in[0,T]} \left\|\hat{\sigma}, \hat{\u}\right\|_{3}$.

\begin{flushleft}
\textbf{Step 2: Estimates for the uniform upper bound.}
\end{flushleft}

By It\^o's formula (see Appendix \ref{append}), it holds
\begin{align}
\int_{U}\d \left(  \frac{1}{2} \mathcal{D} \w \cdot \w \right) \d x = \int_{U} \frac{1}{2} \d   \mathcal{D} \w \cdot \w \d x + \int_{U} \mathcal{D} \w\cdot \d  \w \d x+ \int_{U}  \mathcal{D}\mathbb{F}\cdot\mathbb{F} \d x \d t.
\end{align}
We integrate
\begin{align}
&\mathcal{D} \d  \w\cdot \w+\left(\tA^{1}\left(\hat{\w}\right)\w_{,1} \cdot \w + \tA^{2}\left(\hat{\w}\right)\w_{,2}\cdot \w+\tA^{3}\left(\hat{\w}\right)\w_{,3}\cdot \w+ \B \w\cdot \w  \right)\d t\\
=& -\tilde{\mathcal{L}}_{\hat{\u}}\cdot \w\d t+\tilde{\mathcal{L}}_{\hat{\phi}}\cdot \w\d t+ \tilde{f}\cdot \w\notag
\end{align}
over the domain $U$, we gain
\begin{align}
 & \int_{U} \mathcal{D} \w\cdot \d  \w \d x \notag\\
=& \int_{U}\left(-\left(\tA^{1}\w_{,1}\cdot \w + \tA^{2}\w_{,2}\cdot \w+ \tA^{3}\w_{,3}\cdot \w + \B \w\cdot \w\right) - \tilde{\mathcal{L}}_{\hat{\u}}\cdot \w + \tilde{\mathcal{L}}_{\hat{\phi}}\cdot \w  \right) \d x \d t \\
 & +\int_{U} \nabla h\left(\hat{\sigma}\right)\cdot \w \d x \d t + \int_{U}\mathcal{D} \mathbb{F}\d  W\cdot \w\d x . \notag
\end{align}
By the integration by parts, we have
\begin{align}\label{1 energy est in existence to ls}
 &\int_{U} \left( \tA^{1}\w_{,1}\cdot \w + \tA^{2}\w_{,2}\cdot \w+ \tA^{3}\w_{,3}\cdot \w + \B \w\cdot \w \right) \d x \d t \notag\\
=&\int_{U} \left(-\frac{1}{2}\left(\w \tA^{1}_{1}\w +\w \tA^{2}_{2}\w + \w \tA^{3}_{3}\w \right)+ \B \left|\w\right|^{2}\right) \d x \d t + \int_{U} \left(\w \tA^{j}\w \right)_{,j} \d x \d t.
\end{align}
On account of the insulated boundary condition $\hat{\u}\cdot \nu|_{\partial U} =0$, it holds
\begin{align}\label{2 energy est in existence to ls}
&\int_{U} \left(\w \tA^{j}\w \right)_{,j} \d x\\
 = &\int_{\partial U} \left(\left(\hat{\u}\cdot \nu\right)\left(Q'\left(\bar{\rho}+\hat{\sigma}\right)\sigma^{2}+\left(\bar{\rho}+\hat{\sigma}\right)\left|\u\right|^{2}+2\hat{\rho} \u Q'\left(\bar{\rho}+\hat{\sigma}\right)\left(\bar{\rho}+\hat{\sigma}\right)\right)\right)\d S \equiv 0. \notag
\end{align}
In summary, there holds
\begin{align}
&\int_{U}\d \left( \frac{1}{2} \mathcal{D} \w \cdot \w \right) \d x \notag\\
= &\int_{U} \frac{1}{2} \d \mathcal{D} \w \cdot \w \d x -\int_{U} \left(-\frac{1}{2}\left(\w \tA^{1}_{1}\w +\w \tA^{2}_{2}\w + \w \tA^{3}_{3}\w \right)+ \B \left|\w\right|^{2}\right) \d x \d t \\
&+ \int_{U}\left(-\tilde{\mathcal{L}}_{\hat{\u}}\cdot \w + \tilde{\mathcal{L}}_{\hat{\phi}}\cdot \w  \right) \d x \d t+\int_{U} \nabla h\left(\hat{\sigma}\right)\cdot \w \d x \d t + \int_{U}\mathcal{D} \mathbb{F}\d  W\cdot \w\d x \notag \\
 &  + \int_{U}  \mathcal{D}\mathbb{F}\cdot\mathbb{F} \d x \d t .\notag
\end{align}
 Direct calculation shows that
\begin{align}
&-\frac{1}{2}\tilde{\mathcal{A}}^{i}_{,i}+\B
-\operatorname{diag}\left[-\frac{1}{2}\left(u^{i}Q'\left(\rho\right)\right)_{,i},-\frac{1}{2}\left(u^{i}\rho\right)_{,i},
-\frac{1}{2}\left(u^{i}\rho\right)_{,i},-\frac{1}{2}\left(u^{i}\rho\right)_{,i}\right]\notag\\
=& \left[\begin{array}{cccc}
0 & \footnotesize{-\frac{1}{2}\left\{\rho q\right\}_{,1}+\bar{\rho}_{,1} q} & -\frac{1}{2}\left\{\rho q\right\}_{,2}+\bar{\rho}_{,2} q & -\frac{1}{2}\left\{\rho q\right\}_{,3}+\bar{\rho}_{,3} q\\
-\frac{1}{2}\left\{\rho q\right\}_{,1}+\rho Q''\left(\bar{\rho}\right) \bar{\rho} & 0 & 0 & 0 \\
-\frac{1}{2}\left\{\rho q\right\}_{,2}+\rho Q''\left(\bar{\rho}\right) \bar{\rho} & 0 & 0 & 0 \\
-\frac{1}{2}\left\{\rho q\right\}_{,3}+\rho Q''\left(\bar{\rho}\right) \bar{\rho} & 0 & 0 & 0
\end{array}\right]
 \end{align}
 is anti-symmetric \cite{Guo2006StabilityOS}, where $q=Q'\left(\rho\right)$. Then we estimate
\begin{align}\label{3 energy est in existence to ls}
\int_{U} \left(-\frac{1}{2}\left(\w \tA^{1}_{1}\w +\w \tA^{2}_{2}\w + \w \tA^{3}_{3}\w \right)+ \B \left|\w\right|^{2}\right) \d x \d t \ls C \left\| \w \right\|^{2} \left(\left\| \hat{\sigma} \right\|_{3}+\left\| \hat{\u} \right\|_{3}\right)\d t.
\end{align}
Recalling \eqref{Lu and Lphi}, we have
\begin{align}
\int_{U} \tilde{\mathcal{L}}_{\hat{\u}}\left(\hat{\sigma}, \u\right)\cdot \w \d x \d t= \int_{U} \left(\bar{\rho}+\hat{\sigma}\right)\left|\u\right|^{2}\d x \d t \geqslant C\int_{U} \bar{\rho}\left|\u\right|^{2}\d x \d t,
\end{align}
and
\begin{align}
\int_{U} \tilde{\mathcal{L}}_{\phi}\left(\hat{\sigma}, \phi\right)\cdot \w \d x \d t= & \int_{U} \left(\bar{\rho}+\hat{\sigma}\right)\nabla \phi \cdot \u \d x \d t = -\int_{U}\nabla \cdot \left(\left(\bar{\rho}+\hat{\sigma}\right)\u\right)\phi \d x \d t  \\
=& \int_{U}\sigma_{t}\phi \d x \d t = \int_{U}\left(\triangle \phi \right)_{t}\phi \d x \d t=-\d  \int_{U}\left|\nabla\phi\right|^{2} \d x. \notag
\end{align}
For $\tilde{f}$, there holds
\begin{align}
\int_{0}^{t}\int_{U} \tilde{f}\cdot \w \d x = & C \int_{0}^{t}\int_{U} \hat{\sigma}^{2}\cdot\u \d x \d t + \left|\int_{0}^{t}\int_{U}\left(\bar{\rho}+\hat{\sigma}\right)\mathbb{F}\cdot\u \d x \d  W\right|,
\end{align}
where
\begin{align}
\int_{0}^{t}\int_{U} \hat{\sigma}^{2}\cdot\u \d x \d s \ls  \int_{0}^{t}\left\| \hat{\sigma} \right\|_{2} \left\|\hat{\sigma}\right\| \left\|\w\right\| \d s.
\end{align}
One can see the definition of stochastic integral $\int_{0}^{t}\int_{U}\left(\bar{\rho}+\hat{\sigma}\right)\mathbb{F}\cdot\u \d x \d  W$ in Appendix \ref{append}.
Since $\left|\mathbb{F}\left(\hat{\sigma}, \hat{\u}\right)\right|^{2}\ls C \left|\left(\bar{\rho}+\hat{\sigma}\right)\hat{\u}\right|^{4}$, there hold
\begin{align}
 \int_{U} \mathcal{D}\mathbb{F}\cdot\mathbb{F}\d x \d t \ls C \left\|\hat{\u}\right\|^{2}\left(\left\|\hat{\u}\right\|_{3}^{2}\left\|\hat{\sigma}\right\|_{3}^{4}\right)\d t \ls CM^{8}\d t ,
 \end{align}
 and
\begin{align}\label{estimate for sto integral in local existence}
    &\mathbb{E}\left[\left|\int_{0}^{t} \int_{U} \mathbb{F}\cdot\u\d x \d  W \right|^{m}\right]
\ls \mathbb{E}\left[\left(\int_{0}^{t} \left|\int_{U} \mathbb{F}\cdot\u\d x\right|^{2}\d s\right)^{\frac{m}{2}}\right]\notag\\
\ls &\mathbb{E}\left[\left(C\int_{0}^{t} \left|\int_{U} \left|\left(\bar{\rho}+\hat{\sigma}\right)\hat{\u}\right|^{2}\u\d x\right|^{2}\d s\right)^{\frac{m}{2}}\right]\\
\ls &\mathbb{E}\left[\left(C \sup\limits_{s\in[0, t]}\left\|\u \right\|^{2} \int_{0}^{t} \left\| \left(\bar{\rho}+\hat{\sigma}\right)\right\|_{3}^{4} \left\| \hat{\u} \right\|_{3}^{4} \d s\right)^{\frac{m}{2}}\right]\notag\\
\ls &\delta_{1}^{m}\mathbb{E}\left[\left(\sup\limits_{s\in[0, t]} \left\|\u\right\|^{2}\right)^{m}\right] +C_{\delta_{1}}^{m}\mathbb{E}\left[\left(\int_{0}^{t} \left\| \hat{\u} \right\|_{3}^{4}\left\| \left(\bar{\rho}+\hat{\sigma}\right)\right\|_{3}^{4} \d s\right)^{m}\right], \notag
\end{align}
by Burkholder-Davis-Gundy's inequality (see Appendix \ref{append}), where $\delta_{1}$ is taken such that $\delta_{1} \sup\limits_{s\in [0,t]} \left\|\u \right\|^{2}$ can be balanced by the left side. We estimate
\begin{align}
 &\int_{U} \w \left(\d \mathcal{D}\right)\w\d x \notag\\
=& \int_{U}\w\left(\operatorname{diag}\left\{Q'\left(\bar{\rho}+\hat{\sigma}\right)_{t},\left(\bar{\rho}+\hat{\sigma}\right)_{t} ,\left(\bar{\rho}+\hat{\sigma}\right)_{t},\left(\bar{\rho}+\hat{\sigma}\right)_{t}\right\}\right)\w \d x \d t\notag\\
=& \int_{U} \left(Q''\left(\bar{\rho}+\hat{\sigma}\right)\hat{\sigma}_{t}\sigma^{2}+\hat{\sigma}_{t}\left|\u\right|^{2}\right)\d x \d t\\
=&  \int_{U} \left(Q''\left(\bar{\rho}\right)+O\left(\hat{\sigma}\right)\right)\left(-\nabla\cdot\left(\left(\bar{\rho}+\hat{\sigma}\right)\hat{\u}\right)\right)\sigma^{2}\d x + \int_{U} \left(-\nabla\cdot\left(\left(\bar{\rho}+\hat{\sigma}\right)\hat{\u}\right)\right)\left|\u\right|^{2}\d x \d t\notag \\
\ls & C \left\|\w\right\|^{2}\left(\left\|\hat{\u}\right\|_{2} +\left\|\hat{\sigma}\right\|_{2}\left\|\hat{\u}\right\|_{2} +\left\|\hat{\sigma} \right\|_{2} \left\|\hat{\sigma}\right\|_{3}\left\|\hat{\u}\right\|_{2} \right)\d t, \notag
\end{align}
where $O$ means the same order.
In summary, there holds
\begin{align}
   & \mathbb{E}\left[\left( \sup\limits_{s\in[0,t]} \int_{0}^{s} \d \left(\int_{U}\bar{\rho}\left|\w\right|^{2}\d x + \int_{U}\left|\nabla\phi\right|^{2}\d x\right) +c_{1}\int_{0}^{t}\int_{U}\bar{\rho} \left|\u\right|^{2}\d x \d s \right)^{m}\right] \notag\\
\ls &  \mathbb{E}\left[\left( C\int_{0}^{t}\left( \left\|\hat{\u}\right\|^{2}\left(1+\left\|\hat{\sigma}\right\|_{H^{1}}^{2}\right)+\left\| \w \right\|^{2}\left\| \hat{\w} \right\|+\left\|\hat{\w}\right\|^{2}\left\|\w\right\|\right)\d s \right)^{m}\right]\notag\\
 &+ \mathbb{E}\left[\left(C\int_{0}^{t} \interleave \hat{\u} \interleave^{4}\interleave \left(\bar{\rho}+\hat{\sigma}\right)\interleave^{4} \d s\right)^{m}\right] \\
 &+\mathbb{E}\left[\left(C\int_{0}^{t}\left\|\w\right\|^{2}\left(\left\|\hat{\u}\right\|_{2} +\left\|\hat{\sigma}\right\|_{2}\left\|\hat{\u}\right\|_{2} +\left\|\hat{\sigma} \right\|_{2} \left\|\hat{\sigma}\right\|_{3}\left\|\hat{\u}\right\|_{2}\right) \d s\right)^{m}\right]\notag\\
\ls & \mathbb{E}\left[\left(C  \int_{0}^{t}\left(M+M^{2}+M^{4}\right)\d s + \int_{0}^{t} M \left\| \w \right\|^{2} \d s + \int_{0}^{t} M^{2}\left\|\w\right\| \d s\right)^{m}\right]\notag\\
&+ \mathbb{E}\left[\left(C \int_{0}^{t}(M^{4}+ M^{8}) \d s\right)^{m}\right]+C\mathbb{E}\left[\left(\int_{0}^{t}\left\|\w\right\|^{2}\left(M +M^{2}+M^{3}\right)\d s\right)^{m}\right]. \notag
\end{align}
Furthermore, for $\bar{\rho}$ with a positive lower bound, we have
\begin{align}
   & \mathbb{E}\left[\left(\sup\limits_{s\in[0,t]}\int_{0}^{s} \d \left(\int_{U}\left|\w\right|^{2}\d x + \int_{U}\left|\nabla\phi\right|^{2}\d x\right) \right)^{m}\right] \ls
 C_{M,m}\left(t^{m}+\mathbb{E}\left[\left(\int_{0}^{t} \left\|\w\right\|^{2}\d s \right)^{m}\right]\right),
\end{align}
where $C_{M,m}$ is a constant depending on $m, M$.
Similarly, we take higher-order derivatives to the system \eqref{linearized system} up to third order, and we do the {\it a priori} estimates. There holds
\begin{align}
   & \mathbb{E}\left[\left(\sup\limits_{s\in[0,t]}\int_{0}^{s} \d \left(\left\|\w\right\|_{3}^{2} + \left\|\nabla \phi\right\|^{2} \right) \right)^{m}\right]\\
  \ls &C_{M,m}\left(t^{m}+\mathbb{E}\left[\left(\int_{0}^{t} \left(\left\|\w\right\|_{3}^{2} + \left\|\nabla \phi\right\|^{2} \right) \d s \right)^{m}\right]\right).\notag
\end{align}
By Gr\"onwall's inequality, we have $\w\in L^{2m}\left(\Omega; C\left([0,T];H^{3}\left(U\right)\right)\right)$. More precisely,
\begin{align}\label{local estimates}
&\mathbb{E}\left[\left(\sup\limits_{s\in[0,t]}\left(\left\|\w\right\|_{3}^{2} + \left\|\nabla \phi\right\|^{2} \right)(s) \right)^{m}\right]\notag\\
\ls &\mathbb{E}\left[\left(\left(\left\|\w\right\|_{3}^{2} + \left\|\nabla \phi\right\|^{2} \right) (0)\right)^{m}\right]+C_{M,m}t^{m} \\
&+ \int_{0}^{t}\left(\mathbb{E}\left[\left(\left(\left\|\w\right\|_{3}^{2} + \left\|\nabla \phi\right\|^{2} \right) (0) \right)^{m}\right]+C_{M,m}t^{m}\right)C_{M,m}e^{\int_{0}^{s}C_{M,m}\d \tau}\d s\notag\\
\ls & \left(\mathbb{E}\left[\left(\left(\left\|\w_{0}\right\|_{3}^{2} + \left\|\nabla \phi_{0}\right\|^{2} \right) \right)^{m}\right]+C_{M,m}t^{m}\right) e^{C_{M,m}t}.\notag
\end{align}
  From the estimates of time shift of solutions similar as \eqref{local estimates}, by applying Kolmogorov-Centov's theorem (see Appendix \ref{append}), following the standard argument in stochastic analysis \cite{BreitFeireislHofmanova-book2018}, we deduce the time continuity of $\w$ up to a modification in probability space $\left(\Omega, \mathcal{F}, \mathbb{P}\right)$, and we omit the details.

The iteration scheme is
\begin{align}\label{iteration system}
&\mathcal{D}\left(\sigma_{n-1}\right) \d  \w_{n} + \left(\tA^{1}\left(\w_{n-1}\right)\w_{n,1} + \tA^{2}\left(\w_{n-1}\right)\w_{n,2}+\tA^{3}\left(\w_{n-1}\right)\w_{n,3}+ \B\w_{n}\right)\d t  \\
 = &-L_{\u_{n-1}}\left(\sigma_{n-1}, \u_{n}\right)\d t + \tilde{\mathcal{L}}_{\phi}\left(\sigma_{n-1}, \phi_{n}\right)\d t+ \tilde{f}\left(\w_{n-1}\right). \notag
\end{align}
By energy estimates \eqref{local estimates}, we take $T_{0}$ such that
\begin{align}
e^{C_{M,m}T_{0}}\ls 2, \quad C_{M,m}T_{0} \ls \mathbb{E}\left[\left(\left(\left\|\w_{0}\right\|_{3}^{2} + \left\|\nabla \phi_{0}\right\|^{2} \right) \right)^{m}\right],
\end{align}
if
\begin{align}
\mathbb{E}\left[\left(\sup\limits_{s\in[0,t]} \left\| \w_{n-1}(s)\right\|_{3}^{2}\right)^{m}\right] \ls 4\mathbb{E}\left[\left(\left(\left\|\w_{0}\right\|_{3}^{2} + \left\|\nabla \phi_{0}\right\|^{2} \right) \right)^{m}\right] ,
\end{align}
 then
 \begin{align}
 \mathbb{E}\left[\left(\sup\limits_{s\in[0,t]}\left\| \w_{n}(s)\right\|_{3}^{2}\right)^{m}\right] \ls 4\mathbb{E}\left[\left(\left(\left\|\w_{0}\right\|_{3}^{2} + \left\|\nabla \phi_{0}\right\|^{2} \right) \right)^{m}\right] .
 \end{align}

 \begin{remark}
 For general stochastic forces without the condition \eqref{2 Condition for stochastic force}, there also holds
\begin{align}
   & \mathbb{E}\left[\left(\sup\limits_{s\in[0,t]}\int_{0}^{s} \d \left(\left\|\w\right\|_{3}^{2} + \left\|\nabla \phi\right\|^{2} \right) \right)^{m}\right]\\
  \ls &C_{M,m}\left(t^{m}+\mathbb{E}\left[\left(\int_{0}^{t} \left(\left\|\w\right\|_{3}^{2} + \left\|\nabla \phi\right\|^{2} \right) \d s \right)^{m}\right]\right).\notag
\end{align}
with another expression of the constant $C_{M,m}$.
Thus, we get the uniform bound by Gr\"onwall's inequality similarly to the above statement.
 \end{remark}
\begin{flushleft}
\textbf{Step 3: Contraction.}
\end{flushleft}

For $\left\|\w_{n}-\w_{n-1}\right\|_{3} $, we show that it is a Cauchy sequence. $\left(\w_{n}-\w_{n-1}\right)$ satisfies
\begin{align}\label{contraction system}
&\mathcal{D}\left(\sigma_{n-1}\right) \left(\d  \w_{n}-\d  \w_{n-1}\right) + \left(\mathcal{D}\left(\sigma_{n-1}\right)-\mathcal{D}\left(\sigma_{n-2}\right)\right)\d  \w_{n-1}\notag\\
 &+ \tA^{1}\left(\w_{n-1}\right)\left(\w_{n,1}-\w_{n-1,1}\right)\d t+\left(\tA^{1}\left(\w_{n-1}\right)-\tA^{1}\left(\w_{n-2}\right)\right)\w_{n-1,1}\d t\notag\\
&+ \tA^{2}\left(\w_{n-1}\right)\left(\w_{n,2}-\w_{n-1,2}\right)\d t+\left(\tA^{2}\left(\w_{n-1}\right)-\tA^{2}\left(\w_{n-2}\right)\right)\w_{n-1,2}\d t\notag\\
&+\tA^{3}\left(\w_{n-1}\right)\left(\w_{n,3}-\w_{n-1,3}\right)\d t+\left(\tA^{3}\left(\w_{n-1}\right)-\tA^{3}\left(\w_{n-2}\right)\right)\w_{n-1,3}\d t\\
&+\B\left(\bar{\w}\right) \left(\w_{n}-\w_{n-1}\right) \d t \notag\\
 =& -\tilde{\mathcal{L}}_{\u}\left(\sigma_{n-1}, \u_{n}\right)\d t+\tilde{\mathcal{L}}_{\u}\left(\sigma_{n-2}, \u_{n-1}\right)\d t + \left(\tilde{\mathcal{L}}_{\phi}\left(\sigma_{n-1}, \phi_{n}\right)-\tilde{\mathcal{L}}_{\phi}\left(\sigma_{n-2}, \phi_{n-1}\right)\right) \d t \notag\\
 &+ \left(\tilde{f}\left(\w_{n-1}\right)-\tilde{f}\left(\w_{n-2}\right)\right). \notag
\end{align}
Then we multiply the above formula with $ \left(\w_{n}-\w_{n-1}\right)$, the estimates of some terms
\begin{align}
&\tA^{1}\left(\w_{n-1}\right)\left(\w_{n,1}-\w_{n-1,1}\right)\d t+ \tA^{2}\left(\w_{n-1}\right)\left(\w_{n,2}-\w_{n-1,2}\right)\d t\\
&+\tA^{3}\left(\w_{n-1}\right)\left(\w_{n,3}-\w_{n-1,3}\right)\d t+\B\left(\bar{\w}\right) \left(\w_{n}-\w_{n-1}\right)\d t\notag
\end{align}
are similar to \eqref{1 energy est in existence to ls}, \eqref{2 energy est in existence to ls} and \eqref{3 energy est in existence to ls}, we omit it here. We focus on the estimates of
\begin{align}
&\sum \left(\tA^{i}\left(\w_{n-1}\right)-\tA^{i}\left(\w_{n-2}\right)\right)\w_{n-1,i}\d t,
\end{align}
and the right-hand side terms in \eqref{contraction system}.
By the expression formula of $\tA^{i}$, it holds
\begin{align}
&\int_{0}^{1} \sum \left(\tA^{i}\left(\w_{n-1}\right)-\tA^{i}\left(\w_{n-2}\right)\right)\w_{n-1,i}\cdot \left(\w_{n}-\w_{n-1}\right)\d x \d t \\
\ls &C \left\|\w_{n}-\w_{n-1}\right\|\left\|\w_{n-1}-\w_{n-2}\right\|\d t. \notag
\end{align}
Since
\begin{align}
 &-\tilde{\mathcal{L}}_{\u}\left(\sigma_{n-1}, \u_{n}\right) \d t + \tilde{\mathcal{L}}_{\u}\left(\sigma_{n-2}, \u_{n-1}\right) \d t\\
=&\left( \tilde{\mathcal{L}}_{\u}\left(\sigma_{n-1}, \u_{n-1}\right)-\tilde{\mathcal{L}}_{\u}\left(\sigma_{n-1}, \u_{n}\right)\right) \d t+ \left(\tilde{\mathcal{L}}_{\u}\left(\sigma_{n-2}, \u_{n-1}\right) -\tilde{\mathcal{L}}_{\u}\left(\sigma_{n-1}, \u_{n-1}\right) \right) \d t,\notag
\end{align}
we estimate
\begin{align}
&\int_{0}^{1}\left(-\tilde{\mathcal{L}}_{\u}\left(\sigma_{n-1}, \u_{n}\right) \d t + \tilde{\mathcal{L}}_{\u}\left(\sigma_{n-2}, \u_{n-1}\right)\right)\cdot \left(\w_{n}-\w_{n-1}\right)\d x \d t\notag\\
=&-\int_{0}^{1}\left(\bar{\rho}+\sigma_{n-1}\right)\left|\u_{n}-\u_{n-1}\right|^{2}\d x\d t -\int_{0}^{1}\left(\sigma_{n-1}-\sigma_{n-2}\right)\u_{n-1}\cdot\left(\u_{n}-\u_{n-1}\right)\d x \d t\\
\ls &-\int_{0}^{1}\frac{\bar{\rho}}{2}\left|\u_{n}-\u_{n-1}\right|^{2}\d x\d t-\int_{0}^{1}\left(\sigma_{n-1}-\sigma_{n-2}\right) \u_{n-1}\cdot\left(\u_{n}-\u_{n-1}\right) \d x \d t, \notag
\end{align}
where
\begin{align}
\int_{0}^{1}\left(\sigma_{n-1}-\sigma_{n-2}\right) \u_{n-1}\left(\u_{n}-\u_{n-1}\right) \d x \d t \ls C \left\|\w_{n}-\w_{n-1}\right\| \left\|\w_{n-1}-\w_{n-2}\right\|\d t.
\end{align}
Since
\begin{align}
&\left(\tilde{\mathcal{L}}_{\phi}\left(\sigma_{n-1}, \phi_{n}\right)-\tilde{\mathcal{L}}_{\phi}\left(\sigma_{n-2}, \phi_{n-1}\right)\right) \d t\notag\\
=&\left(\tilde{\mathcal{L}}_{\phi}\left(\sigma_{n-1}, \phi_{n}\right)-\tilde{\mathcal{L}}_{\phi}\left(\sigma_{n-1}, \phi_{n-1}\right)\right) \d t+\left(\tilde{\mathcal{L}}_{\phi}\left(\sigma_{n-1}, \phi_{n-1}\right)-\tilde{\mathcal{L}}_{\phi}\left(\sigma_{n-2}, \phi_{n-1}\right)\right) \d t,
\end{align}
then we have
\begin{align}
&\int_{0}^{1}\left(\tilde{\mathcal{L}}_{\phi}\left(\sigma_{n-1}, \phi_{n}\right)-\tilde{\mathcal{L}}_{\phi}\left(\sigma_{n-2}, \phi_{n-1}\right)\right) \cdot\left(\w_{n}-\w_{n-1}\right)\d x \d t \notag\\
=&\int_{0}^{1}\left(\tilde{\mathcal{L}}_{\phi}\left(\sigma_{n-1}, \phi_{n}\right) - \tilde{\mathcal{L}}_{\phi}\left(\sigma_{n-1}, \phi_{n-1}\right)\right)\cdot\left(\w_{n}-\w_{n-1}\right)\d x  \d t\\
&+\int_{0}^{1}\left(\tilde{\mathcal{L}}_{\phi}\left(\sigma_{n-1}, \phi_{n-1}\right)-\tilde{\mathcal{L}}_{\phi}\left(\sigma_{n-2}, \phi_{n-1}\right)\right)\cdot\left(\w_{n}-\w_{n-1}\right)\d x  \d t.\notag
\end{align}
By the continuity equation, there holds
\begin{align}
 & \int_{U} \tilde{\mathcal{L}}_{\phi}\left(\sigma_{n-1}, \phi_{n}-\phi_{n-1}\right)\cdot \left(\w_{n}-\w_{n-1}\right) \d x \d t \notag\\
=& \int_{U} \left(\bar{\rho}+\sigma_{n-1}\right)\nabla \left(\phi_{n}-\phi_{n-1}\right) \cdot \left(\u_{n}-\u_{n-1}\right) \d x \d t \notag \\
=& -\int_{U}\nabla \cdot \left(\left(\bar{\rho}+\sigma_{n-1}\right)\left(\u_{n}-\u_{n-1}\right)\right)\left(\phi_{n}-\phi_{n-1}\right) \d x \d t  \notag\\
=& \int_{U}\left(\sigma_{n}-\sigma_{n-1}\right)_{t}\left(\phi_{n}-\phi_{n-1}\right) \d x \d t+ \int_{U}\nabla \cdot \left(\left(\sigma_{n-1}-\sigma_{n-2}\right)\u_{n-1}\right)\left(\phi_{n}-\phi_{n-1}\right) \d x \d t \\
=& \int_{U}\left(\triangle \left(\phi_{n}-\phi_{n-1}\right) \right)_{t}\left(\phi_{n}-\phi_{n-1}\right) \d x \d t + \int_{U}\nabla \cdot \left(\left(\sigma_{n-1}-\sigma_{n-2}\right)\u_{n-1}\right)\left(\phi_{n}-\phi_{n-1}\right) \d x \d t \notag\\
=&-\d  \int_{U}\left|\nabla\left(\phi_{n}-\phi_{n-1}\right)\right|^{2} \d x -\int_{U}\left(\sigma_{n-1}-\sigma_{n-2}\right)\u_{n-1}\cdot\nabla\left(\phi_{n}-\phi_{n-1}\right)\d x \d t \notag\\
\ls &-\d  \int_{U}\left|\nabla\left(\phi_{n}-\phi_{n-1}\right)\right|^{2} \d x+C \left\|\w_{n}-\w_{n-1}\right\| \left\|\w_{n-1}-\w_{n-2}\right\|\d t ,\notag
\end{align}
and
\begin{align}
&\int_{0}^{1}\left(\tilde{\mathcal{L}}_{\phi}\left(\sigma_{n-1}, \phi_{n-1}\right)-\tilde{\mathcal{L}}_{\phi}\left(\sigma_{n-2}, \phi_{n-1}\right)\right)\cdot\left(\w_{n}-\w_{n-1}\right)\d x  \d t \notag\\
=&\int_{U}\left(\sigma_{n-1}-\sigma_{n-2}\right)\nabla \phi_{n-1}\cdot\left(\u_{n}-\u_{n-1}\right)\d x \d t\\
\ls & C \left\|\w_{n}-\w_{n-1}\right\|\left\|\w_{n-1}-\w_{n-2}\right\|\d t. \notag
\end{align}
For the terms in $\tilde{f}$, similarly, we have
\begin{align}
&\int_{0}^{1} \left(\left(\bar{\rho}+\sigma_{n-1}\right)\nabla h\left(\sigma_{n-1}\right)-\left(\bar{\rho}+\sigma_{n-2}\right)\nabla h\left(\sigma_{n-2}\right)\right) \cdot\left(\w_{n}-\w_{n-1}\right)\d x  \d t\\
\ls & C \left\|\w_{n}-\w_{n-1}\right\|\left\|\w_{n-1}-\w_{n-2}\right\|\d t;\notag
\end{align}
and
\begin{align}
  & \mathbb{E} \left[\left|\int_{0}^{t}\int_{0}^{1} \left(\mathbb{F}_{n-1}-\mathbb{F}_{n-2}\right) \d  W \cdot \left(\w_{n}-\w_{n-1}\right)\d x \right|^{m}\right]\notag\\
= & \mathbb{E} \left[\left|\int_{0}^{t}\int_{0}^{1} \left(\mathbb{F}_{n-1}-\mathbb{F}_{n-2}\right)  \cdot\left(\w_{n}-\w_{n-1}\right)\d x \d  W\right|^{m}\right] \notag\\
\ls & \mathbb{E} \left[\left|C\int_{0}^{t}\left|\int_{0}^{1} \left(\mathbb{F}_{n-1}-\mathbb{F}_{n-2}\right) \cdot\left(\w_{n}-\w_{n-1}\right)\d x\right|^{2}\d s\right|^{\frac{m}{2}}\right]\\
\ls &  \mathbb{E}\left[\left|C\int_{0}^{t}\left\|\w_{n}-\w_{n-1}\right\|^{2}\left\|\w_{n-1}-\w_{n-2}\right\|^{2}\d s \right|^{\frac{m}{2}}\right] \notag \\
\ls &  \mathbb{E}\left[\left(C\sup_{s\in[0,t]}\left\|\w_{n-1}-\w_{n-2}\right\|^{2}\int_{0}^{t}\left\|\w_{n}-\w_{n-1}\right\|^{2}\d s\right)^{\frac{m}{2}}\right]. \notag\\
\ls & \mathbb{E}\left[\left(\delta_{2}\sup_{s\in[0,t]}\left\|\w_{n-1}-\w_{n-2}\right\|^{2}\right)^{m}\right]+ \mathbb{E}\left[\left(\int_{0}^{t}C_{\delta_{2}}\left\|\w_{n}-\w_{n-1}\right\|^{2}\d s\right)^{m}\right].\notag
\end{align}
By It\^o's formula, we have
\begin{align}
&\d \left(\mathcal{D}\left(\sigma_{n-1}\right)\left(\w_{n}-\w_{n-1}\right)\cdot \left(\w_{n}-\w_{n-1}\right)\right)\notag\\
=& \d  \mathcal{D}\left(\sigma_{n-1}\right) \left(\w_{n}-\w_{n-1}\right)\cdot \left(\w_{n}-\w_{n-1}\right)+ 2\mathcal{D}\left(\sigma_{n-1}\right)\left(\w_{n}-\w_{n-1}\right)\cdot \d  \left(\w_{n}-\w_{n-1}\right) \\
&+ \mathcal{D}\left(\sigma_{n-1}\right) \left\langle \d  \left(\w_{n}-\w_{n-1}\right),\d  \left(\w_{n}-\w_{n-1}\right) \right\rangle, \notag
\end{align}
where $\left\langle \d  \left(\w_{n}-\w_{n-1}\right),\d  \left(\w_{n}-\w_{n-1}\right) \right\rangle$ is a shorthand for the more detailed expression for quadratic variation
$$\left\langle ~\left\langle \d  \left(\w_{n}-\w_{n-1}\right),\d  \left(\w_{n}-\w_{n-1}\right) \right\rangle, \left\langle \d  \left(\w_{n}-\w_{n-1}\right),\d  \left(\w_{n}-\w_{n-1}\right) \right\rangle~\right\rangle_{\mathcal{H}}^{\frac{1}{2}},$$
 with a slight abuse of notation.
Moreover, we have
\begin{align}
& \mathcal{D}\left(\sigma_{n-1}\right) \left\langle \d  \left(\w_{n}-\w_{n-1}\right),\d  \left(\w_{n}-\w_{n-1}\right) \right\rangle \\
=& \left(\bar{\rho}+\sigma_{n-1}\right)\left|\mathbb{F}_{n-1}-\mathbb{F}_{n-2}\right|^{2}\d t,\notag
\end{align}
and
\begin{align}
  & \int_{0}^{1} \mathcal{D} \left(\sigma_{n-1}\right) \left\langle \d  \left(\w_{n}-\w_{n-1}\right),\d  \left(\w_{n}-\w_{n-1}\right) \right\rangle \d x
\ls  C\left\|\w_{n-1}-\w_{n-2}\right\|^{2}\d t.
\end{align}
Hence it holds
\begin{align}
\int_{0}^{1} \d \left(\mathcal{D}\left(\sigma_{n-1}\right)\left(\w_{n}-\w_{n-1}\right)\cdot \left(\w_{n}-\w_{n-1}\right)\right) \d x \ls C \left\|\w_{n}-\w_{n-1}\right\|^{2}\d t.
\end{align}
Combining the above estimates, for some $m\geqslant 2$, we have
\begin{align}
&\mathbb{E} \left[\left(\sup\limits_{s\in[0,t]}\int_{0}^{s} \d  \left\|\w_{n}-\w_{n-1}\right\|^{2} \right)^{m}\right]\\
\ls &\mathbb{E}\left[\left|\int_{0}^{t} C \left(\left\|\w_{n}-\w_{n-1}\right\|^{2}+\left\|\w_{n-1}-\w_{n-2}\right\|^{2}+ \left\|\w_{n}-\w_{n-1}\right\|\left\|\w_{n-1}-\w_{n-2}\right\|\right) \d s \right|^{m}\right] \notag\\
&+  \mathbb{E}\left[\left(\delta_{2}\sup_{s\in[0,t]}\left\|\w_{n-1}-\w_{n-2}\right\|^{2}\right)^{m}\right]+ \mathbb{E}\left[\left(\int_{0}^{t}C_{\delta_{2}}\left\|\w_{n}-\w_{n-1}\right\|^{2}\d s\right)^{m}\right],\notag
\end{align}
where $C$ depends on $M$.
By Cauchy's inequality and Jensen's inequality, we have
\begin{align}
&\mathbb{E} \left[ \left(\sup_{s\in[0,t]} \left\|\w_{n}-\w_{n-1}\right\|^{2}\right)^{m}\right]\notag\\
\ls &\mathbb{E} \left[ \left(\int_{0}^{t} C \left(\left\|\w_{n}-\w_{n-1}\right\|^{2}+\left\|\w_{n-1}-\w_{n-2}\right\|^{2}+ \left\|\w_{n}-\w_{n-1}\right\|\left\|\w_{n-1}-\w_{n-2}\right\|\right) \d s \right)^{m}\right]\notag\\
\ls & \mathbb{E} \left[ \left(\int_{0}^{t} C \left(\left\|\w_{n}-\w_{n-1}\right\|^{2}+\left\|\w_{n-1}-\w_{n-2}\right\|^{2}\right) \d s\right)^{m}\right]\\
\ls & \int_{0}^{t}\left(\mathbb{E} \left[\left(C_{0} \left\|\w_{n}-\w_{n-1}\right\|^{2}\right)^{m}\right]+ \mathbb{E} \left[\left(C_{0}\left\|\w_{n-1}-\w_{n-2}\right\|^{2}\right)^{m}\right]\right)\d s.\notag
\end{align}
The higher order contraction estimates are proved similarly to zeroth-order, with the same symmetrizing matrix and the important insulating boundary condition, and the detailed proof is omitted here.
In summary, we have
\begin{align}
&\mathbb{E} \left[ \left(\sup_{s\in[0,t]} \left\|\w_{n}-\w_{n-1}\right\|_{3}^{2}\right)^{m}\right] \\
\ls & \int_{0}^{t}\left(\mathbb{E} \left[\left(C_{0} \left\|\w_{n}-\w_{n-1}\right\|_{3}^{2}\right)^{m}\right]+ \mathbb{E}
\left[\left(C_{0}\left\|\w_{n-1}-\w_{n-2}\right\|_{3}^{2}\right)^{m}\right]\right)\d s. \notag
\end{align}
By Gr\"onwall's inequality, we have
\begin{align}
&\mathbb{E} \left[ \left( \sup_{s\in[0,t]} \left\|\w_{n}-\w_{n-1}\right\|_{3}^{2}\right)^{m}\right]\\
\ls & \mathbb{E} \left[ \left(\sup_{s\in[0,t]}\left\|\w_{n-1}-\w_{n-2}\right\|_{3}^{2}\right)^{m}\right]C_{0}^{m}t+\int_{0}^{t}\mathbb{E} \left[ \left(\sup_{s\in[0,\tau]}\left\|\w_{n-1}-\w_{n-2}\right\|_{3}^{2} \right)^{m}\right] \tau C_{0}^{2r} e^{C_{0}^{m}\tau} \d \tau \notag\\
\ls & 3 C_{0}^{m} \mathbb{E} \left[ \left(\sup_{s\in[0,\tau]}\left\|\w_{n-1}-\w_{n-2}\right\|_{3}^{2} \right)^{m}\right]t. \notag
\end{align}
Let $T_{1}\ls T_{0}$ and $3C_{0}^{m}T_{1}<1$, $e^{C_{0}^{m}T_{1}}\ls 2 $, then
\begin{align}\label{Contraction conclusion}
\mathbb{E} \left[ \left(\sup\limits_{s\in[0,t]} \left\|\w_{n}-\w_{n-1}\right\|_{3}\right)^{m}\right]
 \ls a ~ \mathbb{E} \left[\left( \sup\limits_{s\in[0,t]}\left\|\w_{n-1}-\w_{n-2}\right\|_{3}\right)^{m}\right], \quad a < 1,
 \end{align}
where $a=3 C_{0}^{m}T_{1}$ with $C_{0}$ depending on the initial data by the onto mapping estimates. Hence, $\w_{n}$ is a Cauchy sequence. 
By Banach's fixed point theorem, there exists a unique solution $\w$ in $L^{2m}\left(\Omega; C\left([0,T_{1}]; H^{3}\left(U\right)\right)\right)$. Since $\triangle \phi=\sigma$ holds, $\phi$ is also a unique solution in $L^{2m}\left(\Omega; C\left([0,T_{1}]; H^{5}\left(U\right)\right)\right)$ up to a constant, with the boundary condition $\nabla \phi\cdot \nu=0$.

By the proof of theorem 5.2.9 in \cite{Karatzas1988}, $(\rho,\u,\Phi)$ is the unique strong solutions to SEP, where $\rho,\u\in C\left([0,T_{1}]; H^{3}\left(U\right)\right)$ and $\Phi \in C\left([0,T_{1}]; H^{5}\left(U\right)\right)$ hold $\mathbb{P}$ a.s. We give the definition of the local strong solution as follows.
\begin{definition}
Let $\left(\Omega,\mathcal{F},\mathbb{P} \right)$ be a fixed stochastic basis with a complete right-continuous filtration $\mathcal{F}=\left(\mathcal{F}_{s}\right)_{s\geqslant 0}$ and $W$ be the fixed Wiener process. $\left(\rho, \u, \Phi\right)$ is called a strong solution to initial and boundary problem \eqref{3-D Euler Poisson}-\eqref{insulated boundary condition}-\eqref{initial conditions}-\eqref{Initial formula of phi}-\eqref{condition for F for 3-d}, if:
\begin{enumerate}
  \item $\left(\rho, \u, \Phi\right)$ is adapted to the filtration $\left(\mathcal{F}_{s}\right)_{s\geqslant 0}$;
  \item $\mathbb{P}[\{\left(\rho(0), \u(0), \Phi(0)\right)=\left(\rho_{0}, \u_{0}, \Phi_{0}\right) \}]=1$;
  \item the equation of continuity
\begin{align}
   \rho(t)
  =  \rho_{0}- \int_0^t \nabla\cdot(\rho\u)  \d s,\notag
\end{align}
  holds $\mathbb{P}$ {\rm a.s.}, for any $t\in [0, T_{1}]$;
  \item the momentum equation
\begin{align}
\u(t)=& \u_{0}  -\int_{0}^{t}(\u\cdot\nabla)\u  \d s -\int_{0}^{t} \frac{\nabla  P(\rho)}{\rho} \d s +\int_{0}^{t} \nabla \Phi  \d s  - \int_{0}^{t} \u \d s\\
 & +\int_{0}^{t}  \frac{\mathbb{F}(\rho,\u)}{\rho}\d W(s) ,\notag
\end{align}
 holds $\mathbb{P}$ {\rm a.s.}, for any $t\in [0, T_{1}]$;
 \item the electrostatic potential equation
 \begin{align}
\triangle\Phi  =\rho-b,
\end{align}
 holds $\mathbb{P}$ {\rm a.s.} for any $t\in [0, T_{1}]$.
\end{enumerate}
\end{definition}

\begin{remark}
Reviewing the above proof, \eqref{Contraction conclusion} holds for general stochastic forces without \eqref{2 Condition for stochastic force}. Thus, the local existence also holds.
\end{remark}
\smallskip
\begin{flushleft}
\textbf{Step 4: Energy estimates.}
\end{flushleft}
\subsection{Estimates up to third-order}\label{Estimates up to third-order subtitle}

In this subsection, we begin by symmetrizing the system. Then, we proceed with energy estimates up to third order, taking stochastic forces under the condition \eqref{2 Condition for stochastic force} for instance. 

\subsubsection{Zero-order estimates}

For the system \eqref{sto Euler-Poisson system},
we define the energy
\begin{align}
\mathcal{E}\left(t\right)=\int_{U}\frac{1}{2}\left(\bar{\rho}\left|\u\right|^{2}+Q'\left(\bar{\rho}\right)\sigma^{2}+\left|\nabla\phi\right|^{2} \right)\d x .
\end{align}
By It\^o's formula, we have
\begin{align}
&\d  \int_{U} \frac{1}{2}\left(\bar{\rho}+\sigma\right)\left|\u\right|^{2}\d x \\
= & \int_{U} \frac{1}{2}\d \left(\bar{\rho}+\sigma\right)\left|\u\right|^{2}\d x
+ \int_{U} \left(\bar{\rho}+\sigma\right)\u\cdot \d \u \d x +  \int_{U} \frac{1}{2}\left(\bar{\rho}+\sigma\right)\left|\mathbb{F}\right|^{2}\d t\d x. \notag
\end{align}
But here we will deal with $\bar{\rho}+\sigma$ and $\u$ together by considering the symmetrized system of $\w$. By It\^o's formula, it holds
\begin{align}
\int_{U}\d \left( \frac{1}{2} \mathcal{D} \w \cdot \w \right) \d x = \int_{U}\frac{1}{2} \w \left(\d \mathcal{D}\right)\w\d x + \int_{U} \mathcal{D} \d  \w\cdot \w \d x+ \int_{U} \frac{1}{2} \mathcal{D} \mathbb{F}\cdot \mathbb{F}\d x \d t,
\end{align}
which is $\d  \int_{U}\frac{1}{2}\left(\bar{\rho}\left|\u\right|^{2}+Q'\left(\bar{\rho}\right)\sigma^{2}\right)\d x$.
Over the domain $U$, we integrate
\begin{align}
&\mathcal{D} \d  \w\cdot \w+\left(\tilde{\mathcal{A}}^{1}\w_{,1} \cdot \w + \tilde{\mathcal{A}}^{2}\w_{,2}\cdot \w+\tilde{\mathcal{A}}^{3}\w_{,3}\cdot \w+ \B \w\cdot \w + \tilde{\mathcal{L}}_{\u}\cdot \w \right)\d t\\
=& \tilde{\mathcal{L}}_{\phi}\cdot \w\d t+ \tilde{f}\cdot \w,\notag
\end{align}
then we have
\begin{align}\label{equation of D w w}
 & \int_{U}  \d  \left(\frac{1}{2}\mathcal{D}\w\cdot \w \right) \d x =\int_{U}\frac{1}{2} \w \left(\d \mathcal{D}\right)\w\d x + \int_{U} \mathcal{D} \d  \w \cdot \w \d x+ \int_{U}\frac{1}{2}\mathcal{D}\mathbb{F}\cdot\mathbb{F} \d x \d t\notag\\
=& \int_{U}\frac{1}{2} \w \left(\d \mathcal{D}\right)\w\d x -\int_{U}\left(\tilde{\mathcal{A}}^{1}\w_{,1}\cdot \w + \tilde{\mathcal{A}}^{2}\w_{,2}\cdot \w+ \tilde{\mathcal{A}}^{3}\w_{,3}\cdot \w +\B \w\cdot \w\right) \d x \d t \\
 & -\int_{U} \tilde{\mathcal{L}}_{\u}\cdot \w \d x \d t  +\int_{U} \tilde{\mathcal{L}}_{\phi}\cdot \w \d x \d t +\int_{U} \nabla h\left(\sigma\right)\cdot \w \d x \d t \notag\\
 &  + \int_{U}\mathcal{D} \mathbb{F}\d  W\cdot \w\d x + \int_{U} \mathcal{D}\mathbb{F}\cdot\mathbb{F}\d x \d t. \notag
\end{align}
 Direct calculation shows that
 $$-\frac{1}{2}\tilde{\mathcal{A}}^{i}_{,i}+\B
 -\operatorname{diag}\left[-\frac{1}{2}\left(u^{i}Q'\left(\rho\right)\right)_{,i},-\frac{1}{2}\left(u^{i}\rho\right)_{,i},
 -\frac{1}{2}\left(u^{i}\rho\right)_{,i},-\frac{1}{2}\left(u^{i}\rho\right)_{,i}\right]$$
  is anti-symmetric \cite{Guo2006StabilityOS}. Hence, we have
\begin{align}
 &\int_{U} \left( \tilde{\mathcal{A}}^{1}\w_{,1}\cdot \w + \tilde{\mathcal{A}}^{2}\w_{,2}\cdot \w+ \tilde{\mathcal{A}}^{3}\w_{,3}\cdot \w + \B \w\cdot \w \right) \d x \d t \notag\\
=&\int_{U} \left(-\frac{1}{2}\left(\w \tilde{\mathcal{A}}^{1}_{,1}\w +\w \tilde{\mathcal{A}}^{2}_{,2}\w + \w \tilde{\mathcal{A}}^{3}_{,3}\w \right)+ \B \left|\w\right|^{2}\right) \d x \d t + \int_{U} \left(\w \tilde{\mathcal{A}}^{j}\w \right)_{,j} \d x \d t. \notag
\end{align}
On account of the insulated boundary condition $\u\cdot \nu|_{\partial U} =0$, it holds
\begin{align}
\int_{U} \left(\w \tilde{\mathcal{A}}^{j}\w \right)_{,j} \d x = \int_{\partial U} \left(\left(\u\cdot \nu\right)\left(Q'\left(\bar{\rho}\right)\sigma^{2}+\rho\left|\u\right|^{2}+2\rho Q'\left(\bar{\rho}\right)\right)\right)\d S \equiv 0.
\end{align}
Hence, it holds
\begin{align}
\int_{U} \left(-\frac{1}{2}\left(\w \tilde{\mathcal{A}}^{1}_{1}\w +\w \tilde{\mathcal{A}}^{2}_{2}\w + \w \tilde{\mathcal{A}}^{3}_{3}\w \right)+ \B\left|\w\right|^{2}\right) \d x \d t \ls C \left\|\w\right\|_{3}^{3}\d t.
\end{align}
Recalling \eqref{Lu and Lphi}, there hold
\begin{align}
\int_{U} \tilde{\mathcal{L}}_{\u}\cdot \w \d x \d t= \int_{U} \left(\bar{\rho}+\sigma\right)\left|\u\right|^{2}\d x \d t \geqslant \int_{U} C\bar{\rho}\left|\u\right|^{2}\d x \d t;
\end{align}
\begin{align}
\int_{U} \tilde{\mathcal{L}}_{\phi}\cdot \w \d x \d t= & \int_{U} \left(\bar{\rho}+\sigma\right)\nabla \phi \cdot \u \d x \d t = -\int_{U}\nabla \cdot \left(\left(\bar{\rho}+\sigma\right)\u\right)\phi \d x \d t \\
=& \int_{U}\sigma_{t}\phi \d x \d t= \int_{U}\left(\triangle \phi \right)_{t}\phi \d x \d t=-\d  \int_{U}\left|\nabla\phi\right|^{2} \d x.\notag
\end{align}
For the stochastic term, it holds
\begin{align}
\int_{U} \tilde{f}\cdot \w \d x = \int_{U}\left(O\left(\sigma^{2}\right)\d t-\mathbb{F}\d  W\right)\cdot\u\d x
\ls \left\|\w\right\|_{3}^{3}\d t+ \left|\int_{U} \mathbb{F}\d  W\cdot\u\d x \right|.
\end{align}
For $\left|\mathbb{F}\right|\ls C \left|\rho\u\right|^{2}$, we estimate
\begin{align}
    &\mathbb{E}\left[\left|\int_{0}^{t} \int_{U} \mathbb{F}\cdot\u\d x \d  W \right|^{m}\right]\ls \mathbb{E}\left[\left(\int_{0}^{t} \left|C\int_{U} \mathbb{F}\cdot\u\d x\right|^{2}\d s\right)^{\frac{m}{2}}\right]\notag\\
\ls &\mathbb{E}\left[\left(C\int_{0}^{t} \left|\int_{U} \left|\bar{\rho}\u\right|^{2}\left|\u\right|^{2}\cdot\u\d x\right|^{2}\d s\right)^{\frac{m}{2}}\right]\ls \mathbb{E}\left[\left(C\sup\limits_{s\in [0,t]} \left\|\u \right\|^{2} \int_{0}^{t} \left\| \u \right\|_{3}^{4} \d s\right)^{\frac{m}{2}}\right]\\
\ls &\delta_{3}^{m}\mathbb{E}\left[\left(\sup\limits_{s\in [0,t]} \left\|\u \right\|^{2}\right)^{m}\right] +C_{\delta_{3}}^{m}\mathbb{E}\left[\left(\int_{0}^{t} \left\| \u \right\|_{3}^{3} \d s\right)^{m}\right], \notag
\end{align}
where $\delta_{3}$ is taken such that $\delta_{3}^{m}\mathbb{E}\left[\left(\sup\limits_{s\in [0,t]} \left\|\u \right\|^{2}\right)^{m}\right]$ can be balanced by the left side by the time continuity of solutions. Similarly, it holds
\begin{align}
\int_{U} \frac{1}{2} \mathcal{D} \mathbb{F}\cdot \mathbb{F}\d x \d t \ls C \left\|\w\right\|_{3}^{3} \d t.
\end{align}
Besides, there holds
\begin{align}
 & \int_{U} \w \left(\d \mathcal{D}\right)\w\d x \notag\\
=& \int_{U}\w\left(\operatorname{diag}\left\{Q'\left(\bar{\rho}+\sigma\right)_{t},\left(\bar{\rho}+\sigma\right)_{t} ,\left(\bar{\rho}+\sigma\right)_{t},\left(\bar{\rho}+\sigma\right)_{t}\right\}\right)\w \d x \d t \notag\\
=& \int_{U} \left(Q''\left(\bar{\rho}+\sigma\right)\sigma_{t}\sigma^{2}+\sigma_{t}\left|\u\right|^{2}\right)\d x \d t\\
=&  \int_{U} \left(Q''\left(\bar{\rho}\right)+O\left(\sigma\right)\right)\left(-\nabla\cdot\left(\left(\bar{\rho}+\sigma\right)\u\right)\right)\sigma^{2}\d x \d t+ \int_{U} \left(-\nabla\cdot\left(\left(\bar{\rho}+\sigma\right)\u\right)\right)\left|\u\right|^{2}\d x\d t \notag \\
\ls & C \left\|\w\right\|_{3}^{3}\d t.\notag
\end{align}
In conclusion, as $\bar{\rho}$ have a positive lower bound, we have
\begin{align}\label{zero order estimate of w 3-D}
&\mathbb{E}\left[\left(\sup\limits_{s\in[0,t]}\int_{0}^{s} \d \left(\int_{U}\left|\w\right|^{2}\d x + \int_{U}\left|\nabla\phi\right|^{2}\d x\right) \d s\right)^{m}\right]+\mathbb{E}\left[\left(c_{2}\int_{0}^{t} \int_{U}\left|\u\right|^{2}\d x \d s \right)^{m}\right]\notag \\
\ls  & \mathbb{E}\left[\left(C \int_{0}^{t}  \left\|\w\right\|_{3}^{3}\d s\right)^{m}\right].
\end{align}
Next, we give the estimates of $\int_{0}^{t}\int_{U} \left\|\sigma\right\|^{2} \d x \d s$.
From the velocity equation \eqref{deformed sto Euler-Poisson}, we have
\begin{align}\label{zero order plus estimate eq}
\left(\nabla Q\left(\bar{\rho}+\sigma\right) -\nabla Q\left(\bar{\rho}\right)\right)\d t = -\d  \u -\left(\left(\u \cdot\nabla \right)\u-\u\right)\d t + \nabla \phi \d t + \frac{\mathbb{F}}{\bar{\rho}+\sigma}\d  W,
\end{align}
with
\begin{align}\label{Q(rho)-Q(bar of rho)}
 &\nabla \left(Q\left(\bar{\rho}+\sigma\right)-Q\left(\bar{\rho}\right)\right)
= Q'\left(\bar{\rho}+\sigma\right)\nabla \sigma+Q''\left(\bar{\rho}\right)\sigma \nabla\bar{\rho}+ \mathbf{h},
\end{align}
where
\begin{align}
h_{i}=O\left(\sigma^{2}\right).
\end{align}
We multiply the equation \eqref{zero order plus estimate eq} with $\left(\sigma, \sigma, \sigma\right)^{T}$. By the integration by parts and the insulating boundary condition, due to the condition that $\left|\nabla \bar{\rho}\right|>0$, the left side is
\begin{align}
\int_{U} \left|Q''\left(\bar{\rho}\right)\nabla\bar{\rho}\right| \left|\sigma\right|^{2}\d x + \int_{U} O \left(\sigma^{3}\right)\d x.
\end{align}
By It\^o's formula, there holds
\begin{align}
\left(\d u^{i}\right) \sigma =\d \left(u^{i} \sigma\right) - u^{i} \d \sigma,
\end{align}
where
\begin{align}
-\int_{0}^{t} \d \int_{U} \left(u^{i} \sigma\right) \d x \ls \int_{0}^{t} \d  \left(\frac{1}{2}\left\|\sigma\right\|^{2} + \frac{1}{2}\left\|u^{i}\right\|^{2}\right) .
\end{align}
By the continuity equation, it holds
\begin{align}
\int_{0}^{t}\int_{U} \left|u^{i} \d \sigma\right|\d x \ls C\int_{0}^{t} \left\|\w\right\|_{3}^{3} \d s.
\end{align}
For $-\u\d t$, we directly estimate
\begin{align}
\int_{0}^{t}\int_{U}\left| -u^{i}\sigma\right| \d x \d s \ls \frac{\delta_{4}}{2}\int_{0}^{t}\left\|\sigma\right\|^{2}\d s +C_{\delta_{4}}\int_{0}^{t}\left\|u^{i}\right\|^{2} \d s,
\end{align}
where $\delta_{4}$ is small such that $ \delta_{4}\int_{0}^{t}\left\|\sigma\right\|^{2}\d s$ can be balanced by the left side. For the term $\nabla \phi \d t$ in \eqref{zero order plus estimate eq}, we estimate
\begin{align}
\int_{0}^{t}\int_{U}\left|-\phi_{,i}\sigma\right|\d x\d s \ls \frac{\delta_{4}}{2}\int_{0}^{t} \left\|\sigma\right\|^{2}\d s +C_{\delta_{4}} \sup_{s\in [0,t]}\left\|\phi_{,i}\right\|^{2}.
\end{align}
For the stochastic term, since $\left|\mathbb{F}\right|\ls C\left|\rho\u\right|^{2}$, we estimate
\begin{align}
&\mathbb{E}\left[\left|\int_{0}^{t}\int_{U}\frac{\mathbb{F}^{i}}{\bar{\rho}+\sigma}\d  W \sigma \d x\right|^{m}\right]
\ls  \mathbb{E}\left[\left|C\int_{0}^{t}\left|\int_{U}\frac{\mathbb{F}^{i}}{\bar{\rho}+\sigma} \sigma \d x\right|^{2} \d s\right|^{\frac{m}{2}}\right]\notag\\
\ls & \mathbb{E}\left[\left|C\int_{0}^{t}\left|\int_{U}\left|\bar{\rho}+\sigma\right|\left|\u\right|^{2} \sigma \d x\right|^{2} \d s\right|^{\frac{m}{2}}\right]
\ls  \mathbb{E}\left[\left|C\int_{0}^{t}\left\|\u\right\|^{2}\left\|\bar{\rho}\sigma+\sigma^{2}\right\|_{\infty}^{2}\d s\right|^{\frac{m}{2}}\right]\\
\ls &  \mathbb{E}\left[\left(\frac{1}{4}\sup_{s\in [0,t]}\left\|\u\right\|^{2}\right)^{m}\right] + \mathbb{E}\left[\left(C \int_{0}^{t}\left\|\u\right\|^{2}\left\|\sigma\right\|_{\infty}^{2}\d s\right)^{m}\right]\notag\\
&+ \mathbb{E}\left[\left(\frac{1}{4}\sup_{s\in [0,t]} \left\|\u\right\|^{2}\right)^{m}\right] + \mathbb{E}\left[\left(C \int_{0}^{t}\left\|\u\right\|^{2}\left\|\sigma\right\|_{\infty}^{4}\d s\right)^{m}\right] \notag\\
\ls &  \mathbb{E}\left[\left(\frac{1}{2}\sup_{s\in [0,t]} \left\|\u\right\|^{2}\right)^{m}\right]+\mathbb{E}\left[\left(C\int_{0}^{t} \left\|\w\right\|_{3}^{3}\d s\right)^{m}\right]. \notag
\end{align}
Therefore, we have
\begin{align}\label{zero order plus estimate}
&\mathbb{E}\left[\left(\int_{0}^{t} \left\|\sigma\right\|^{2}\d s\right)^{m}\right]\notag \\
\ls & \mathbb{E}\left[\left(\int_{0}^{t} \d  \left(\frac{1}{2}\left\|\sigma\right\|^{2} + \frac{1}{2}\left\|\u\right\|^{2}\right)\right)^{m}\right]+\mathbb{E}\left[\left(\frac{1}{2}\sup_{s\in [0,t]} \left\|\u\right\|^{2}\right)^{m}\right]\\
& + \mathbb{E}\left[\left( C \sup_{s\in [0,t]}\left\|\nabla\phi\right\|^{2}\right)^{m}\right]+  \mathbb{E}\left[\left(C\int_{0}^{t}\left\|\u\right\|^{2} \d s\right)^{m}\right] +  \mathbb{E}\left[\left(C\int_{0}^{t} \left\|\w\right\|_{3}^{3} \d s\right)^{m}\right]. \notag
\end{align}
Furthermore, we can give the estimate of $\mathbb{E}\left[\left(\int_{0}^{t} \left\|\nabla \phi\right\|^{2}\d s\right)^{m}\right]$. We multiply \eqref{zero order plus estimate eq} with $\nabla \phi$ and integrate it over $U$, then we have
\begin{align}
&\int_{U} \left|\nabla \phi\right|^{2}\d x \d t \notag \\
 = &-\int_{U} \left(\nabla Q\left(\bar{\rho}+\sigma\right) -\nabla Q\left(\bar{\rho}\right)\right)\cdot \nabla \phi\d x \d t +\int_{U} \d  \u \cdot \nabla \phi\d x  \\
 &+\int_{U}\left(\left(\u \cdot\nabla \right)\u-\u\right)\cdot \nabla \phi\d x\d t - \int_{U} \frac{\mathbb{F}}{\bar{\rho}+\sigma}\d  W\cdot \nabla \phi\d x. \notag
\end{align}

From \eqref{Q(rho)-Q(bar of rho)}, by integration by parts and $\triangle \phi =\sigma $, we estimate
\begin{align}
&\int_{0}^{t} \int_{U} \left(\nabla Q\left(\bar{\rho}+\sigma\right) - \nabla Q\left(\bar{\rho}\right) \right) \cdot \nabla \phi\d x \d s \\
\ls & C\left( \sup\limits_{s\in [0,t]}\left\|\nabla \phi\right\|^{2} + \int_{0}^{t}\left\|\sigma\right\|^{2}\d s + \int_{0}^{t}\left\|\sigma\right\|_{3}^{3}\d s \right).\notag
\end{align}
By It\^o's formula, there holds
\begin{align}
\left(\d \u\right)\nabla \phi =\d \left(\u \nabla \phi\right) - \u \d  \nabla \phi ,
\end{align}
where
\begin{align}
-\int_{0}^{t} \d \int_{U} \left(\u\nabla \phi \right) \d x \ls \int_{0}^{t} \d  \left(\frac{1}{2}\left\|\nabla \phi\right\|^{2} + \frac{1}{2}\left\|\u\right\|^{2}\right).
\end{align}
By the continuity equation, it holds
\begin{align}
&\int_{0}^{t}\int_{U} \left|\u \d  \nabla \phi\right|\d x = \int_{0}^{t}\int_{U} \left|\u \d  \nabla \triangle^{-1}\sigma \right|\d x\\
 = &\int_{0}^{t}\int_{U} \left|\u\nabla \triangle^{-1}\d   \sigma \right|\d x \ls C\int_{0}^{t} \left\|\w\right\|_{3}^{3} \d s.\notag
\end{align}
It is clear that
\begin{align}
\int_{0}^{t}\int_{U}\left(\left(\u \cdot\nabla \right)\u\right)\cdot \nabla \phi\d x\d t \ls C\int_{0}^{t} \left\|\w\right\|_{3}^{3} \d s.
\end{align}
For $-\u\d t$, we directly estimate
\begin{align}
\int_{0}^{t}\int_{U}\left| -\u \cdot \nabla \phi\right| \d x \d s \ls \frac{1}{2}\int_{0}^{t}\left\| \nabla \phi\right\|^{2}\d s + \frac{1}{2}\int_{0}^{t}\left\|\u\right\|^{2} \d s.
\end{align}
For the stochastic term, since $\left|\mathbb{F}\right|\ls C\left|\rho\u\right|^{2}$ and $\triangle \phi= \sigma$, we estimate
\begin{align}
&\mathbb{E}\left[\left|\int_{0}^{t}\int_{U}\frac{\mathbb{F}}{\bar{\rho}+\sigma}\d  W \cdot \nabla \phi \d x\right|^{m}\right]
\ls  \mathbb{E}\left[\left|C\int_{0}^{t}\left|\int_{U}\frac{\mathbb{F}}{\bar{\rho}+\sigma}\cdot \nabla \phi\d x\right|^{2} \d s\right|^{\frac{m}{2}}\right]\notag\\
\ls & \mathbb{E}\left[\left|C\int_{0}^{t}\left|\int_{U}\left|\bar{\rho}+\sigma\right|\left|\u\right|^{2}  \left|\nabla \phi\right| \d x\right|^{2} \d s\right|^{\frac{m}{2}}\right]
\ls  \mathbb{E}\left[\left|C\int_{0}^{t}\left\|\u\right\|^{2}\left\|\bar{\rho}+\sigma\right\|_{\infty}^{2}\left\|\nabla \phi\right\|_{\infty}^{2}\d s\right|^{\frac{m}{2}}\right]\\
\ls &  \mathbb{E}\left[\left(\frac{1}{4}\sup_{s\in [0,t]}\left\|\u\right\|^{2}\right)^{m}\right] +  \mathbb{E}\left[\left(C\int_{0}^{t}\left\|\u\right\|^{2}\left\|\sigma\right\|_{1}^{2}\d s\right)^{m}\right]\notag\\
&+ \mathbb{E}\left[\left(\frac{1}{4}\sup_{s\in [0,t]} \left\|\u\right\|^{2}\right)^{m}\right] + \mathbb{E}\left[\left(C \int_{0}^{t}\left\|\u\right\|^{2}\left\|\sigma\right\|_{\infty}^{2}\left\|\sigma\right\|_{1}^{2}\d s\right)^{m}\right] \notag\\
\ls &  \mathbb{E}\left[\left(\frac{1}{2}\sup_{s\in [0,t]} \left\|\u\right\|^{2}\right)^{m}\right]+\mathbb{E}\left[\left(C\int_{0}^{t} \left\|\w\right\|_{3}^{3}\d s\right)^{m}\right]. \notag
\end{align}
Therefore, we have
\begin{align}\label{zero order plus estimate}
&\mathbb{E}\left[\left(\int_{0}^{t} \left\|\nabla \phi\right\|^{2}\d s\right)^{m}\right]\notag \\
\ls & \mathbb{E}\left[\left(\int_{0}^{t} \d  \left(\frac{1}{2}\left\|\nabla \phi\right\|^{2} + \frac{1}{2}\left\|\u\right\|^{2}\right)\right)^{m}\right]+\mathbb{E}\left[\left(\frac{1}{2}\sup_{s\in [0,t]} \left\|\u\right\|^{2}\right)^{m}\right]\\
& + \mathbb{E}\left[\left( C \sup_{s\in [0,t]}\left\|\nabla\phi\right\|^{2}\right)^{m}\right]+  \mathbb{E}\left[\left(C\int_{0}^{t}\left\|\w\right\|^{2} \d s\right)^{m}\right] +  \mathbb{E}\left[\left(C\int_{0}^{t} \left\|\w\right\|_{3}^{3} \d s\right)^{m}\right]. \notag
\end{align}
Multiplying a small constant to it, and we plus the zero-order estimates \eqref{zero order estimate of w 3-D} such that
\begin{align}
&\mathbb{E}\left[\left(\sup\limits_{s\in[0,t]}\int_{0}^{s} \d  \left(\frac{1}{2}\left\|\sigma\right\|^{2} + \left\|\u\right\|^{2}+\frac{1}{2}\left\|\nabla \phi\right\|^{2}\right)\right)^{m}\right]+\mathbb{E}\left[\left(\frac{1}{2}\sup_{s\in [0,t]} \left\|\u\right\|^{2}\right)^{m}\right]\\
&+\mathbb{E}\left[\left( C \sup_{s\in [0,t]}\left\|\nabla\phi\right\|^{2}\right)^{m}\right]+  \mathbb{E}\left[\left(C\int_{0}^{t}\left\|\w\right\|^{2} \d s\right)^{m}\right]\notag
\end{align}
 can be balanced by \eqref{zero order estimate of w 3-D}. Then we obtain
\begin{align}
&\mathbb{E}\left[\left( \sup\limits_{s\in[0,t]} \int_{0}^{s} \d \left(\left\|\w\right\|^{2} + \left\|\nabla\phi\right\|^{2}\right)+\zeta_{5}\left(\int_{0}^{t} \left\|\w\right\|^{2}+ \left\|\nabla\phi\right\|^{2}\right)\d s \right)^{m}\right] \\
\ls  & \mathbb{E}\left[\left(C\int_{0}^{t} \left\|\w\right\|_{3}^{3} \d s\right)^{m}\right]. \notag
\end{align}
where $C$ depends on $m$.
\subsubsection{First order estimates}
Taking derivative to \eqref{symmetrized system}, we have
\begin{align}\label{unsimplified 1 order derivative of symmetrized system}
&\nabla\left(\mathcal{D} \d  \w\right) +\nabla\left(\left(\tilde{\mathcal{A}}^{1}\w_{,1}+\tilde{\mathcal{A}}^{2}\w_{,2}+\tilde{\mathcal{A}}^{3}\w_{,3}\right)+\nabla\left(\B\w\right)+\nabla\left(\tilde{\mathcal{L}}_{\u}\right)\right)\d t\\
& = \nabla\left(\tilde{\mathcal{L}}_{\phi}\right)\d t + \nabla \tilde{f}.\notag
\end{align}
Recalling $\mathcal{D}=\operatorname{diag}\left[Q'\left(\bar{\rho}+\sigma\right),\bar{\rho}+\sigma,\bar{\rho}+\sigma,\bar{\rho}+\sigma\right]$, we calculate
\begin{align}
\nabla\left(\mathcal{D} \d  \w\right)=&\partial_{j}\left(\begin{array}{c}Q'\left(\bar{\rho}+\sigma\right)\sigma_{t} \\ \left(\bar{\rho}+\sigma\right) \d u^1 \\  \left(\bar{\rho}+\sigma\right) \d u^2 \\ \left(\bar{\rho}+\sigma\right) \d u^3 \end{array}\right)=\left(\begin{array}{c}\partial_{j}Q'\left(\bar{\rho}+\sigma\right)\sigma_{t}+Q'\left(\bar{\rho}+\sigma\right)\partial_{j}\sigma_{t}\\ \partial_{j}\left(\bar{\rho}+\sigma\right) \d u^1 +\left(\bar{\rho}+\sigma\right)\partial_{j}\d u^1 \\ \partial_{j}\left(\bar{\rho}+\sigma\right) \d u^2 +\left(\bar{\rho}+\sigma\right)\partial_{j}\d u^2 \\ \partial_{j}\left(\bar{\rho}+\sigma\right) \d u^3 +\left(\bar{\rho}+\sigma\right)\partial_{j}\d u^3 \end{array}\right)\\
 =&\nabla \cdot \mathcal{D} \d \w+ \mathcal{D} \nabla\d \w ,\notag
\end{align}
\begin{align}
\nabla \left( \tA^{i}\w_{,i}\right) = \partial_{l}\left(\tA^{i}_{jk}w_{k,i}\right)=\partial_{l}\tA^{i}_{jk}w_{k,i}+\tA^{i}_{jk}\partial_{l}w_{k,i}=\nabla \tA^{i} \w_{,i} + \tA^{i}\nabla \w_{,i},
\end{align}
\begin{align}
\nabla \left(\B\w\right)=\partial_{l}\left(\B_{jk}w_{k}\right)=\partial_{l}\B_{jk} w_{k}+ \B_{jk}\partial_{l}w_{k}=\nabla \B\w+ \B\nabla \w,
\end{align}
\begin{align}
\nabla\left(\tilde{\mathcal{L}}_{\u}\right)=\left[\begin{array}{c}0 \\ \nabla\u\end{array}\right],
\end{align}
\begin{align}
\nabla\left(\tilde{\mathcal{L}}_{\phi}\right)=\left[\begin{array}{c}0 \\ \nabla\left(\left(\bar{\rho}+\sigma\right)\nabla\phi\right)\end{array}\right],
\end{align}
and
\begin{align}
\nabla\left(\tilde{f}\right)=\left[\begin{array}{c}0 \\ \nabla \left(O\left(\sigma^{2}\right)-\mathbb{F}\d  W\right)\end{array}\right]=\left[\begin{array}{c}0 \\ O\left(\sigma\nabla \sigma\right)\d t-\nabla\mathbb{F}\d  W\end{array}\right].
\end{align}
Hence \eqref{unsimplified 1 order derivative of symmetrized system} is deduced to
\begin{align}\label{1 order derivative of symmetrized system}
 &\d  \w \nabla \cdot \mathcal{D}+ \mathcal{D}\nabla \d  \w+ \left(\nabla \tA^{i} \w_{,i} + \tA^{i}\nabla \w_{,i}+ \nabla \B \w+ \B\nabla \w+ \nabla \u\right)\d t\notag\\
=& \left[\begin{array}{c}0 \\ \nabla\left(\left(\bar{\rho}+\sigma\right)\nabla\phi\right)\end{array}\right]\d t + \left[\begin{array}{c}0 \\ O\left(\sigma\nabla\sigma\right)\end{array}\right]\d t-\left[\begin{array}{c}0 \\\nabla \mathbb{F}\d  W \end{array}\right].
\end{align}
Multiplying \eqref{1 order derivative of symmetrized system} with $\nabla \w$, and integrating it on $U$, we have
\begin{align}
& \int_{U} \frac{1}{2}\mathcal{D} \d  \nabla \w : \nabla \w \d x + \int_{U} \tA^{i}\partial_{i} \left(\left|\nabla \w\right|^{2} \right)\d x \d t + \int_{U}\B\left|\nabla \w \right|^{2}\d x \d t+ \int_{U}\left|\nabla \u\right|^{2}\d x \d t\notag\\
= & \int_{U}\left[\begin{array}{c}0 \\ \nabla\left(\left(\bar{\rho}+\sigma\right)\nabla\phi\right)\end{array}\right]:\nabla \w \d x \d t - \int_{U}\nabla \tA^{i}\w_{,i}: \nabla \w \d x \d t \\
  & -  \int_{U} \nabla \B\w:\nabla \w \d x \d t - \int_{U} \d \w\nabla \cdot \mathcal{D}:\nabla \w \d x  + \int_{U} \left[\begin{array}{c}0 \\ O\left(\sigma\nabla\sigma\right)\end{array}\right]:\nabla \w \d x \d t \notag\\
  &- \int_{U}\left[\begin{array}{c}0 \\\nabla \mathbb{F}\d  W \end{array}\right]:\nabla \w \d x. \notag
\end{align}
Since $\sigma_{t}=-\nabla\cdot \left(\left(\bar{\rho}+\sigma\right)\u\right)$, we estimate
\begin{align}
  &\int_{U}  \frac{1}{2} \mathcal{D}_{t}\nabla \w : \nabla \w \d x \d t\notag\\
= &\int_{U} \nabla \w \operatorname{diag}\left[\left(Q'\left(\bar{\rho}+\sigma\right)\right)_{t},\left(\bar{\rho}+\sigma\right)_{t},\left(\bar{\rho}+\sigma\right)_{t},\left(\bar{\rho}+\sigma\right)_{t}\right]
\nabla \w \d x \d t\notag\\
= & \int_{U} \left( \sigma_{t}Q''\left(\bar{\rho}+\sigma\right)\sigma^{2}+\sigma_{t}\left|\nabla \u \right|^{2}\right) \d x \d t \\
= & \int_{U}\left(-Q''\left(\bar{\rho}+\sigma\right)\nabla\cdot \left(\left(\bar{\rho}+\sigma\right)\u\right)\sigma^{2}-\nabla\cdot \left(\left(\bar{\rho}+\sigma\right)\u\right)\left|\nabla\u\right|^{2}\right)\d x \d t\notag\\
\ls & C\left\|\w\right\|_{3}^{3} \d t.\notag
\end{align}
Due to the boundary conditions $\u\cdot \nu=0$, it follows that
\begin{align}
\int_{U} \tA^{i}\partial_{i} \left(\left|\nabla \w\right|^{2} \right)\d x  + \int_{U}\B\left|\nabla \w \right|^{2}\d x=0,
\end{align}
and
\begin{align}
& \int_{0}^{t}\int_{U}\left[\begin{array}{c}0 \\ \nabla\left(\left(\bar{\rho}+\sigma\right)\nabla\phi\right)\end{array}\right]:\nabla \w \d x \d s - \int_{0}^{t}\left(\int_{U}\nabla \tA^{i}\w_{,i}: \nabla \w \d x + \int_{U} \nabla \B\w:\nabla \w \d x\right) \d s \notag\\
 \ls & C \int_{0}^{t}\left\|\w\right\|_{3}^{3}\d s +  \frac{\delta_{5}}{3} \sup\limits_{s\in [0,t]} \left\|\nabla \u \right\|^{2} + C_{\delta_{5}}\int_{0}^{t}\left\|\nabla \phi \right\|^{2}\d s,
 \end{align}
 where $\int_{0}^{t}\left\|\nabla \phi \right\|^{2}\d s= \int_{0}^{t}\left\|\tilde{e} \right\|^{2}\d s$ can be bounded by $\int_{0}^{t}\left\|\w\right\|_{3}^{3}\d s $ from the zeroth-order energy estimates, $\delta_{5}$ being determined later.
 Similarly, we estimate
\begin{align}
    &\mathbb{E}\left[\left|\int_{0}^{t}\int_{U}\nabla \mathbb{F}\d  W:\nabla \w \d x \right|^{m}\right]\ls \mathbb{E}\left[\left| \frac{\delta_{5}}{3} \sup\limits_{s\in [0,t]} \left\|\u \right\|^{2}\right|^{m}\right] + \mathbb{E}\left[\left| C_{\delta_{5}}\int_{0}^{t}\left\|\w\right\|_{3}^{3} \d s\right|^{m}\right].
\end{align}
From \eqref{sto Euler-Poisson system}, we have
 \begin{align}
   & - \int_{U} \d \w\nabla \mathcal{D}:\nabla \w \d x \notag\\
 = &-\int_{U} \left(\mathcal{A}^{1}\w_{,1}+\mathcal{A}^{2}\w_{,2}+\mathcal{A}^{3}\w_{,3}+\mathcal{B}\w+\mathcal{L}_{\u}\right)\nabla \mathcal{D}:\nabla \w \d x \d t \\
  &+ \int_{U}\mathcal{L}_{\mathcal{\phi}}\nabla \mathcal{D}:\nabla \w \d x \d t +  \int_{U}\left(\mathcal{L}_{\mathcal{\phi}}+O\left(\sigma^{2}\right)\right)\nabla \mathcal{D}:\nabla \w \d x \d t -\int_{U} \mathbb{F}\d  W\nabla \mathcal{D}:\nabla \w \d x \notag\\
  \ls & C\left\|\w\right\|_{3}^{3}\d t + \int_{U} \mathbb{F}\d  W\nabla \mathcal{D}:\nabla \w \d x . \notag
 \end{align}
  Similarly, we have
\begin{align}
&\mathbb{E}\left[\left|\int_{0}^{t}\int_{U} \mathbb{F}\d  W\nabla \mathcal{D}:\nabla \w \d x\right|^{m}\right]\\
 \ls & \mathbb{E}\left[\left(\frac{\delta_{5}}{3} \sup\limits_{s\in [0,t]} \left\|\u \right\|^{2}\right)^{m}\right] + \mathbb{E}\left[\left(C_{\delta_{5}} \int_{0}^{t}\left\|\w\right\|_{3}^{3} \d s \right)^{m}\right].
\end{align}
We take $\delta_{5}$ such that $\frac{\delta_{5}}{3}\left\|\u \right\|^{2}$ and $\frac{\delta_{5}}{3}\left\|\nabla \u \right\|^{2}$ can be balanced by the left side of energy estimates. Similar as the estimates for \eqref{zero order plus estimate}, we have the estimate of $\int_{0}^{t} \left\|\nabla\sigma\right\|^{2} \d s$.
In conclusion, we have
\begin{align}
&\mathbb{E}\left[\left|\sup\limits_{s\in[0,t]} \int_{0}^{s} \d  \left( \frac{1}{2}\int_{U} \mathcal{D} \nabla \w : \nabla \w  \d x\right)\right|^{m}\right] + \mathbb{E}\left[\left|c_{4}\int_{0}^{t} \int_{U}\left|\nabla \w \right|^{2}\d x \d s \right|^{m}\right] \\
\ls  &\mathbb{E}\left[\left|\int_{0}^{t} C \left\|\w\right\|_{3}^{3} \d s\right|^{m}\right],\notag
\end{align}
where $C$ is independent on $t$.
\subsubsection{Second order estimates}\label{second order estimates subsec}
We write \eqref{sto Euler-Poisson system} in the form of components, and the $i-$th equation is
\begin{align}
&d_{i} \d w_{i}+ \left(\left(\tA^{1}\right)_{ij}w_{j,1}+ \left(\tA^{2}\right)_{ij}w_{j,2}+\left(\tA^{3}\right)_{ij}w_{j,3}+\left(\B\right)_{ij}w_{j}+d_{i}w_{i}\right)\d t\\
=&\left(d_{i}\phi_{,i}+h\left(\sigma\right)_{i}\right)\d t - \mathbb{F}_{i}\d  W.\notag
\end{align}
Taking the second-order derivatives, we have
\begin{align}\label{system of 2 order derivative}
&\partial_{kl}^{2}\left(d_{i} \d  w_{i}+ \left(\left(\tA^{1}\right)_{ij}w_{j,1}+ \left(\tA^{2}\right)_{ij}w_{j,2}+\left(\tA^{3}\right)_{ij}w_{j,3}+\left(\B\right)_{ij}w_{j}+d_{i}w_{i}\right)\d t\right)\\
&=\partial_{kl}^{2}\left(\left(d_{i}\phi_{,i}+h\left(\sigma\right)_{i}\right)\d t-\mathbb{F}_{i}\d  W\right).\notag
\end{align}
Multiplying \eqref{system of 2 order derivative} with $\partial_{k}\partial_{l}w_{i}$ and integrating it over $U$, we have
\begin{align}
\int_{U} d_{i} \partial_{k}\partial_{l}\d w_{i}\partial_{k}\partial_{l}w_{i} \d x  = \int_{U} d_{i} \d  \left|\partial_{k}\partial_{l}w_{i}\right|^{2} \d x,
\end{align}
By the insulated boundary condition $\u\cdot \nu=0$,
for all $i,j$, there holds
\begin{align}
\int_{U}\partial_{k}\partial_{l}w_{j}\left(-\frac{1}{2} \left(\left(\tA^{1}\right)_{ij,1}+\left(\tA^{2}\right)_{ij,2}+\left(\tA^{3}\right)_{ij,3}\right)+\left(\B\right)_{ij}\right) \partial_{k}\partial_{l}w_{i} \d x=0.
\end{align}
By\eqref{sto Euler-Poisson system}, 
the integral in the deterministic terms are bounded by $C \left\|\w\right\|_{3}^{3}$. 
The stochastic term is estimated as follows
\begin{align}
&\mathbb{E}\left[\left|\int_{0}^{t}\int_{U} \partial_{kl}^{2}\mathbb{F}_{i} \partial_{k}\partial_{l}w_{i}\d x \d  W \right|^{m}\right]\\
 \ls &\mathbb{E}\left[\left(\delta_{6}\sup\limits_{s\in [0,t]}\left\|\nabla \w\right\|^{2} \right)^{m}\right]+ \mathbb{E}\left[\left(C_{\delta_{6}}\int_{0}^{t}\left\|\w\right\|_{3}^{3}\d s \right)^{m}\right], \notag
\end{align}
where $\delta_{6}$ is taken such that $\delta_{6}\sup\limits_{s\in [0,t]}\left\|\nabla \w\right\|^{2}$ can be obtained by left side in first-order estimates. Similar to the estimates \eqref{zero order plus estimate}, we have the estimates for $\int_{0}^{t} \int_{U}\left|\partial^{2} \sigma \right|^{2}\d x \d s$.
Taking the sum over all the index $i=1,2,3,4$, we have
\begin{align}
 &\mathbb{E}\left[\left|\sup\limits_{s\in[0,t]} \int_{0}^{s} \d  \left(\int_{U} \frac{1}{2} \left| \partial^{2} \w \right|^{2} \d x\right) \right|^{m}\right]
  + \mathbb{E}\left[\left|c_{5}\int_{0}^{t} \int_{U}\left|\partial^{2} \w \right|^{2}\d x \d s\right|^{m}\right] \\
  \ls & \mathbb{E}\left[\left| \int_{0}^{t} C \left\|\w\right\|_{3}^{3} \d s \right|^{m}\right],\notag
\end{align}
 with the assumption that $\bar{\rho}$ have a positive lower bound, where $C$ is independent on $t$.
\subsubsection{Third-order estimates}

Considering the $3$-order estimates, we take an additional derivative of \eqref{system of 2 order derivative}. Repeating the argument in subsection \ref{second order estimates subsec}, we have
\begin{align}
&\mathbb{E}\left[\left| \sup\limits_{s\in[0,t]}\int_{0}^{s} \d  \left( \frac{1}{2}\int_{U} \mathcal{D} \nabla \w : \nabla \w  \d x\right)\right|^{m}\right] + \mathbb{E}\left[\left|c_{4}\int_{0}^{t} \int_{U}\left|\nabla \w \right|^{2}\d x \d s \right|^{m}\right] \\
\ls  &\mathbb{E}\left[\left|\int_{0}^{t} C \left\|\w\right\|_{3}^{3} \d s\right|^{m}\right],\notag
\end{align}
where $C$ is independent on $t$.
 \begin{align}
 &\mathbb{E}\left[\left|\sup\limits_{s\in[0,t]} \int_{0}^{s} \d \left( \int_{U} \frac{1}{2} \left| \partial^{3} \w \right|^{2} \d x\right) \right|^{m}\right] +  \mathbb{E}\left[\left|c_{6} \int_{0}^{t} \int_{U}\left|\partial^{3} \w\right|^{2}\d x \d s \right|^{m}\right] \\
 \ls &\mathbb{E}\left[\left|\int_{0}^{t} C \left\|\w\right\|_{3}^{3} \d s \right|^{m}\right],\notag
 \end{align}
  with the assumption that $\bar{\rho}$ have a positive lower bound.

\smallskip
\begin{flushleft}
\textbf{Step 5: Global existence.}
\end{flushleft}
\subsection{Global existence}

In this subsection, we show the global existence for both cases on stochastic forces under \eqref{2 Condition for stochastic force} and general forces.

 \subsubsection{For stochastic forces under \eqref{2 Condition for stochastic force} and small perturbation for initial data \eqref{initial conditions assumption}}
We combine the energy estimates up to third order. Then, the assumption that $\bar{\rho}$ have a positive lower bound, leads to the following inequality:
\begin{align}\label{energy estimate for insulating bdy con}
&\mathbb{E}\left[\left|\sup\limits_{s\in[0,t]}\left(\left\|\w\right\|_{3}^{2}(s)+\left\|\nabla \phi\right\|^{2}(s)\right) +\alpha \int_{0}^{t}\left(\left\|\w\right\|_{3}^{2}+\left\|\nabla \phi\right\|^{2}\right) (s)\d s \right|^{m}\right]\\
\ls &\mathbb{E}\left[\left|C\int_{0}^{t}\left\|\w\right\|_{3}^{3}\d s\right|^{m}\right] +\mathbb{E}\left[\left(C \left(\left\|\w_{0}\right\|_{3}^{2} + \left\|\nabla \phi_{0}\right\|^{2} \right)  \right)^{m}\right], \notag
\end{align}
 where $\alpha \ls c_{i}, i=1,\cdots, 6$, and $C$ depends on $\bar{\rho}$, $m$ and the domain $U$, but is independent on $t$.
Since $\left\|\w\right\|_{3}$ is small, we have
\begin{align}
&\mathbb{E}\left[\sup\limits_{s\in[0,t]}\left(\left|\left\|\w\right\|_{3}^{2}+ \left\|\nabla \phi\right\|^{2}\right)+\alpha \int_{0}^{t}\left(\left\|\w\right\|_{3}^{2}+\left\|\nabla \phi\right\|^{2}\right)(s)\d s -C\int_{0}^{t}\left\|\w\right\|_{3}^{3}\d s\right|^{m}\right]\\
\ls &\mathbb{E}\left[\left(C\left(\left\|\w_{0}\right\|_{3}^{2} + \left\|\nabla \phi_{0}\right\|^{2} \right)  \right)^{m}\right],\notag
\end{align}
and,
\begin{align}
\mathbb{E}\left[\left(\sup\limits_{s\in[0,t]}\left(\left\|\w\right\|_{3}^{2}+ \left\|\nabla \phi\right\|^{2}\right)\right)^{m} \right]\ls \mathbb{E}\left[\left(C\left(\left\|\w_{0}\right\|_{3}^{2} + \left\|\nabla \phi_{0}\right\|^{2} \right)  \right)^{m}\right],
\end{align}
where $C$ is independent on $t$.
With the above uniform estimates for any time $t$, and the local existence on $\left[0, T_{1}\right]$, we can extend the existence to $\left[T_{1}, T_{1}+\tilde{T}\right]$, and extend to any time $T_{1}+k\tilde{T}, \forall ~k\in \mathds{N}^{+}$.
More specifically, for the estimate of onto mapping, if
\begin{align}
\mathbb{E}\left[\left(\sup\limits_{s\in[T_{1},t]} \left\| \w_{n-1}(s)\right\|_{3}^{2}\right)^{m}\right] \ls &4\mathbb{E}\left[\left(\left(\left\|\w(T_{1})\right\|_{3}^{2} + \left\|\nabla \phi(T_{1})\right\|^{2} \right) \right)^{m}\right]\\
 \ls & 4 \mathbb{E}\left[\left(C\left(\left\|\w_{0}\right\|_{3}^{2} + \left\|\nabla \phi_{0}\right\|^{2} \right)  \right)^{m}\right] , \notag
\end{align}
 then
 \begin{align}
 \mathbb{E}\left[\left(\sup\limits_{s\in[T_{1},t]}\left\| \w_{n}(s)\right\|_{3}^{2}\right)^{m}\right] \ls
 4 \mathbb{E}\left[\left(C\left(\left\|\w_{0}\right\|_{3}^{2} + \left\|\nabla \phi_{0}\right\|^{2} \right)  \right)^{m}\right].
 \end{align}
Similarly, the contraction holds from $T_{1}$ to $T_{1}+\tilde{T}$. Then the existence is extended to $T_{1}+k\tilde{T}$ for any $k\in \mathds{N}^{+}$.
In conclusion, we obtain the global existence of $\w$ and $\phi$, which is equivalent to the global existence of strong solutions $(\rho,\u,\Phi)$ stated by the following proposition.
 \begin{proposition}
In $(\Omega, \mathcal{F}, \mathbb{P})$, there exists a unique global-in-time strong solution $(\rho,\u,\Phi)$ to \eqref{3-D Euler Poisson}:
\begin{align}
 \rho,  \u \in  C\left([0,T]; H^{3}\left(U\right)\right), \Phi\in  C\left([0,T]; H^{5}\left(U\right)\right), \forall ~T>0,
\end{align}
up to a modification, where $m\geqslant 2$ is a constant.
\end{proposition}

\subsubsection{For general stochastic forces}
 If the stochastic forces has linear growth in $\rho\u$, then the following energy estimates hold
  \begin{align}
 &\mathbb{E}\left[\left|\sup\limits_{s\in[0,t]} \int_{0}^{s} \d \left(\left\|\w\right\|_{3}^{2}(s)+\left\|\nabla \phi\right\|^{2}(s)\right)+\alpha \int_{0}^{t} \int_{U}\left(\left\|\w\right\|_{3}^{2}(s)+\left\|\nabla \phi\right\|^{2}(s)\right)\d x \d s \right|^{m}\right]\notag \\
 \ls &\mathbb{E}\left[\left|\int_{0}^{t} C \left\|\w\right\|_{3}^{3} \d s \right|^{m}\right]+\mathbb{E}\left[\left|\int_{0}^{t} C \left\|\w\right\|_{3}^{2} \d s \right|^{m}\right].
 \end{align}
 Without the small perturbation of initial data \eqref{initial conditions assumption}, we can use the generalized Gr\"onwall's inequality to obtain
\begin{align}\label{estimate without small pert}
\mathbb{E}\left[\left(\sup\limits_{s\in[0,t]}\left(\left\|\w\right\|_{3}^{2}+ \left\|\nabla \phi\right\|^{2}\right)\right)^{m} \right]\ls \mathbb{E}\left[\left(C(t)\left(\left\|\w_{0}\right\|_{3}^{2} + \left\|\nabla \phi_{0}\right\|^{2} \right)  \right)^{m}\right],
\end{align}
where $C(t)$ is increasing with respect to $t$.
Similarly, if the stochastic forces have cubic growth in $\rho\u$, then the energy estimates become
   \begin{align}
 &\mathbb{E}\left[\left|\sup\limits_{s\in[0,t]} \int_{0}^{s} \d \left(\left\|\w\right\|_{3}^{2}(s)+\left\|\nabla \phi\right\|^{2}(s)\right)+\alpha \int_{0}^{t} \int_{U}\left(\left\|\w\right\|_{3}^{2}(s)+\left\|\nabla \phi\right\|^{2}(s)\right)\d x \d s \right|^{m}\right]\notag\\
 \ls &\mathbb{E}\left[\left|\int_{0}^{t} C \left\|\w\right\|_{3}^{3} \d s \right|^{m}\right]+\mathbb{E}\left[\left|\int_{0}^{t} C \left\|\w\right\|_{3}^{4} \d s \right|^{m}\right].
 \end{align}
 By the generalized Gr\"onwall's inequality, there also holds \eqref{estimate without small pert}.
Hence, for the smooth $Y$ in \eqref{condition for F for 3-d} and can be bounded by the homogeneous polynomials, the estimates of \eqref{estimate without small pert} holds as well. For the estimate of onto mapping, for any fixed $T$, $t\in [0,T]$, if
\begin{align}
\mathbb{E}\left[\left(\sup\limits_{s\in[T_{1},t]} \left\| \w_{n-1}(s)\right\|_{3}^{2}\right)^{m}\right] \ls &4\mathbb{E}\left[\left(\left(\left\|\w(T_{1})\right\|_{3}^{2} + \left\|\nabla \phi(T_{1})\right\|^{2} \right) \right)^{m}\right]\\
 \ls & 4 \mathbb{E}\left[\left(C(T)\left(\left\|\w_{0}\right\|_{3}^{2} + \left\|\nabla \phi_{0}\right\|^{2} \right)  \right)^{m}\right] , \notag
\end{align}
 then
 \begin{align}
 \mathbb{E}\left[\left(\sup\limits_{s\in[T_{1},t]}\left\| \w_{n}(s)\right\|_{3}^{2}\right)^{m}\right] \ls
 4 \mathbb{E}\left[\left(C(T)\left(\left\|\w_{0}\right\|_{3}^{2} + \left\|\nabla \phi_{0}\right\|^{2} \right)  \right)^{m}\right].
 \end{align}
Thus, we extend the local existence on $\left[0, T_{1}\right]$ to $\left[0, T_{1}+\tilde{T}\right]$, and to $\left[0, T_{1}+k\tilde{T}\right], \forall ~k\in \mathds{N}^{+}$. By Zorn's lemma, the global existence holds.

\section{Asymptotic stability of solutions}

In this section, we consider the stability under the assumptions of \eqref{2 Condition for stochastic force} and \eqref{initial conditions assumption}. The {\it a priori} estimates \eqref{energy estimate for insulating bdy con} shows the stability of solutions around the steady state. However, \eqref{energy estimate for insulating bdy con} is insufficient for investigating the decay rate since the {\it a priori} estimates are already in the form of time integrals rather than a differential inequality. Integrating twice with respect time might not be wise as it could lead to disappearance of the favorable temporal properties.
The asymptotic decay of solution is then derived from the following weighted estimates up to second-order. To manipulate the weighted energy estimates for stochastic system, we need multiply $\d  \left(\frac{1}{2}\mathcal{D}\w\cdot \w \right)$ directly with $e^{\alpha t }$ first, where $\alpha$ is in \eqref{energy estimate for insulating bdy con}. Then we integrate it with respect to $x$, $t$, and $\omega$, to estimate the time integral.

\subsection{Weighted decay estimates}\label{Weighted decay estimates}

\subsubsection{Zeroth-order weighted estimates}\label{Zeroth-order weighted estimates}
\smallskip
We multiply \eqref{equation of D w w} with $e^{\alpha t }$, then we have
\begin{align}\label{weighted equation of D w w}
 & \int_{U}  e^{\alpha t } \d  \left(\frac{1}{2}\mathcal{D}\w\cdot \w \right) \d x \notag\\
=& e^{\alpha t }\int_{U}\frac{1}{2} \w \left(\d \mathcal{D}\right)\w\d x -e^{\alpha t }\int_{U}\left(\tilde{\mathcal{A}}^{1}\w_{,1}\cdot \w + \tilde{\mathcal{A}}^{2}\w_{,2}\cdot \w+ \tilde{\mathcal{A}}^{3}\w_{,3}\cdot \w +\B \w\cdot \w\right) \d x \d t \notag \\
 & -e^{\alpha t }\int_{U} \tilde{\mathcal{L}}_{\u}\cdot \w \d x \d t  + e^{\alpha t }\int_{U} \tilde{\mathcal{L}}_{\phi}\cdot \w \d x \d t +e^{\alpha t }\int_{U} \nabla h\left(\sigma\right)\cdot \w \d x \d t \\
 &  + e^{\alpha t }\int_{U}\mathcal{D} \mathbb{F}\d  W\cdot \w\d x + e^{\alpha t }\int_{U} \mathcal{D}\mathbb{F}\cdot\mathbb{F}\d x \d t. \notag
\end{align}
From the estimates of zeroth-order estimates in subsection \ref{Estimates up to third-order subtitle}, we conclude the following estimates omitting detailed calculation:
\begin{align}
\int_{U} e^{\alpha t }\left(-\frac{1}{2}\left(\w \tilde{\mathcal{A}}^{1}_{1}\w +\w \tilde{\mathcal{A}}^{2}_{2}\w + \w \tilde{\mathcal{A}}^{3}_{3}\w \right)+ \B\left|\w\right|^{2}\right) \d x \d t \ls C e^{\alpha t }\left\|\w\right\|_{3}^{3}\d t;
\end{align}
\begin{align}
e^{\alpha t } \int_{U}\tilde{\mathcal{L}}_{\u}\cdot \w \d x \d t \geqslant e^{\alpha t } \int_{U} C\bar{\rho}\left|\u\right|^{2}\d x \d t \geqslant \alpha e^{\alpha t } \int_{U} C\bar{\rho}\left|\u\right|^{2}\d x \d t;
\end{align}
\begin{align}
e^{\alpha t } \int_{U} \tilde{\mathcal{L}}_{\phi}\cdot \w \d x \d t=-e^{\alpha t } \d  \int_{U}\left|\nabla\phi\right|^{2} \d x;
\end{align}
\begin{align}
 e^{\alpha t } \int_{U} \w \left(\d \mathcal{D}\right)\w\d x \ls  C e^{\alpha t }\left\|\w\right\|_{3}^{3}\d t ;
\end{align}
\begin{align}
 e^{\alpha t } \int_{U} \frac{1}{2} \mathcal{D} \mathbb{F}\cdot \mathbb{F}\d x \d t \ls C  e^{\alpha t}\left\|\w\right\|_{3}^{3}\d t.
\end{align}
For the estimates of stochastic integral, it holds
\begin{align}
 e^{\alpha t }\int_{U} \tilde{f}\cdot \w \d x \ls  e^{\alpha t } \left\|\w\right\|_{3}^{3}+  e^{\alpha t }\left|\int_{U} \mathbb{F}\d  W\cdot\u\d x \right|.
\end{align}
For $\left|\mathbb{F}\right|\ls C \left|\rho\u\right|^{2}$,
\begin{align}
    &\mathbb{E}\left[\left|\int_{0}^{t}  e^{\alpha s }\int_{U} \mathbb{F}\cdot\u\d x \d  W \right|^{m}\right]\ls \mathbb{E}\left[\left(C\int_{0}^{t} e^{2\alpha s } \left|\int_{U} \mathbb{F}\cdot\u\d x\right|^{2}\d s\right)^{\frac{m}{2}}\right]\notag\\
\ls &  \mathbb{E}\left[\left(C\int_{0}^{t}  e^{2\alpha s } \left|\int_{U} \left|\bar{\rho}\u\right|^{2}\left|\u\right|\d x\right|^{2}\d s\right)^{\frac{m}{2}}\right]\ls \mathbb{E}\left[\left(C\sup\limits_{s\in [0,t]} \left\|\u \right\|^{2} \int_{0}^{t}  e^{2\alpha s } \left\| \u \right\|_{3}^{4} \d s\right)^{\frac{m}{2}}\right] \notag\\
\ls & e^{\alpha mt} \mathbb{E}\left[\left(\sup\limits_{s\in [0,t]} \left\|\u \right\|^{2}\right)^{m}\right] + e^{\alpha mt}\mathbb{E}\left[\left(C\int_{0}^{t} \left\| \u \right\|_{3}^{3} \d s\right)^{m}\right]\\
\ls & e^{\alpha mt}\mathbb{E}\left[\left(C\int_{0}^{t} \left\| \u \right\|_{3}^{3} \d s\right)^{m}\right],\notag
\end{align}
where the last inequality holds due to the zeroth-order estimates in subsection \ref{Estimates up to third-order subtitle}, $C$ is a general constant.
In summary, as $\bar{\rho}$ have a positive lower bound, we have
\begin{align}\label{zero order estimate of w 3-D summary}
&\mathbb{E}\left[\left(\int_{0}^{t} e^{\alpha s } \d \left(\int_{U}\left|\w\right|^{2}\d x + \int_{U}\left|\nabla\phi\right|^{2}\d x\right) \right)^{m}\right]+\mathbb{E}\left[\left(\int_{0}^{t} e^{\alpha s }\int_{U}\left|\u\right|^{2}\d x \d s \right)^{m}\right]\notag \\
\ls  &e^{\alpha mt} \mathbb{E}\left[\left(C\int_{0}^{t} \left\|\w\right\|_{3}^{3} \d s\right)^{m}\right].
\end{align}

Next, we give the estimates of $ \int_{0}^{t} e^{\alpha s }\int_{U} \left\|\sigma\right\|^{2} \d x \d s$.
From the velocity equation \eqref{deformed sto Euler-Poisson}, we have
\begin{align}\label{weighted zero order plus estimate eq}
&e^{\alpha t } \left(\nabla Q\left(\bar{\rho}+\sigma\right) -\nabla Q\left(\bar{\rho}\right)\right)\d t\\
 = &-e^{\alpha t }\d  \u -e^{\alpha t }\left(\left(\u \cdot\nabla \right)\u-\u\right)\d t + e^{\alpha t }\nabla \phi \d t + e^{\alpha t }\frac{\mathbb{F}}{\bar{\rho}+\sigma}\d  W, \notag
\end{align}
with
\begin{align}
 &\nabla \left(Q\left(\bar{\rho}+\sigma\right)-Q\left(\bar{\rho}\right)\right)
= Q'\left(\bar{\rho}+\sigma\right)\nabla \sigma+Q''\left(\bar{\rho}\right)\sigma \nabla\bar{\rho}+ \mathbf{h},\notag
\end{align}
where
\begin{align}
h_{i}=O\left(\sigma^{2}\right).
\end{align}
We multiply the equation \eqref{weighted zero order plus estimate eq} with $\left(\sigma, \sigma, \sigma\right)^{T}$. The left side of \eqref{weighted zero order plus estimate eq} is
\begin{align}
e^{\alpha t } \int_{U} \left|Q''\left(\bar{\rho}\right)\nabla\bar{\rho}\right| \left|\sigma\right|^{2}\d x + e^{\alpha t } \int_{U} O \left(\sigma^{3}\right)\d x.
\end{align}
By It\^o's formula,
\begin{align}
e^{\alpha t }\left(\d u^{i}\right) \sigma =e^{\alpha t }\d \left(u^{i} \sigma\right) - e^{\alpha t } u^{i} \d \sigma,
\end{align}
where
\begin{align}
-\int_{0}^{t} e^{\alpha s }\d \int_{U} \left(u^{i} \sigma\right) \d x \ls \int_{0}^{t}  e^{\alpha s }\d  \left(\frac{1}{2}\left\|\sigma\right\|^{2} + \frac{1}{2}\left\|u^{i}\right\|^{2}\right) .
\end{align}
By the continuity equation, it holds
\begin{align}
\int_{0}^{t} e^{\alpha s }\int_{U} \left|u^{i} \d \sigma\right|\d x \ls C\int_{0}^{t} e^{\alpha s }\left\|\w\right\|_{3}^{3} \d s.
\end{align}
For $-\u\d t$, we directly estimate
\begin{align}
\int_{0}^{t}\int_{U} e^{\alpha s }\left| -u^{i}\sigma\right| \d x \d s \ls \frac{\delta_{4}}{2}\int_{0}^{t} e^{\alpha s }\left\|\sigma\right\|^{2}\d s +C_{\delta_{4}}\int_{0}^{t}\left\|u^{i}\right\|^{2} \d s,
\end{align}
where $\delta_{4}$ is small such that $ \delta_{4}\int_{0}^{t} e^{\alpha s } \left\|\sigma\right\|^{2}\d s$ can be balanced by the left side. For the term $\nabla \phi \d t$ in \eqref{zero order plus estimate eq}, we estimate
\begin{align}
\int_{0}^{t}\int_{U} e^{\alpha s }\left|-\phi_{,i}\sigma\right|\d x\d s \ls \frac{\delta_{4}}{2}\int_{0}^{t} e^{\alpha s } \left\|\sigma\right\|^{2}\d s + C_{\delta_{4}} e^{\alpha t } \sup_{s\in [0,t]}\left\|\phi_{,i}\right\|^{2}.
\end{align}
For the stochastic term, since $\left|\mathbb{F}\right|\ls C\left|\rho\u\right|^{2}$, we estimate
\begin{align}
&\mathbb{E}\left[\left|\int_{0}^{t} e^{\alpha s }\int_{U}\frac{\mathbb{F}^{i}}{\bar{\rho}+\sigma}\d  W \sigma \d x\right|^{m}\right]
\ls \mathbb{E}\left[\left|C\int_{0}^{t} e^{\alpha s }\left|\int_{U}\frac{\mathbb{F}^{i}}{\bar{\rho}+\sigma} \sigma \d x\right|^{2} \d s\right|^{\frac{m}{2}}\right]\notag \\
\ls & \mathbb{E}\left[\left|C\int_{0}^{t} e^{\alpha s }\left|\int_{U}\left|\bar{\rho}+\sigma\right|\left|\u\right|^{2} \sigma \d x\right|^{2} \d s\right|^{\frac{m}{2}}\right]
\ls  \mathbb{E}\left[\left|C\int_{0}^{t} e^{\alpha s }\left\|\u\right\|^{2}\left\|\bar{\rho}\sigma+\sigma^{2}\right\|_{\infty}^{2}\d s\right|^{\frac{m}{2}}\right] \notag \\
\ls &  e^{\alpha mt} \mathbb{E}\left[\left(\frac{1}{4}\sup_{s\in [0,t]}\left\|\u\right\|^{2}\right)^{m}\right] +\mathbb{E}\left[\left(C\int_{0}^{t} e^{\alpha s }\left\|\u\right\|^{2}\left\|\sigma\right\|_{\infty}^{2}\d s\right)^{m}\right]  \notag\\
&+ e^{\alpha mt} \mathbb{E}\left[\left(\frac{1}{4}\sup_{s\in [0,t]} \left\|\u\right\|^{2}\right)^{m}\right] +  \mathbb{E}\left[\left(C\int_{0}^{t} e^{\alpha s }\left\|\u\right\|^{2}\left\|\sigma\right\|_{\infty}^{4}\d s\right)^{m}\right] \\
\ls & e^{\alpha mt}\mathbb{E}\left[\left(\frac{1}{2}\sup_{s\in [0,t]} \left\|\u\right\|^{2}\right)^{m}\right] + \mathbb{E}\left[\left(C\int_{0}^{t} e^{\alpha s } \left\|\w\right\|_{3}^{3}\d s\right)^{m}\right] \notag\\
\ls & e^{\alpha mt} \mathbb{E}\left[\left(C\int_{0}^{t} \left\|\w\right\|_{3}^{3}\d s\right)^{m}\right] + \mathbb{E}\left[\left(C\int_{0}^{t} e^{\alpha s } \left\|\w\right\|_{3}^{3}\d s\right)^{m}\right].\notag
\end{align}
Therefore, we have
\begin{align}\label{zero order plus sigma estimate}
&\mathbb{E}\left[\left(\int_{0}^{t} e^{\alpha s } \left\|\sigma\right\|^{2}\d s\right)^{m}\right]\notag \\
\ls & \mathbb{E}\left[\left(\int_{0}^{t} e^{\alpha s } \d  \left(\frac{1}{2}\left\|\sigma\right\|^{2} + \frac{1}{2}\left\|\u\right\|^{2}\right)\right)^{m}\right]\\
&+ e^{\alpha mt} \mathbb{E}\left[\left(C\int_{0}^{t}\left\|\w\right\|_{3}^{3}\d s\right)^{m}\right] + \mathbb{E}\left[\left(C\int_{0}^{t} e^{\alpha s } \left\|\w\right\|_{3}^{3}\d s\right)^{m}\right]\notag\\
\ls &  \mathbb{E}\left[\left(\int_{0}^{t} e^{\alpha s } \d  \left(\frac{1}{2}\left\|\sigma\right\|^{2} + \frac{1}{2}\left\|\u\right\|^{2}\right)\right)^{m}\right]+e^{\alpha mt} \mathbb{E}\left[\left(C\int_{0}^{t} \left\|\w\right\|_{3}^{3}\d s\right)^{m}\right]. \notag
\end{align}

Similarly, the estimates for $\mathbb{E}\left[\left(\int_{0}^{t} e^{\alpha s } \left\|\nabla \phi \right\|^{2}\d s\right)^{m}\right]$ holds:
\begin{align}\label{zero order plus phi estimate}
&\mathbb{E}\left[\left(\int_{0}^{t}  e^{\alpha s }\left\|\nabla \phi\right\|^{2}\d s\right)^{m}\right]\notag \\
\ls & \mathbb{E}\left[\left(\int_{0}^{t} e^{\alpha s } \d  \left(\frac{1}{2}\left\|\nabla \phi\right\|^{2} + \frac{1}{2}\left\|\u\right\|^{2}\right)\right)^{m}\right]\\
&+ e^{\alpha mt} \mathbb{E}\left[\left(C\int_{0}^{t}\left\|\w\right\|_{3}^{3}\d s\right)^{m}\right] + \mathbb{E}\left[\left(C\int_{0}^{t} e^{\alpha s } \left\|\w\right\|_{3}^{3}\d s\right)^{m}\right]\notag\\
\ls &  \mathbb{E}\left[\left(\int_{0}^{t} e^{\alpha s } \d  \left(\frac{1}{2}\left\|\nabla \phi\right\|^{2} + \frac{1}{2}\left\|\u\right\|^{2}\right)\right)^{m}\right]+e^{\alpha mt} \mathbb{E}\left[\left(C\int_{0}^{t} \left\|\w\right\|_{3}^{3}\d s\right)^{m}\right]. \notag
\end{align}

Multiplying a small constant to \eqref{zero order plus sigma estimate} and \eqref{zero order plus phi estimate}, we plus the zero-order estimates \eqref{zero order estimate of w 3-D} such that
\begin{align}
&\mathbb{E}\left[\left(\int_{0}^{t} e^{\alpha s }\d  \left(\frac{1}{2}\left\|\sigma\right\|^{2}+ \left\|\u\right\|^{2}+ \frac{1}{2}\left\|\nabla \phi \right\|^{2}\right)\right)^{m}\right]
\end{align}
 can be balanced by \eqref{zero order estimate of w 3-D}. Then we obtain
\begin{align}
&\mathbb{E}\left[\left(\int_{0}^{t} e^{\alpha s }\d  \left(\left\|\w\right\|^{2}+  \left\|\nabla\phi\right\|^{2} \right)+ \alpha\int_{0}^{t} e^{\alpha s } \left( \left\|\w\right\|^{2}+\left\|\nabla\phi\right\|^{2}\right) \d s \right)^{m}\right]\\
 \ls  &e^{\alpha mt} \mathbb{E}\left[\left(C \int_{0}^{t}\left\|\w\right\|_{3}^{3} \d t\right)^{m}\right].\notag
\end{align}

\subsubsection{First-order weighted estimates}

Multiplying \eqref{1 order derivative of symmetrized system} by $e^{\alpha t }\nabla \w$ and integrating it over $U$, we can repeat the argument from subsection \ref{Zeroth-order weighted estimates} to obtain:
\begin{align}
&\mathbb{E}\left[\left| \int_{0}^{t} e^{\alpha s } \d  \left( \int_{U} \mathcal{D} \nabla \w : \nabla \w  \d x\right)\right|^{m}\right] + \mathbb{E}\left[\left|\int_{0}^{t} \alpha e^{\alpha s }\int_{U}\left|\nabla \w \right|^{2}\d x \d s \right|^{m}\right] \\
\ls  &  \mathbb{E}\left[\left|e^{\alpha t } \int_{0}^{t} C \left\|\w\right\|_{3}^{3} \d s\right|^{m}\right]. \notag
\end{align}

\subsubsection{Second-order weighted estimates}

Similarly, we multiply \eqref{system of 2 order derivative} with $e^{\alpha t } \partial^{2} \w$, and then integrate it on $U$. Repeating the procedure in subsection \ref{Zeroth-order weighted estimates}, we have
\begin{align}
 &\mathbb{E}\left[\left| \int_{0}^{t} e^{\alpha s } \d  \left(\int_{U} \left| \partial^{2} \w \right|^{2} \d x\right) \right|^{m}\right]
  + \mathbb{E}\left[\left| \int_{0}^{t} \alpha e^{\alpha s } \int_{U}\left|\partial^{2} \w \right|^{2}\d x \d s\right|^{m}\right] \\
  \ls & \mathbb{E}\left[\left|e^{\alpha t }  \int_{0}^{t} C \left\|\w\right\|_{3}^{3} \d s \right|^{m}\right].\notag
\end{align}

\subsubsection{Third-order weighted estimates}

 Considering the $3$-order weighted estimates, following the standard bootstrap of subsection \ref{Zeroth-order weighted estimates}, we have
 \begin{align}
 &\mathbb{E}\left[\left| \int_{0}^{t} e^{\alpha s }\d \left( \int_{U}  \left| \partial^{3} \w \right|^{2} \d x\right) \right|^{m}\right] +  \mathbb{E}\left[\left| \int_{0}^{t} \alpha e^{\alpha s }\int_{U}\left|\partial^{3} \w\right|^{2}\d x \d s \right|^{m}\right]\\
  \ls &\mathbb{E}\left[\left|e^{\alpha t }\int_{0}^{t} C \left\|\w\right\|_{3}^{3} \d s \right|^{m}\right].\notag
 \end{align}

\subsection{Asymptotic stability}

Combining the weighted estimates in the previous subsections, we obtain
 \begin{align}
 &\mathbb{E}\left[\left| \int_{0}^{t} e^{\alpha s }\d \left(\left\|\w\right\|_{3}^{2} + \left\|\nabla \phi\right\|^{2} \right) +\int_{0}^{t} \alpha e^{\alpha s }\left(\left\|\w\right\|_{3}^{2}+\left\|\nabla \phi\right\|^{2}\right)\d s \right|^{m}\right]\\
  \ls &\mathbb{E}\left[\left|e^{\alpha t }\int_{0}^{t} C \left\|\w\right\|_{3}^{3} \d s \right|^{m}\right].\notag
 \end{align}
Therefore, we have
 \begin{align}
 &\mathbb{E}\left[\left|  e^{\alpha t }\left(\left\|\w\right\|_{3}^{2} + \left\|\nabla \phi\right\|^{2} \right) \right|^{m}\right]\\
  \ls &\mathbb{E}\left[\left|\left(\left\|\w_{0}\right\|_{3}^{2} + \left\|\nabla \phi_{0}\right\|^{2} \right) \right|^{m}\right] + \mathbb{E}\left[\left|e^{\alpha t }\int_{0}^{t} C \left\|\w\right\|_{3}^{3} \d s \right|^{m}\right].\notag
 \end{align}
 Since $\left\|\w_{0}\right\|_{3}^{2} + \left\|\nabla \phi_{0}\right\|^{2}$ is small, we have
 \begin{align}
&\mathbb{E} \left[\left|e^{\alpha t }\left(\left\|\w\right\|_{3}^{2} + \left\|\nabla \phi\right\|^{2} \right) - e^{\alpha t }\int_{0}^{t} C \left\|\w\right\|_{3}^{3} \d s\right|^{m}\right]
\ls\mathbb{E}\left[\left|\left(\left\|\w_{0}\right\|_{3}^{2} + \left\|\nabla \phi_{0}\right\|^{2} \right) \right|^{m}\right].
\end{align}
 We estimate
\begin{align}
 e^{\alpha t } \int_{0}^{t}\left\|\w\right\|_{3}^{3}\d s
\ls e^{\alpha t } t \sup\limits_{s\in[0,t]}\left\|\w\right\|_{3}^{3}\ls e^{\frac{3\alpha t }{2}}\sup\limits_{s\in[0,t]}\left\|\w\right\|_{3}^{3} .\notag
\end{align}
Therefore, we obtain the asymptotic decay estimates
\begin{align}
\mathbb{E} \left[\left|\sup\limits_{s\in[0,t]}\left(\left\|\w\right\|_{3}^{2} + \left\|\nabla \phi\right\|^{2} \right)\right|^{m}\right] \ls e^{-\alpha m t }\mathbb{E}\left[\left|C\left(\left(\left\|\w_{0}\right\|_{3}^{2} + \left\|\nabla \phi_{0}\right\|^{2} \right)\right) \right|^{m}\right],
\end{align}
on account that $\left\|\w_{0}\right\|_{3}^{2} + \left\|\nabla \phi_{0}\right\|^{2} $ is sufficiently small, where $m\geqslant 2$.

\section{Invariant measures}\label{section invariant measure}

The law generated by the initial data $\z_{0}:=\left(\rho_{0}, \u_{0}, \Phi_{0}\right)$ in probability space $\left(\Omega, \mathcal{F}, \mathbb{P}\right)$ is denoted by $\mathcal{L}\left(\z_{0}\right)$. We denote $\mathscr{H}:=H^{3}\left(U\right)\times H^{3}\left(U\right)\times H^{5}\left(U\right)$. With the initial data $\z_{0}:=\left(\rho_{0}, \u_{0}, \Phi_{0}\right) \in \mathscr{H}$ and the assumptions of \eqref{2 Condition for stochastic force} and \eqref{initial conditions assumption}, SEP system \eqref{3-D Euler Poisson} admits a unique strong solution
\begin{align}
\z(t,x,\omega):=\left(\rho, \u, \Phi\right) \in \mathscr{H}.
\end{align}
Let $\mathcal{S}_t$ be the transition semigroup \cite{Da-Prato-Zabczyk2014}:
\begin{align}\label{def trans semigroup}
\mathcal{S}_t \psi(\z_{0})=\mathbb{E}[\psi\left(\z((t, \z_{0}))\right)],  \quad t \geqslant 0,
\end{align}
where $\psi$ is the bounded function on $\mathscr{H}$, i.e., $\psi \in C_b(\mathscr{H})$.
 $\mathcal{S}(t, \z_{0}, \Gamma)$ is the transition function:
\begin{align}
\qquad \mathcal{S}(t, \z_{0}, \Gamma) := \mathcal{S}_t (\z_{0}, \Gamma)=\mathcal{S}_t \mathds{1}_{\Gamma}(\z_{0})=\mathcal{L}\left( \z(t, \z_{0})\right)(\Gamma), ~ \z_{0} \in \mathscr{H}, ~ \Gamma \in \mathscr{B}(\mathscr{H}), ~t \geqslant 0.
\end{align}

For $\v_{0}:=\left(\rho_{0}-\bar{\rho}, \u_{0}, \Phi_{0}-\bar{\Phi}\right)$ in probability space $\left(\Omega, \mathcal{F}, \mathbb{P}\right)$, the perturbed system \eqref{deformed sto Euler-Poisson} admits a unique strong solution
\begin{align}
\v(t,x,\omega):=\left(\rho-\bar{\rho}, \u, \Phi-\bar{\Phi}\right) \in \mathscr{H}.
\end{align}
 $\tilde{\mathcal{S}}_t$ is the transition semigroup:
\begin{align}\label{def trans semigroup}
\tilde{\mathcal{S}}_t \psi(\v_{0})=\mathbb{E}[\psi\left(\v((t, \v_{0}))\right)],  \quad t \geqslant 0,
\end{align}
where $\psi$ is the bounded function on $\mathscr{H}$, i.e., $\psi \in C_b(\mathscr{H})$. The transition function for the perturbed system \eqref{deformed sto Euler-Poisson} is denoted by $\tilde{\mathcal{S}}(t, \z_{0}, \Gamma)$.


We give the definition of stationary solution for \eqref{3-D Euler Poisson}.
\begin{definition}
 A strong solution $\left(\rho; \u; \Phi\right)$ to system \eqref{3-D Euler Poisson} under the initial boundary conditions \eqref{insulated boundary condition}-\eqref{initial conditions} is called stationary, provided that the
transition function $\left(\mathcal{S}_{\tau} \rho, \mathcal{S}_{\tau} \u, \mathcal{S}_{\tau} \Phi\right)$ on
$
C\left([0,T]; H^{3}\left(U\right)\right)\times  C\left([0,T]; H^{3}\left(U\right)\right)\times C\left([0,T]; H^{5}\left(U\right)\right)
$
is independent of $\tau \geqslant 0$.
\end{definition}

Let $\mathscr{M}\left(\mathscr{H}\right)$ be the space of all bounded measures on $\left(\mathscr{H}, \mathscr{B}\left(\mathscr{H}\right)\right)$. For any $\psi \in C_b\left(\mathscr{H}\right)$ and any $\mu \in \mathscr{M}\left(\mathscr{H}\right)$, we set
\begin{align}
\langle \psi, \mu\rangle_{\mathscr{H}}=\int_{\mathscr{H}} \psi(x) \mu(\d x).
\end{align}
For $t \geqslant 0$, $\mu \in \mathscr{M}\left(\mathscr{H}\right)$, $\mathcal{S}_t^*$ acts on $\mathscr{M}\left(\mathscr{H}\right)$ by
\begin{align}
\mathcal{S}_t^* \mu(\Gamma)=\int_{\mathscr{H}} \mathcal{S}(t, x, \Gamma) \mu(\d x), \quad \Gamma \in \mathscr{B}\left(\mathscr{H}\right) .
\end{align}
Moreover, there holds
\begin{align}
\left\langle\psi, \mathcal{S}_t^* \mu\right\rangle_{\mathscr{H}}=\left\langle \mathcal{S}_t \psi, \mu\right\rangle_{\mathscr{H}}, \quad \forall ~\psi \in C_b(\mathscr{H}), \quad\mu \in \mathscr{M}\left(\mathscr{H}\right).
\end{align}
Particularly, for the perturbed system \eqref{deformed sto Euler-Poisson} and $\v_{0}:=\left(\rho_{0}-\bar{\rho}, \u_{0}, \Phi_{0}-\bar{\Phi}\right)$ in probability space $\left(\Omega, \mathcal{F}, \mathbb{P}\right)$, there holds $\mathcal{S}_t^* \mathscr{L}(\v_{0}) =\mathscr{L}(\v(t, \v_{0}))$.
%
In other words,
\begin{align}
\left(\mathcal{S}_{t}\psi\right) \mathcal{L}\left(\v_{0}\right)=\mathbb{E}\left[\psi\left(\v(t)\right)\right],
\end{align}
 where $\psi\in C_{b}(\mathscr{H})$.
\begin{definition}
A measure $\mu$ in $\mathscr{M}(\mathscr{H})$ is said to be an invariant (stationary) measure if
\begin{align}
P_t^* \mu=\mu, \quad \forall ~t>0 .
\end{align}
\end{definition}


The Dirac measure centered at the steady state $\left(\bar{\rho}, 0, \bar{\Phi}\right)$ is the invariant measure for the \eqref{steady state for nd insulating}, since it keeps unchange after the action of the transition semigroup for \eqref{steady state for nd insulating}. 

For $\z_{0} \in \mathscr{H}$ and $T>0$, the formula
\begin{align}
\frac{1}{T} \int_0^T \mathcal{S}_t(\z_{0}, \Gamma) \d t=R_T(\z_{0}, \Gamma), \quad \Gamma \in  \mathscr{B}\left(\mathscr{H}\right),
\end{align}
defines a probability measure. For any $\nu \in \mathscr{M}(H), R_T^* \nu$ is defined as follows:
\begin{align}
R_T^* \nu(\Gamma)=\int_{\mathscr{H}} R_T(x, \Gamma) \nu(\d x), \quad \Gamma \in \mathscr{B}\left(\mathscr{H}\right).
\end{align}
For any $\psi \in C_b(\mathscr{H})$, there holds
\begin{align}
\left\langle R_T^* \nu, \psi\right\rangle_{\mathscr{H}}=\frac{1}{T} \int_0^T\left\langle \mathcal{S}_t^* \nu, \psi\right\rangle_{\mathscr{H}} \d t .
\end{align}

 $\mathcal{S}_t$, is a Feller semigroup provided that, for arbitrary $\psi \in C_b\left(\mathscr{H}\right)$, the function
\begin{align}
[0,+\infty) \times \mathscr{H}, \quad(t, x) \mapsto \mathcal{S}_t \psi(x)
\end{align}
is continuous. Since the solution is continuous and unique, we do not need the Markov selection as in \cite{Flandoli-Romito2008,Hofmanova-Zhu-Zhu2022}.

The method of constructing an invariant measure described in the following theorem is due to Krylov-Bogoliubov \cite{KRYLOV-BOGOLIUBOV1937}.
\begin{theorem}
 If for some $\nu \in \mathscr{M}\left(\mathscr{H}\right)$ and some sequence $T_n \uparrow+\infty, R_{T_n}^* \nu \rightarrow \mu$ weakly as $n \rightarrow \infty$, then $\mu$ is an invariant measure for Feller semigroup $\mathcal{S}_t, t \geqslant 0$.
\end{theorem}
 The following lemma is obtained similarly to \cite{Breit-Feireisl-Hofmanova-Maslowski2019}, and we provide a proof for the convenience of the readers. $\v_{t}^{\v_{0}}$ represents the stochastic process initiated from $\v_{0}$ for the sake of expediency in exposition.
\begin{lemma}
The SEP \eqref{deformed sto Euler-Poisson} defines a Feller-Markov process, i.e., $\tilde{\mathcal{S}}_{t}: C_{b}(\mathscr{H})\ra C_{b}(\mathscr{H})$, and
\begin{align}
\mathbb{E}\left[\left.\psi\left(\v_{t+s}^{\v_{0}}\right)\right|\mathcal{F}_{t}\right]=\left(\tilde{\mathcal{S}}_{s}\psi\right)\left(\v_{t}^{\v_{0}}\right), \quad \forall ~\v_{0}\in \mathscr{H}, \quad \psi\in C_{b}(\mathscr{H}), \quad \forall ~t,s>0,
\end{align}
\end{lemma}
\begin{proof}
From the continuity of solutions, it is easy to see the Feller property that $\mathcal{S}_{t}: C_{b}(\mathscr{H})\ra C_{b}(\mathscr{H})$ is continuous. For the Markov property, it suffices to prove
\begin{align}
\mathbb{E}\left[\psi\left(\v_{t+s}^{\v_{0}}\right) X\right]=\mathbb{E}\left[\mathcal{S}_{s}\psi\left(\v_{t}^{\v_{0}}\right)X\right],
\end{align}
where $X\in \mathcal{F}_{t}$.

Let $\mathbf{D}$ be any $\mathcal{F}_t$-measurable random variable. We denote $\mathbf{D}_{n}=\sum\limits_{i=1}^n \mathbf{D}^i \mathbf{1}_{\Omega^i}$, where $\mathbf{D}^i \in H$ are deterministic and $\left(\Omega^i\right) \subset \mathcal{F}_t$ is a collection of disjoint sets such that $\bigcup\limits_i \Omega^i=\Omega$.  $\mathbf{D}_n \rightarrow \mathbf{D}$ in $\mathscr{H}$ implies $\mathcal{S}_t \varphi\left(\mathbf{D}_n\right) \rightarrow \mathcal{S}_t \varphi(\mathbf{D})$ in $\mathscr{H}$. For every deterministic $\mathbf{D} \in \mathcal{F}_t$, the random variable $\v_{t, t+s}^{\mathbf{D}}$ depends only on the increments of the Brownian motion $W_{t+s}-W_{t}$ and hence it is independent of $\mathcal{F}_t$. Therefore, it holds
\begin{align}
\mathbb{E}\left[\psi\left(\v_{t, t+s}^{\mathbf{D}}\right) X\right]=\mathbb{E}\left[\psi\left(\v_{t, t+s}^{\mathbf{D}}\right)\right] \mathbb{E}[X], \quad \forall~\mathbf{D}\in \mathcal{F}_t.
\end{align}
Since $\v_{t, t+s}^{\mathbf{D}}$ has the same law as $\v_s^{\mathbf{D}}$ by uniqueness, we have
\begin{align}
\mathbb{E}\left[\psi\left(\v_{t, t+s}^{\mathbf{D}}\right) X\right]=\mathbb{E}\left[\psi\left(\v_s^{\mathbf{D}}\right)\right] \mathbb{E}[X]=\mathcal{S}_s \psi(\mathbf{D}) \mathbb{E}[X]=\mathbb{E}\left[\mathcal{S}_s \psi(\mathbf{D}) X\right].
\end{align}
Thus, there holds
\begin{align}
\mathbb{E}\left[\varphi\left(\v_{t, t+s}^{\mathbf{D}}\right) X\right]=\mathbb{E}\left[\left(\mathcal{S}_s \varphi\right)(\mathbf{D}) X\right]
\end{align}
 for every $\mathbf{D}$. By uniqueness, we have
\begin{align}
\v_{t+s}^{\v_{0}}=\v_{t, t+s}^{\v_{t}}, \quad  \mathbb{P}\quad {\rm a.s.},
\end{align}
which completes the proof. \hfill $\square$
\end{proof}
We shall prove the tightness of the law
\begin{align}\label{required tightness of the law}
\left\{\frac{1}{T}\int_{0}^{T} \mathcal{L}\left(\w(t)\right) \times \mathcal{L}\left(\phi(t)\right) \d t,\quad T>0\right\},
\end{align}
so as to apply Krylov-Bogoliubov's theorem.
\begin{theorem}\label{exists of invariant meas}
There exists an invariant measure for the system \eqref{deformed sto Euler-Poisson}.
\end{theorem}
\begin{proof}
From the energy estimates of global existence, we know that
\begin{align}\label{energy in proof of inv meas}
\mathbb{E}\left[\left(\sup\limits_{t\in[0,T]}\left\|\w(t)\right\|_{3}^{2}\right)^{m}\right]\ls \mathbb{E}\left[\left(C\left(\left\|\w_{0}\right\|_{3}^{2} + \left\|\nabla \phi_{0}\right\|^{2} \right) \right)^{m}\right].
\end{align}
The sets
\begin{align}
B_{L}:=\left\{\left.\w(t)\in H^{3}(U)\right|\left\|\w(t)\right\|_{3}\ls L\right\},\quad L>0,
\end{align}
is compact in $C^{1}(U)$.
Consequently, there holds
\begin{align}
\frac{1}{T}\int_{0}^{T} \mathcal{L}\left(\w(t)\right)\left(B_{L}^{c}\right) \d t
=&\frac{1}{T}\int_{0}^{T}\mathbb{P}\left[\left\{\left\|\w(t)\right\|_{3}>L\right\}\right]\d t  \notag \\
\ls &\frac{1}{L^{2m}T}\int_{0}^{T}\mathbb{E}\left[\left\|\w(t)\right\|_{3}^{2m}\right]\d t \\
\ls &\frac{1}{L^{2m}}\mathbb{E}\left[\left(C \left(\left\|\w_{0}\right\|_{3}^{2} + \left\|\nabla \phi_{0}\right\|^{2} \right) \right)^{m}\right]\notag \\
 &\ra 0, \text{ as } L\ra +\infty.\notag
\end{align}
This gives the tightness of  $\frac{1}{T}\int_{0}^{T} \mathcal{L}\left(\w(t)\right)\d t$. The tightness of $\frac{1}{T}\int_{0}^{T} \mathcal{L}\left(\phi(t)\right)\d t$ is obtained similarly due to the energy estimate
\begin{align}
\mathbb{E} \left[\left| \sup\limits_{t\in[0,T]}\left\|\nabla \phi(t)\right\|^{2} \right|^{m}\right] \ls \mathbb{E}\left[\left(C\left(\left\|\w_{0}\right\|_{3}^{2} + \left\|\nabla \phi_{0}\right\|^{2} \right) \right)^{m}\right].
\end{align}
 Hence the tightness of \eqref{required tightness of the law} holds. Therefore, there exists an invariant measure by Krylov-Bogoliubov's theorem.  \hfill $\square$
\end{proof}

\begin{remark}
In the above proof, we need the constant in energy estimate \eqref{energy in proof of inv meas} is independent on $T$. That is the reason why we assume \eqref{2 Condition for stochastic force} and \eqref{initial conditions assumption}.
\end{remark}
\eqref{3-D Euler Poisson} define a Feller-Markov process as well, similarly to \eqref{deformed sto Euler-Poisson}. Since $\left(\bar{\rho}, \bar{\u}, \bar{\phi}\right)$ is smooth, by the uniqueness of solutions, $\frac{1}{T}\int_{0}^{T} \mathcal{L}\left(\rho\right)\times\mathcal{L}\left(\u\right)\times\mathcal{L}\left(\Phi\right) \d s$ is also a tight measure, which generates an invariant measure. Actually,
for compact sets
\begin{align}
B_{\rho,L}=\left\{\left.\rho\in H^{3}(U)\right|\left\|\rho\right\|_{3}\ls L\right\},\quad L>0,
\end{align}
 in $C^{1}(U)$, there holds
\begin{align}
\frac{1}{T}\int_{0}^{T} \mathcal{L}\left(\rho\right)\left(B_{L}^{c}\right) \d t
=&\frac{1}{T}\int_{0}^{T}\mathbb{P}\left[\left\{\left\|\rho\right\|_{3}>L\right\}\right]\d t  \notag \\
\ls &\frac{1}{L^{2m}T}\int_{0}^{T}\mathbb{E}\left[\left\|\rho\right\|_{3}^{2m}\right]\d t \\
\ls &\frac{1}{L^{2m}}C\left(\mathbb{E}\left[\left\|\rho_{0}\right\|_{3}^{2m}\right]
+\mathbb{E}\left[\left\|\bar{\rho}\right\|_{3}^{2m}\right]\right)\notag \\
 &\ra 0, \text{ as } L\ra +\infty. \notag
\end{align}

We also care about what the limit of $\frac{1}{T}\int_{0}^{T} \mathcal{L}\left(\rho\right)\times\mathcal{L}\left(\u\right)\times\mathcal{L}\left(\Phi\right) \d t$ is.
\begin{theorem}
The invariant measure generated by $\frac{1}{T}\int_{0}^{T} \mathcal{L}\left(\rho\right)\times\mathcal{L}\left(\u\right)\times\mathcal{L}\left(\Phi\right) \d t$, for system \eqref{3-D Euler Poisson}, is the Dirac measure of the steady state $\left(\bar{\rho}, 0, \bar{\Phi}\right)$.
That is, the limit
\begin{align}\lim\limits_{T\ra +\infty} \frac{1}{T}\int_{0}^{T} \mathcal{L}\left(\rho\right)\times\mathcal{L}\left(\u\right)\times\mathcal{L}\left(\Phi\right)\d t=\delta_{\bar{\rho}}\times\delta_{0}\times \delta_{\bar{\Phi}}
\end{align}
 holds weakly.
\end{theorem}
\begin{proof}
For any $\psi\in C_{b}\left(H^{3}\right)$, we have
\begin{align}
\lim_{T\ra +\infty}\frac{1}{T} \int_{0}^{T} \langle \mathcal{L}\left(\rho\right),\psi \rangle_{\mathscr{H}} \d t = &\lim_{T\ra +\infty} \frac{1}{T}  \int_{0}^{T} \mathbb{E}\left[ \psi\left(\rho\right) \right] \d t\\
=& \lim_{T\ra +\infty} \frac{1}{T} \int_{0}^{T} \left( \mathbb{E}\left[ \psi\left(\rho\right)-\psi\left(\bar{\rho}\right)  \right] + \mathbb{E}\left[ \psi\left(\bar{\rho}\right) \right]\right) \d t.\notag
\end{align}
We claim that $\lim\limits_{T\ra +\infty} \frac{1}{T}\int_{0}^{T}  \mathbb{E}\left[ \psi\left(\rho\right)-\psi\left(\bar{\rho}\right)  \right] \d s=0$. Actually,
 we separate $\Omega$ into
\begin{align}
\Omega_{t}=\left\{\psi\left(\rho\right)-\psi\left(\bar{\rho}\right)\ls \frac{1}{\sqrt{t}}\right\}, \quad t>0,
\end{align}
and $\Omega_{t}^{c}$.
Then there holds
\begin{align}
 \mathbb{E}\left[ \psi\left(\rho\right) - \psi\left(\bar{\rho}\right) \right]=& \int_{\Omega}\left( \psi\left(\rho\right) - \psi\left(\bar{\rho}\right)\right) \mathbb{P}\left(\d \omega\right)\notag\\
= & \int_{\Omega\cap \Omega_{t}}\left( \psi\left(\rho\right) - \psi\left(\bar{\rho}\right)\right) \mathbb{P}\left(\d \omega\right)+\int_{\Omega\cap \Omega_{t}^{c}}\left( \psi\left(\rho\right) - \psi\left(\bar{\rho}\right)\right) \mathbb{P}\left(\d \omega\right) \\
=&I_{1}+ I_{2}.\notag
\end{align}
 For $I_{1}$, it holds
\begin{align}
\lim_{T\ra +\infty} \frac{1}{T}\int_{0}^{T}  \int_{\Omega\cap \Omega_{t}}\left( \psi\left(\rho\right) - \psi\left(\bar{\rho}\right)\right) \mathbb{P}\left(\d \omega\right)\d t
\ls \lim_{T\ra +\infty} \frac{1}{T}\int_{0}^{T} \frac{1}{\sqrt{t}}\d t =0.
\end{align}
For $I_{2}$, by the weighted energy estimates and Chebyshev's inequality, there holds
\begin{align}
 &\int_{\Omega\cap \Omega_{t}^{c}}\left( \psi\left(\rho\right) - \psi\left(\bar{\rho}\right)\right) \mathbb{P}\left(\d \omega\right)
\ls\int_{\Omega\cap \Omega_{t}^{c}}\left(\left| \psi\left(\rho\right) \right|+\left| \psi\left(\bar{\rho}\right)\right| \right) \mathbb{P}\left(\d \omega\right)\notag\\
\ls & C\int_{\Omega\cap \Omega_{t}^{c}}\left(\left\|\rho\right\|_{3}+\left\|\bar{\rho}\right\|_{3}\right)\mathbb{P}\left(\d \omega\right)\ls C\mathbb{P}\left[\left\{\psi\left(\rho\right)-\psi\left(\bar{\rho}\right)> \frac{1}{\sqrt{t}}\right\}\right]\\
\ls & C\frac{\mathbb{E}\left[\left| \psi\left(\rho\right) - \psi\left(\bar{\rho}\right)\right|^{2m}\right]}{\left(\frac{1}{\sqrt{t}}\right)^{2m}}\ls Ct^{m}e^{-\gamma m t}\mathbb{E}\left[\left| \rho_{0} - \bar{\rho}\right|^{2m}\right].\notag
\end{align}
Hence, we have
\begin{align}
\lim_{T\ra +\infty} \frac{1}{T}\int_{0}^{T} \int_{\Omega\cap \Omega_{t}^{c}}\left( \psi\left(\rho\right) - \psi\left(\bar{\rho}\right)\right) \mathbb{P}\left(\d \omega\right)\d t \ls \lim_{T\ra +\infty}  C\frac{1}{T}\int_{0}^{T}Ct^{m}e^{-\gamma m t}\d t =0.
\end{align}
Therefore, there holds
\begin{align}
\lim_{T\ra +\infty}\frac{1}{T} \int_{0}^{T} \langle \mathcal{L}\left(\rho\right),\psi \rangle \d t =
=& \lim_{T\ra +\infty} \frac{1}{T} \int_{0}^{T}  \mathbb{E}\left[ \psi\left(\bar{\rho}\right) \right]\d t=\mathbb{E}\left[ \psi\left(\bar{\rho}\right) \right] =\langle \delta_{\bar{\rho}},\psi \rangle.
\end{align}
A similar calculation shows that
\begin{align}
\lim\limits_{T\ra +\infty} \int_{0}^{T}  \frac{1}{T} \mathcal{L}\left(\u\right)\d t=\delta_{0};
\end{align}
and
\begin{align}
\lim\limits_{T\ra +\infty} \int_{0}^{T}  \frac{1}{T} \mathcal{L}\left(\Phi\right)\d t=\delta_{\bar{\Phi}}.
\end{align}
This completes the proof by the tightness of a joint distributions.
\hfill $\square$
\end{proof}

\smallskip
\smallskip
\section{Appendix}\label{append}
We provide an overview of the fundamental theory concerning stochastic analysis.
Let $E$ be a separable Banach space and $\mathscr{B}(E)$ be the $\sigma$-field of its Borel subsets, respectively. Let $\left(\Omega, \mathcal{F}, \mathbb{P}\right)$ be a stochastic basis. A filtration $\mathcal{F}=\left(\mathcal{F}_{t}\right)_{t \in \mathbf{T}}$ is a family of $\sigma$-algebras on $\Omega$ indexed by $\mathbf{T}$ such that $\mathcal{F}_{s} \subseteq$ $\mathcal{F}_{t} \subseteq \mathcal{F}$, $s \leq t$, $s, t \in \mathbf{T}$. $\left(\Omega, \mathcal{F}, \mathbb{P}\right)$ is also called a filtered space.
 We first list some definitions.
\begin{enumerate}
  \item [1.] {\bf $E$-valued random variables. \cite{Da-Prato-Zabczyk2014}}
 For $(\Omega, \mathscr{F})$ and $(E, \mathscr{E})$ being two measurable spaces, a mapping $X$ from $\Omega$ into $E$ such that the set $\{\omega \in \Omega: X(\omega) \in A\}=\{X \in A\}$ belongs to $\mathscr{F}$ for arbitrary $A \in \mathscr{E}$, is called a measurable mapping or a random variable from $(\Omega, \mathscr{F})$ into $(E, \mathscr{E})$ or an $E$-valued random variable.

  \item [2.]{\bf Strongly measurable operator valued random variables. \cite{Da-Prato-Zabczyk2014}} Let $\mathcal{U}$ and $\mathcal{H}$ be two separable Hilbert spaces which can be infinite-dimensional, and denote by $L(\mathcal{U}, \mathcal{H})$ the set of all linear bounded operators from $\mathcal{U}$ into $\mathcal{H}$. A functional operator $\Psi(\cdot)$ from $\Omega$ into $L(\mathcal{U}, \mathcal{H})$ is said to be strongly measurable, if for arbitrary $X \in \mathcal{U}$ the function $\Psi(\cdot) X$ is measurable, as a mapping from $(\Omega, \mathscr{F})$ into $(\mathcal{H}, \mathscr{B}(\mathcal{H}))$. Let $\mathscr{L}$ be the smallest $\sigma$-field of subsets of $L(\mathcal{U}, \mathcal{H})$ containing all sets of the form
\begin{align}
\{\Psi \in L(\mathcal{U}, \mathcal{H}): \Psi X \in A\}, \quad X \in \mathcal{U}, ~A \in \mathscr{B}(\mathcal{H}),
\end{align}
then $\Psi: \Omega \rightarrow L(\mathcal{U}, \mathcal{H})$ is a strongly measurable mapping from $(\Omega, \mathscr{F})$ into $(L(\mathcal{U}, \mathcal{H}), \mathscr{L})$.
\item  [3.]{\bf Law of a random variable.} For an $E$-valued random variable
 $X:(\Omega, \mathcal{F}) \rightarrow (E, \mathscr{E})$, 
we denote by $\mathcal{L}[X]$ the law of $X$ on $E$, that is, $\mathcal{L}[X]$ is the probability measure on $(E, \mathscr{E})$ given by
\begin{equation}
\mathcal{L}[X](A)=\mathbb{P}[X \in A], \quad A \in \mathscr{E}.\\
\end{equation} 
%

 \item [4.]
{\bf Stochastic process. \cite{Da-Prato-Zabczyk2014}}
A stochastic process $X$ is defined as an arbitrary family $X = \{X_t\}_{t\in \mathbf{T}}$ of $E$-valued random variables $X_t$, $t \in \mathbf{T}$. $X$ is also regarded as a mapping from $\Omega$ into a Banach space like $C([0, T] ; E)$ or $L^p=L^p(0, T ; E), 1 \leq p<+\infty$, by associating $\omega \in \Omega$ with the trajectory $X(\cdot, \omega)$.
\item [5.]{\bf Cylindrical Wiener Process valued in Hilbert space.} \cite{Da-Prato-Zabczyk2014} 
 A $\mathcal{U}$-valued stochastic process $W(t), t \geqslant 0$, is called a cylindrical Wiener process if
\begin{itemize}
 \item  $W(0)=0$;
 \item $W$ has continuous trajectories;
 \item $W$ has independent increments;
 \item The distribution of $(W(t)-W(s))$ is $\mathscr{N}(0,(t-s)), \quad 0\ls s \ls t$.
\end{itemize}

  \item [6.]
{\bf Adapted stochastic process.} A stochastic process $X$ is $\mathcal{F}$-adapted if $X_{t}$ is $\mathcal{F}_{t}$-measurable for every $t \in \mathbf{T}$; 

\item [7.]{\bf Martingale.} The $E$-valued process $X$ is called integrable provided $\mathbb{E}\left[\|X_t\|\right]<+\infty$ for every $t \in \mathbf{T}$. An integrable and adapted $E$-valued process $X_t, t \in \mathbf{T}$, is a martingale if
    \begin{itemize}
      \item  $X$ is adapted;
      \item $X_{s}=\mathbb{E}\left[X_{t} \mid \mathcal{F}_{s}\right]$, for arbitrary $t, s \in \mathbf{T},~ 0\ls s \ls  t$.
    \end{itemize}

\item [8.]\label{stopping time def}{\bf Stopping time.} On $\left(\Omega, \mathcal{F}, \mathbb{P}\right)$, a random time is a measurable mapping $\tau: \Omega \rightarrow \mathbf{T} \cup \infty$. A random time is a stopping time if $\{\tau \leq t\} \in \mathcal{F}_{t}$ for every $t \in \mathbf{T}$.
For a process $X$ and a subset $V$ of the state space we define the hitting time of $X$ in $V$ as
\begin{align}
\tau_{V}(\omega)=\inf \left\{\left.t \in \mathbf{T}\right| X_{t}(\omega) \in V\right\}.
\end{align}
If $X$ is a continuous adapted process and $V$ is closed, then $\tau_{V}$ is a stopping time.

 \item [9.]
{\bf Modification.}
 A stochastic process $Y$ is called a modification or a version of $X$ if
\begin{align}
\mathbb{P}[\{\omega \in \Omega: X(t, \omega) \neq Y(t, \omega)\}]=0 \quad \text { for all } t \in \mathbf{T}.
\end{align}

\item [10.]
{\bf Progressive measurability.} In $\left(\Omega,\mathcal{F},\mathbb{P} \right)$, stochastic process $X$ is progressively measurable or simply progressively measurable, if for $\omega\in \Omega$, $(\omega, s) \mapsto X(s,\omega),~ s \ls t$ is $\mathcal{F}_{t} \otimes \mathscr{B}(\mathbf{T} \cap[0, t])$-measurable for every $t \in \mathbf{T}$.

\item [11.]
{\bf Progressive measurability of continuous functions.} Let $X(t), t \in[0, T]$, be a stochastically continuous and adapted process with values in a separable Banach space $E$. Then $X$ has a progressively measurable modification.

  \item [12.]{\bf Cross quadratic variation.} Fixing a number $T>0$, we denote by $\mathcal{M}_T^2(E)$ the space of all $E$-valued continuous, square integrable martingales $M$, such that $M(0)=0$.
If $M \in \mathcal{M}_T^2\left(\mathbb{R}^1\right)$ then there exists a unique increasing predictable process $\langle M(\cdot)\rangle$, starting from 0 , such that the process
\begin{align}
M^2(t)-\langle M(\cdot)\rangle, \quad t \in[0, T]
\end{align}
is a continuous martingale. The process $\langle M(\cdot)\rangle$ is called the quadratic variation of $M$. If $M_1, M_2 \in \mathcal{M}_T^2\left(\mathbb{R}^1\right)$ then the process
\begin{align}
\left\langle M_1(t), M_2(t)\right\rangle=\frac{1}{4}\left[\left\langle\left(M_1+M_2\right)(t)\right\rangle-\left\langle\left(M_1-M_2\right)(t)\right\rangle\right]
\end{align}
is called the cross quadratic variation of $M_1, M_2$. It is the unique, predictable process with trajectories of bounded variation, starting from 0 such that
\begin{align}
M_1(t) M_2(t)-\left\langle M_1(t), M_2(t)\right\rangle, \quad t \in[0, T]
\end{align}
is a continuous martingale.\\
For $M\in\mathcal{M}_T^2(\mathcal{H})$, where $\mathcal{H}$ is Hilbert space,
the quadratic variation is defined by
\begin{align}
\langle M(t)\rangle=\sum_{i, j=1}^{\infty}\left\langle M_i(t), M_j(t)\right \rangle  e_i \otimes e_j , \quad t \in[0, T],
\end{align}
as an integrable adapted process, where $M_i(t)$ and $M_j(t)$ are in $\mathcal{M}_T^2\left(\mathbb{R}^1\right)$. If $a \in \mathcal{H}_1, b \in \mathcal{H}_2$, then $a \otimes b$ denotes a linear operator from $\mathcal{H}_2$ into $\mathcal{H}_1$ given by the formula
\begin{align}
(a \otimes b) x=a\langle b, x\rangle_{\mathcal{H}_2}, ~x \in \mathcal{H}_2.
\end{align}
We define a cross quadratic variation for $M^1 \in \mathcal{M}_T^2\left(\mathcal{H}_1\right)$, $M^2 \in \mathcal{M}_T^2\left(\mathcal{H}_2\right)$ where $\mathcal{H}_1$ and $\mathcal{H}_2$ are two Hilbert spaces. Namely we define
\begin{align}
\left\langle M^1(t), M^2(t)\right\rangle=\sum_{i, j=1}^{\infty}\left\langle M_i^1(t), M_j^2(t)\right\rangle e_i^1 \otimes e_j^2, \quad t \in[0, T],
\end{align}
where $\left\{e_i^1\right\}$ and $\left\{e_j^2\right\}$ are complete orthonormal bases in $\mathcal{H}_1$ and $\mathcal{H}_2$ respectively.

\item [13.]{\bf Stochastic integral.} Let $W$ be the Wiener process. Let $\Psi(t), t \in [0, T ]$, be a measurable Hilbert--Schmidt operators in $L(\mathcal{U}, \mathcal{H})$, which is set in the space $\mathcal{L}_{2}$ such that
    \begin{align}
    \mathbb{E}\left[\int_{0}^{t}\|\Psi(s)\|_{\mathcal{L}_{2}}^2  \d s\right]:=\mathbb{E} \int_0^t \langle\Psi(s),\Psi^{\star}(s)\rangle_{\mathcal{H}}\d s <+\infty,
    \end{align}
where $\langle\cdot,\cdot\rangle_{\mathcal{H}}$ means the inner product in $\mathcal{H}$.
 For the stochastic integral $\int_{0}^{t}\Psi \d W$, there holds
\begin{equation}
\mathbb{E}\left[\left(\int_{0}^{t}\Psi \d W\right)^{2}\right]=\mathbb{E}\left[\int_{0}^{t}\|\Psi(s)\|_{\mathcal{L}_{2}}^2 \d s\right].
\end{equation}
Furthermore, the following properties hold
\begin{itemize}
                                             \item Linearity: $\int(a \Psi_{1}+b \Psi_{2}) \d W=a \int \Psi_{1} \d W+b \int \Psi_{2} \d W$ for constants $a$ and $b$;
                                             \item Stopping property: $\int 1_{\{\cdot \leq \tau\}} \Psi \d W=\int \Psi \d M^{\tau}=\int_{0}^{\cdot \wedge \tau} \Psi \d W$;
                                             \item It\^o-isometry: for every $t$,
\begin{equation}
\mathbb{E}\left[\left(\int_{0}^{t} \Psi \d W\right)^{2}\right]=\mathbb{E}\left[\int_{0}^{t} \|\Psi(s)\|_{\mathcal{L}_{2}}^2 \d s\right].
\end{equation}
\end{itemize}

\item [14.] {\bf Dirac measure}. Let $(E, \mathscr{B}(E))$ be a measurable space. Given $x \in E$, the Dirac measure $\delta_x$ at $x$ is the measure defined by
\begin{equation}
\delta_x(A):= \begin{cases}1, & x \in A \\ 0, & x \notin A\end{cases}
\end{equation}
for each measurable set $A \subseteq E$. In this paper, there holds
$$\delta_{\bar{\rho}}=\mathcal{L}[\bar{\rho}](A) = \mathbb{P}\left[\left\{\omega\in \Omega|\bar{\rho}(x)\in A \right\}\right]=1.$$

\item [15.]{\bf Tightness of measures.} \cite{Billingsley2013tightness}
Let $E$ be a Hausdorff space, and let $\mathscr{E}$ be a $\sigma$-algebra on $E$. Let $\mathscr{M}$ be a collection of measures defined on $\mathscr{E}$. The collection $\mathscr{M}$ is called tight if, for any $\eps>0$, there is a compact subset $K_{\eps}$ of $E$ such that, for all measures $\mu \in \mathscr{M}$,
\begin{equation}\label{tightness def 1}
|\mu|\left(E \backslash K_{\eps}\right)<\eps,
\end{equation}
where $|\mu|$ is the total variation measure of $\mu$. More specially, for probability measures $\mu$, \eqref{tightness def 1} can be written as
\begin{equation}\label{tightness def 2}
\mu\left(K_{\eps}\right)>1-\eps.
\end{equation}

\end{enumerate}
\smallskip

 We list some important theorems in stochastic analysis.

\begin{enumerate}
\item [1.]{\bf It\^o's formula.} \cite{Ito1944,Da-Prato-Zabczyk2014} Assume that $\Psi$ is an $\mathcal{L}_{2}$-valued process stochastically integrable in $[0, T], \varphi$ being a $\mathcal{H}$-valued predictable process Bochner integrable on $[0, T], \mathbb{P}$-a.s., and $X(0)$ being a $\mathscr{F}_{0}$-measurable $\mathcal{\mathcal{H}}$-valued random variable. Then the following process
\begin{align}\label{form of X}
X(t)=X(0)+\int_0^t \varphi(s) d s+\int_0^t \Psi(s) \d W(s), \quad t \in[0, T]
\end{align}
is well defined. Assume that a function $F:[0, T] \times \mathcal{H} \rightarrow \mathbb{R}^1$ and its partial derivatives $F_t, F_x, F_{x x}$, are uniformly continuous on bounded subsets of $[0, T] \times \mathcal{H}$. Under the above conditions, $\mathbb{P}$-a.s., for all $t \in[0, T]$,
\begin{align}
~\qquad  F(t, X(t))= & F(0, X(0))+\int_0^t \left\langle F_x(s, X(s)), \Psi(s) \d W(s)\right\rangle_{\mathcal{H}} \\
& +\int_0^t\left\{F_t(s, X(s))+\left\langle F_x(s, X(s)), \varphi(s)\right\rangle_{\mathcal{H}}
 +\frac{1}{2} F_{x x}(s, X(s))\|\Psi(s)\|_{\mathcal{L}_{2}}^2 \right\} \d s.\notag
\end{align}
Applying the above formula for $F=\langle x, x \rangle_{\mathcal{H}}$, we have It\^o's formula for $\langle X, X \rangle_{\mathcal{H}}$. Then by
\begin{align}
\langle X, Y \rangle_{\mathcal{H}} =\frac{\langle X+Y, X+Y \rangle_{\mathcal{H}} -\langle X-Y, X-Y \rangle_{\mathcal{H}} }{4}
\end{align}
in Hilbert space,
 the following It\^o's formula holds for $X$ and $Y$ in form of \eqref{form of X},
\begin{equation}
\begin{aligned}
\langle X, Y \rangle_{\mathcal{H}} &= \langle X_{0}, Y_{0}\rangle_{\mathcal{H}} +\int \langle X, \d Y \rangle_{\mathcal{H}}+ \int  \langle Y, \d X \rangle_{\mathcal{H}}+\int \d\left\langle~\langle X,Y \rangle,  \langle X,Y \rangle~\right\rangle_{\mathcal{H}}^{\frac{1}{2}}\\
&=\langle X_{0}, Y_{0}\rangle_{\mathcal{H}} +\int \langle X, \d Y \rangle_{\mathcal{H}}+ \int  \langle Y, \d X \rangle_{\mathcal{H}}+\left\langle~\langle X,Y \rangle,  \langle X,Y \rangle~\right\rangle_{\mathcal{H}}^{\frac{1}{2}},\\
\end{aligned}
\end{equation}
where $\langle X,Y\rangle$ means the cross quadratic  variation of $X$ and $Y$ defined above.

\item [2.]{\bf Chebyshev's inequality}. Let $Y$ be a random variable in probability space $\left(\Omega, \mathcal{F}, \mathbb{P}\right)$, $\varepsilon>0$. For every $0<r<\infty$, Chebyshev's inequality reads
\begin{equation}
\mathbb{P}[\left\{|Y| \geq \varepsilon \right\}] \leq \frac{1}{\varepsilon^r}\mathbb{E}\left[|Y|^r\right] .
\end{equation}

\item [3.] {\bf Burkholder-Davis-Gundy's inequality}. \cite{BDG-inequality,Da-Prato-Zabczyk2014} Let $M$ be a continuous local martingale in $\mathcal{H}$. Let $M^{\ast}=\max\limits_{0\ls s\ls t}|M(s)|$, for any $m \geqslant 1$. $\langle M\rangle_{T}$ denotes the quadratic variation stopped by $T$. Then there exist constants $K^{m}$ and $K_{m}$ such that
\begin{equation}
 K_{m} \mathbb{E}\left[\left(\langle M\rangle_{T}\right)^{m}\right]\ls \mathbb{E}\left[\left(M^{\ast}_{T}\right)^{2m}\right]\ls K^{m} \mathbb{E}\left[\left(\langle M\rangle_{T}\right)^{m}\right],
\end{equation}
for every stopping time $T$. For $m\geqslant  1$, $K^{m}=\left(\frac{2m}{2m-1}\right)^{\frac{2m(2m-2)}{2}}$, which is equivalent to $e^{m}$ as $m\ra \infty$.
Specifically, for every $m \geqslant 1$, and for every $t \geqslant 0$, there holds
\begin{equation}
\mathbb{E}\left[ \sup _{s \in[0, t]}\left|\int_0^t \Psi(s) \d W(s)\right|^{2m} \right]\leq  K^{m}\left(\mathbb{E}\left[\int_0^{t} \|\Psi(s)\|_{\mathcal{L}_{2}}^2\d s\right]\right)^{m}
\end{equation}

\item [4.] {\bf Stochastic Fubini theorem}. Assume that $(E, \mathscr{E})$ is a measurable space and let
$$
\Psi:(t, \omega, x) \rightarrow \Psi(t, \omega, x)
$$
be a measurable mapping from $\left(\Omega_T \times E, \mathscr{B}(\Omega_T) \times \mathscr{B}(E)\right)$ into $\left(\mathcal{L}^{2}, \mathscr{B}\left(\mathcal{L}^{2}\right)\right)$. Assume moreover that
\begin{equation}
\int_E \left[\mathbb{E} \int_0^T \langle\Psi(s),\Psi^{\star}(s)\rangle_{\mathcal{H}}\d t\right]^{\frac{1}{2}} \mu(\d x)<+\infty,
\end{equation}
then $\mathbb{P}$-a.s. there holds
\begin{equation}
\int_E\left[\int_0^T \Psi(t, x) \d W(t)\right] \mu(\d x)=\int_0^T\left[\int_E \Psi(t, x) \mu(\d x)\right] \d W(t).
\end{equation}

 \item [5.]\label{Centov thm}{\bf Kolmogorov-Centov's continuity theorem.} \cite{Karatzas1988,Da-Prato-Zabczyk2014} Let $(\Omega, \mathcal{F}, \mathbb{P})$ be a probability space and $\bar{X}$ a process on $[0, T]$ with values in a complete metric space $(E, \mathscr{E})$. Suppose that
\begin{equation}
\mathbb{E}\left[\left|\bar{X}_{t}-\bar{X}_{s}\right|^{a}\right] \leq C|t-s|^{1+b},
\end{equation}
for every $s<t \leq T$ and some strictly positive constants $a, b$ and $C$. Then $\bar{X}$ admits a continuous modification $X$, $\mathbb{P}\left[\left\{X_{t}=\bar{X}_{t}\right\}\right]=1$ for every $t$, and $X$ is locally H\"older continuous for every exponent $0<\gamma<\frac{b }{a},$ namely,
\begin{equation}
\mathbb{P}\left[\left\{\omega: \sum_{0<t-s<h(\omega), t, s \leq T} \frac{\left|X_{t}(\omega)-X_{s}(\omega)\right|}{|t-s|^{\gamma}} \leq \delta\right\}\right]=1,
\end{equation}
where $h(\omega)$ is an strictly positive random variable a.s., and the constant satisfies $\delta>0$.
\end{enumerate}
\smallskip
\smallskip

{\bf Acknowledgements. } The authors would like to express their thanks to Prof. Deng Zhang for the valuable discussions. This work was commenced
when L. Zhang visited McGill University as a joint Ph.D training student supported by China Scholarship Council (CSC). She would like to
express her gratitude to McGill University and CSC.  The research of Y. Li was supported in part
by National Natural Science Foundation of China under grants 12371221, 12161141004, and 11831011. Y. Li was also grateful to the supports by the Fundamental Research Funds for the Central Universities and Shanghai Frontiers Science Center of Modern Analysis. The research of M. Mei was supported by NSERC grant RGPIN 2022-03374 and NNSFC Grant W2431005.


\begin{thebibliography}{10}

\bibitem{Bedrossian-Liss2024}
J.~Bedrossian and K.~Liss.
\newblock Stationary measures for stochastic differential equations with
  degenerate damping.
\newblock {\em Probab. Theory Relat. Fields}, 189:101--178, 2024.

\bibitem{Billingsley2013tightness}
P.~Billingsley.
\newblock {\em Convergence of probability measures}.
\newblock John Wiley \& Sons, New York, 2013.

\bibitem{deBisschop2018StochasticEI}
P.~D. Bisschop and E.~Hendrickx.
\newblock Stochastic effects in {EUV} lithography.
\newblock In {\em Advanced Lithography}, Cham, 2018. Springer.

\bibitem{Blotekjaer}
K.~Blotekjaer.
\newblock Transport equations for electrons in two-valley semiconductors.
\newblock {\em IEEE Trans. Electron Devices}, 17(1):38--47, 1970.

\bibitem{BreitFeireislHofmanova-book2018}
D.~Breit, E.~Feireisl, and M.~Hofmanov\'a.
\newblock {\em Stochastically forced compressible fluid flows}.
\newblock Walter de Gruyter GmbH, Berlin, 2018.

\bibitem{Breit-Feireisl-Hofmanova-Maslowski2019}
D.~Breit, E.~Feireisl, M.~Hofmanov{\'a}, and B.~Maslowski.
\newblock Stationary solutions to the compressible {N}avier{-}{S}tokes system
  driven by stochastic forces.
\newblock {\em Probab. Theory. Relat. Fields.}, 174(3-4):981--1032, 2019.

\bibitem{BDG-inequality}
D.~L. Burkholder, B.~J. Davis, and R.~F. Gundy.
\newblock Berkeley symposium on mathematical statistics and probability:
  Integral inequalities for convex functions of operators on martingales.
\newblock 2:223--240, 1945-1971.

\bibitem{Cruzeiro1989}
A.~B. Cruzeiro.
\newblock Solutions et mesures invariantes pour des equations stochastiques du
  type {N}avier--{S}tokes.
\newblock {\em Expo. Math.}, 7:73--82, 1989.

\bibitem{Doeblin1938}
W.~Doeblin.
\newblock Sur deux probl\`emes de m. {K}olmogoroff concernant les cha\^ines
  d\'enombrables.
\newblock {\em Bull. Soc. Math.}, 66:210--220, 1938.

\bibitem{DMRS}
D.~Donatelli, M.~Mei, B.~Rubino, and R.~Sampalmieri.
\newblock Asymptotic behavior of solutions to the cauchy problem of
  {E}uler--{P}oisson equations.
\newblock {\em J. Differential Equations}, 255:3150--3184, 2013.

\bibitem{DZZ2023}
Z.~Dong, R.~Zhang, and T.~Zhang.
\newblock Ergodicity for stochastic conservation laws with multiplicative
  noise.
\newblock {\em Commun. Math. Phys.}, 400:1739--1789, 2023.

\bibitem{Doob1953}
J.~L. Doob.
\newblock {\em Stochastic Processes}.
\newblock John Wiley \& Sons, New York, 1953.

\bibitem{Feller1957}
W.~Feller.
\newblock {\em An introduction to probability theory and its applications}.
\newblock John Wiley \& Sons, New York, 1957.

\bibitem{Flandoli1994}
F.~Flandoli.
\newblock Dissipativity and invariant measures for stochastic {N}avier-{S}tokes
  equations.
\newblock {\em Nonlinear Differ. Equ. Appl.}, 1:403--423, 1994.

\bibitem{Flandoli1-Gatarek1995}
F.~Flandoli and D.~Gatarek.
\newblock Martingale and stationary solutions for stochastic {N}avier--{S}tokes
  equations.
\newblock {\em Probab. Theory Relat. Fields}, 102:367--391, 1995.

\bibitem{Flandoli-Luo2021}
F.~Flandoli and D.~Luo.
\newblock High mode transport noise improves vorticity blow-up control in 3{D}
  {N}avier--{S}tokes equations.
\newblock {\em Probab. Theory Relat. Fields}, 180:309--363, 2021.

\bibitem{Flandoli-Romito2008}
F.~Flandoli and M.~Romito.
\newblock Markov selections for the 3{D} stochastic {N}avier--{S}tokes
  equations.
\newblock {\em Probab. Theory Relat. Fields}, 140:407--458, 2008.

\bibitem{Gess-Souganidis2017}
B.~Gess and P.~E. Souganidis.
\newblock Long-time behavior, invariant measures, and regularizing effects for
  stochastic scalar conservation laws.
\newblock {\em Commun. Pur. Appl. Math.}, 70(8):1562--1597, 2017.

\bibitem{Glatt-Holtz-Vicol2014}
N.~E. Glatt-Holtz and V.~Vicol.
\newblock Local and global existence of smooth solutions for the stochastic
  euler equations with multiplicative noise.
\newblock {\em Ann. Probab.}, 42:80--145, 2014.

\bibitem{Goldys-Maslowski2005}
B.~Goldys and B.~Maslowski.
\newblock Exponential ergodicity for stochastic {B}urgers and 2{D}
  {N}avier--{S}tokes equations.
\newblock {\em J. Funct. Anal.}, 226(1):230--255, 2005.

\bibitem{Guo2006StabilityOS}
Y.~Guo and W.~Strauss.
\newblock Stability of semiconductor states with insulating and contact
  boundary conditions.
\newblock {\em Arch. Rational Mech. Anal.}, 179:0--30, 2006.

\bibitem{Halmos1946}
P.~Halmos.
\newblock An ergodic theorem.
\newblock {\em Proc. Natl. Acad. Sci.}, 32:156--161, 1946.

\bibitem{Halmos1947}
P.~Halmos.
\newblock Invariant measures.
\newblock {\em Ann. Math.}, 48:735--754, 1947.

\bibitem{Harris1956}
T.~E. Harris.
\newblock The existence of stationary measures for certain {M}arkov processes.
\newblock 1956.

\bibitem{Harris-Robbins1953}
T.~E. Harris and H.~E. Robbins.
\newblock Ergodic theory of markov chains admitting an infinite invariant
  measure.
\newblock {\em Proc. Nat. Acad. Sci.}, 39:860--864, 1953.

\bibitem{Hofmanova-Zhu-Zhu2024}
M.~Hofmanov'a, R.~Zhu, and X.~Zhu.
\newblock Non-unique ergodicity for deterministic and stochastic 3{D}
  {N}avier--{S}tokes and {E}uler equations.
\newblock 2022.

\bibitem{Hofmanova-Zhu-Zhu2022}
M.~Hofmanov\'a, R.~Zhu, and X.~Zhu.
\newblock On ill- and well-posedness of dissipative martingale solutions to
  stochastic 3{D} {E}uler equations.
\newblock {\em Commun. Pure Appl. Math.}, 75:2446--2510, 2022.

\bibitem{Hopf1932TheoryOM}
E.~Hopf.
\newblock Theory of measure and invariant integrals.
\newblock {\em Trans. Am. Math. Soc.}, 34:373--393, 1932.

\bibitem{Hsiao2001AsymptoticsOI}
L.~Hsiao and T.~Yang.
\newblock Asymptotics of initial boundary value problems for hydrodynamic and
  drift diffusion models for semiconductors.
\newblock {\em J. Differential Equations}, 170:472--493, 2001.

\bibitem{Huang1}
F.~Huang, M.~Mei, and Y.~Wang.
\newblock Large-time behavior of solutions to n-dimensional bipolar
  hydrodynamical model of semiconductors.
\newblock {\em SIAM J. Math. Anal.}, 43:1595--1630, 2011.

\bibitem{Huang2}
F.~Huang, M.~Mei, Y.~Wang, and T.~Yang.
\newblock Long-time behavior of solutions for bipolar hydrodynamic model of
  semiconductors with boundary effects.
\newblock {\em SIAM J. Math. Anal.}, 44:1134--1164, 2012.

\bibitem{Huang2011}
F.~Huang, M.~Mei, Y.~Wang, and H.~Yu.
\newblock Asymptotic convergence to planar stationary waves for
  multi-dimensional unipolar hydrodynamic model of semiconductors.
\newblock {\em J. Differ. Equations}, 251:1305--1331, 2011.

\bibitem{Huang2011SIAM}
F.~Huang, M.~Mei, Y.~Wang, and H.~Yu.
\newblock Asymptotic convergence to stationary waves for unipolar hydrodynamic
  model of semiconductors.
\newblock {\em SIAM J. Math. Anal.}, 43(1):411--429, 2011.

\bibitem{Ito1944}
K.~It\^o.
\newblock Stochastic integral.
\newblock {\em Proc. Imp. Acad. Tokyo}, 20:519--524, 1944.

\bibitem{Karatzas1988}
I.~Karatzas and S.~Shreve.
\newblock {\em Brownian motion and stochastic calculus}.
\newblock Springer-Verlag, New York, 1988.

\bibitem{Kawashima1984SystemsOA}
S.~Kawashima.
\newblock Systems of a hyperbolic-parabolic composite type, with applications
  to the equations of magnetohydrodynamics.
\newblock 1984.

\bibitem{Kawashima2003LargeTimeBO}
S.~Kawashima, Y.~Nikkuni, and S.~Nishibata.
\newblock Large-time behavior of solutions to hyperbolic-elliptic coupled
  systems.
\newblock {\em Arch. Rational Mech. Anal.}, 170:297--329, 2003.

\bibitem{KRYLOV-BOGOLIUBOV1937}
N.~Krylov and N.~Bogoliubov.
\newblock La th\'eorie g\'en\'erale de la mesure dans son application \'a
  l'\'etude des syst\'emes de la m\'ecanique nonlin\'eaire.
\newblock {\em Ann. Math.}, 38:65--113, 1937.

\bibitem{Li-Markowich-Mei-2002}
H.~Li, P.~Markowich, and M.~Mei.
\newblock Asymptotic behaviour of solutions of the hydrodynamic model of
  semiconductors.
\newblock {\em Proc. R. Soc. Edinb., Sect. A, Math.}, 132(2):359--378, 2002.

\bibitem{Luo1}
T.~Luo and H.~Zeng.
\newblock Global existence of smooth solutions and convergence to barenblatt
  solutions for the physical vacuum free boundary problem of compressible euler
  equations with damping.
\newblock {\em Comm. Pure Appl. Math.}, 69:1354--1396, 2016.

\bibitem{Mattingly2002}
J.~Mattingly.
\newblock Exponential convergence for the stochastically forced
  {N}avier--{S}tokes equations and other partially dissipative dynamics.
\newblock {\em Commun. Math. Phys.}, 230(3):421--462, 2002.

\bibitem{MeiWu2021Stability}
M.~Mei, X.~Wu, and Y.~Zhang.
\newblock Stability of steady-state for 3-{D} hydrodynamic model of unipolar
  semiconductor with {O}hmic contact boundary in hollow ball.
\newblock {\em J. Differential Equtions}, 277:57--113, 2021.

\bibitem{Mai}
R.~Meng, L.~Mai, and M.~Mei.
\newblock Free boundary value problem for damped euler equations and related
  models with vacuum.
\newblock {\em J. Differential Equations}, 321:349--380, 2022.

\bibitem{Suzuki-2007}
S.~Nishibata and M.~Suzuki.
\newblock Asymptotic stability of a stationary solution to a hydrodynamic model
  of semiconductors.
\newblock {\em Osaka J. Math.}, 44:639--665, 2007.

\bibitem{Nishibata2009AsymptoticSO}
S.~Nishibata and M.~Suzuki.
\newblock Asymptotic stability of a stationary solution to a thermal
  hydrodynamic model for semiconductors.
\newblock {\em Arch. Rational Mech. Anal.}, 192:187--215, 2009.

\bibitem{DaPrato-Debussche2003}
G.~Da Prato and A.~Debussche.
\newblock Ergodicity for the 3{D} stochastic {N}avier--{S}tokes equations.
\newblock {\em J. Math. Pures Appl.}, 82:877--947, 2003.

\bibitem{DaPrato-Gatarek1995}
G.~Da Prato and D.~Gatarek.
\newblock Stochastic burgers equation with correlated noise.
\newblock {\em Stochastics}, 52:29--41, 1995.

\bibitem{Da-Prato-Zabczyk2014}
G.~Da Prato and J.~Zabczyk.
\newblock {\em Stochastic Equations in Infinite Dimensions}.
\newblock Cambridge University Press, Cambridge, 2 edition, 2014.

\bibitem{Sideris-Thomases-Wang}
T.~Sideris, B.~Thomases, and D.~Wang.
\newblock Long time behavior of solutions to the 3d compressible {E}uler
  equations with damping.
\newblock {\em Comm. Partial Differ. Equ.}, 28:795--816, 2003.

\bibitem{Wang-Chen1998JDE}
D.~Wang and G.-Q. Chen.
\newblock Formation of singularities in compressible {E}uler-{P}oisson fluids
  with heat diffusion and damping relaxation.
\newblock {\em J. Differential Equations}, 144:44--65, 1998.

\bibitem{Zeng}
H.~Zeng.
\newblock Global solution to the physical vacuum problem of compressible
  {E}uler equations with damping and gravity.
\newblock {\em SIAM J. Math. Anal.}, 55:6375--6424, 2023.

\end{thebibliography}


\end{document}